\documentclass[12pt]{article}
\input epsf.tex

\usepackage{graphicx}
\usepackage{amsmath,amsthm,amsfonts,amscd,amssymb,comment,eucal,latexsym,mathrsfs}
\usepackage{stmaryrd}
\usepackage[all]{xy}

\usepackage{epsfig}

\usepackage[all]{xy}
\xyoption{poly}
\usepackage{fancyhdr}
\usepackage{wrapfig}
\usepackage{epsfig}

\usepackage{bbm}

\theoremstyle{plain}
\newtheorem{thm}{Theorem}[section]
\newtheorem{prop}[thm]{Proposition}
\newtheorem{lem}[thm]{Lemma}
\newtheorem{cor}[thm]{Corollary}

\newtheorem{claim}{Claim}
\theoremstyle{definition}
\newtheorem{defn}{Definition}
\theoremstyle{remark}
\newtheorem{remark}{Remark}

\advance \topmargin by -\headheight
\advance \topmargin by -\headsep
\evensidemargin \oddsidemargin
    \def\E{{\mathbb{E}}}         \def\N{{\mathbb{N}}}  \def\P{{\mathbb{P}}}  \def\R{{\mathbb{R}}}        \def\Z{{\mathbb{Z}}}

\def\cA{{\mathcal{A}}}  \def\cC{{\mathcal{C}}}  \def\cE{{\mathcal{E}}} \def\cF{{\mathcal{F}}}  \def\cH{{\mathcal{H}}}       \def\cO{{\mathcal{O}}}   \def\cR{{\mathcal{R}}}  \def\cT{{\mathcal{T}}}      

 \def\sB{{\mathscr{B}}}   \def\sE{{\mathscr{E}}}      \def\sK{{\mathscr{K}}}  \def\sM{{\mathscr{M}}}             

\def\tA{{\tilde{A}}}  \def\tC{{\tilde{C}}}                 \def\tT{{\tilde{T}}}  \def\tV{{\tilde{V}}}    

       \def\th{{\tilde{h}}}                \def\tx{{\tilde{x}}} \def\ty{{\tilde{y}}} 

     \def\tchi{{\widetilde{\chi}}}

                   \def\vT{{\vec{T}}}      

    \def\ve{{\vec{e}}}           \def\vp{{\vec{p}}}  \def\vtr{{\vec{r}}} \def\vs{{\vec{s}}} \def\vt{{\vec{t}}} \def\vu{{\vec{u}}}

\newcommand{\G}{\Gamma}
\newcommand{\Ga}{\Gamma}

\newcommand{\Si}{\Sigma}
\newcommand{\eps}{\epsilon}

\renewcommand\a{\alpha}

\renewcommand\d{\delta}
\newcommand\g{\gamma}

\renewcommand\l{\lambda}
\newcommand\s{\sigma}

\newcommand\z{\zeta}

\newcommand\Bin{\operatorname{Bin}}

\newcommand\Fix{{\operatorname{Fix}}}
\newcommand\Hom{{\operatorname{Hom}}}

\newcommand\Map{{\operatorname{Map}}}

\newcommand\Orb{\operatorname{Orb}}

\newcommand\Part{\operatorname{Part}}
\newcommand\past{\operatorname{Past}}
\newcommand\Prob{\operatorname{Pr}}
\newcommand\proj{\operatorname{pr}}

\newcommand\supp{\operatorname{supp}}

\newcommand\Sym{{\operatorname{Sym}}}

\newcommand\sym{\operatorname{Sym}}

\newcommand\indic[1]{\mathbbm{1}_{#1}}

\def\cc{{\curvearrowright}}
\newcommand{\resto}{\upharpoonright}

\newcommand{\pr}[1]{\left ( #1 \right)}

\newcommand{\abs}[1]{\left | #1 \right |}
\newcommand{\floor}[1]{\lfloor #1 \rfloor}

\begin{document}
\title{A topological dynamical system with two different positive sofic entropies}
\author{Dylan Airey \footnote{supported in part by NSF grant DGE-1656466}, Lewis Bowen\footnote{supported in part by NSF grant DMS-1900386}  and Yuqing Frank Lin\footnote{supported in part by NSF grant DMS-1900386}\\ University of Texas at Austin}
\maketitle

\begin{abstract}
A sofic approximation to a countable group is a sequence of partial actions on finite sets that asymptotically approximates the action of the group on itself by left-translations. A group is sofic if it admits a sofic approximation. Sofic entropy theory is a generalization of classical entropy theory in dynamics to actions by sofic groups. However, the sofic entropy of an action may depend on a choice of sofic approximation. All previously known examples showing this dependence rely on degenerate behavior. This paper exhibits an explicit example of a mixing subshift of finite type with two different positive sofic entropies. The example is inspired by statistical physics literature on 2-colorings of random hyper-graphs.
\end{abstract}

\noindent
{\bf Keywords}: sofic entropy, random hyper-graphs\\
{\bf MSC}:37A35\\

\noindent
\tableofcontents

\section{Introduction}

The topological entropy of a homeomorphism $T:X\to X$ of a compact Hausdorff space $X$ was introduced in \cite{adler-konheim-mcandrew}.  It was generalized to actions of amenable groups via F\o lner sequences in the 1970s \cite{ollagnier-book} and to certain non-amenable groups via sofic approximations more recently \cite{kerr-li-variational}. It plays a major role in the classification and structure theory of topological dynamical systems.

To explain further, suppose $\G$ is a countable group with identity $1_\G$ and $\sigma:\Ga \to \sym(V)$ is a map where $V$ is a finite set and $\sym(V)$ is the group of permutations of $V$. It is not required that $\sigma$ is a homomorphism. Let $D \Subset \G$ (the symbol $\Subset$ denotes a finite subset) and $\d>0$. Then $\s$ is called 
\begin{itemize}
\item {\bf $(D,\d)$-multiplicative} if 
$$\#\{v\in V:~ \sigma(gh)v = \sigma(g)\sigma(h)v ~ \forall g,h \in D\} > (1- \delta )|V|,$$
\item {\bf $(D,\d)$-trace preserving} if 
$$\#\{v\in V:~ \sigma(f)v\ne v~\forall f\in D\setminus\{1_\G\} \} >(1- \delta) |V|,$$
\item {\bf $(D,\d)$-sofic} if it is both $(D,\d)$-multiplicative and $(D,\d)$-trace preserving.
\end{itemize}
A {\bf sofic approximation} to $\Ga$ consists of a sequence $\Si = \{\sigma_i\}_{i\in \N}$ of maps $\sigma_i:\Ga \to \sym(V_i)$ such that for all finite $D\subset \G$, $\d>0$ and all but finitely many $i$, $\sigma_i$ is $(D,\d)$-sofic. A group is {\bf sofic} if it admits a sofic approximation.  In this paper we will usually assume $|V_i| = i$.

If $\G$ acts by homeomorphisms on a compact Hausdorff space $X$ and a sofic approximation $\Si$ to $\G$ is given then the {\bf $\Si$-entropy} of the action is a topological conjugacy invariant, denoted by $h_\Si(\G \cc X) \in \{-\infty\} \cup [0,\infty]$. It is also called {\bf sofic entropy} if $\Si$ is understood. It was first defined in \cite{kerr-li-variational} where the authors obtain a variational principle connecting it with the previously introduced notion of sofic measure entropy \cite{bowen-jams-2010}. It is monotone under embeddings and additive under direct products but not monotone under factor maps. See \cite{MR4138907} for a survey.

A curious feature of this new entropy is that it may depend on the choice of sofic approximation. This is not always the case; for example, if $\G$ is amenable then sofic entropy and classical entropy always agree. However, there are examples of actions $\G \cc X$  by non-amenable groups $\G$ with sofic approximations $\Si_1,\Si_2$ satisfying
$$h_{\Si_1}(\G \cc X) = -\infty < h_{\Si_2}(\G \cc X).$$
See \cite[Theorem 4.1]{MR4138907}. The case $h_{\Si_1}(\G \cc X) = -\infty$ is considered degenerate: it implies that there are no good models for the action with respect to the given sofic approximation. Until this paper, it was an open problem whether a mixing action could have two different {\em non-negative} values of sofic entropy. Our main result is:


\begin{thm}\label{thm:main0}
There exists a countable group $\G$, a mixing action $\G \cc X$ by homeomorphisms on a compact metrizable space $X$ and two sofic approximations $\Si_1,\Si_2$ to $\G$ such that 
$$0< h_{\Si_1}(\G \cc X) < h_{\Si_2}(\G \cc X) <\infty.$$
\end{thm}

\begin{remark}
The {\bf range of sofic entropies} for an action $\G \cc X$ is the set of all non-negative numbers of the form $h_\Si(\G \cc X)$ as $\Si$ varies over all sofic approximations to $\G$. By taking disjoint unions of copies of sofic approximations, it is possible to show the range of sofic entropies is an interval (which may be empty or a singleton). So for the example of Theorem \ref{thm:main0}, the range of sofic entropies is uncountable. 
\end{remark}

\begin{remark}
It remains an open problem whether there is a measure-preserving action $\G \cc (X,\mu)$ with two different non-negative sofic entropies. Theorem \ref{thm:main0} does not settle this problem because it is entirely possible that any invariant measure $\mu$ on $X$ with $h_{\Si_2}(\G \cc (X,\mu)) > h_{\Si_1}(\G \cc X)$ satisfies $h_{\Si_1}(\G \cc (X,\mu)) = -\infty$.
\end{remark}


In this paper we often assume $V_n = [n] := \{1,2,...,n\}$.

\subsection{Random sofic approximations}
We do not know of any explicit sofic approximations to $\G$ which are amenable to analysis. Instead, we study {\bf random sofic approximations}. For the purposes of this paper, these are sequences $\{\P_n\}_n$ of probability measures $\P_n$ on spaces of homomorphisms $\Hom(\G, \sym(n))$ such that, for any finite $D\subset \G$ and $\d>0$ there is an $\eps>0$ such that
$$\P_n(\s \textrm{ is } (D,\d)\textrm{-sofic}) > 1-n^{-\eps n}$$
for all sufficiently large $n$. Because $n^{-\eps n}$ decays super-exponentially, if $\Omega_n \subset \Hom(\G,\sym(n))$ is any sequence with an exponential lower bound of the form $\P_n(\Omega_n)>e^{-cn}$ (for some constant $c>0$) then there exists a sofic approximation $\Si=\{\s_n\}$ with $\s_n \in \Omega_n$ for all $n$. 

It is this non-constructive existence result that enables us to use random sofic approximations to prove Theorem \ref{thm:main0}. 



\subsection{Proper colorings of random hyper-graphs from a statistical physics viewpoint}

The idea for our main construction comes from studies of proper colorings of random hyper-graphs. Although these studies have very different motivations than those that inspired this paper, the examples that they provide are roughly the same as the examples used to prove Theorem \ref{thm:main0}. The relevant literature and an outline is presented next. 

A {\bf hyper-graph} is a pair $G=(V,E)$ where $E$ is a collection of subsets of $V$. Elements of $E$ are called {\bf hyper-edges} but we will call them {\bf edges} for brevity's sake. $G$ is {\bf $k$-uniform} if every edge $e\in E$ has cardinality $k$. 

A {\bf 2-coloring} of $G$ is a map $\chi:V \to \{0,1\}$. An edge $e \in E$ is {\bf monochromatic} for $\chi$ if $|\chi(e)|=1$. A coloring is {\bf proper} if it has no monochromatic edges. 


Let $H_k(n,m)$ denote a hyper-graph chosen uniformly among all ${ { n \choose k} \choose m}$ $k$-uniform hyper-graphs with $n$ vertices and $m$ edges. We will consider the number of proper 2-colorings of $H_k(n,m)$ when $k$ is large but fixed, and the ratio of edges to vertices $r:=m/n$ is bounded above and below by constants. 

This random hyper-graph model was studied in  \cite{MR2263010, coja-zdeb-hypergraph, MR3205212}. These works are motivated by the {\em satisfiability conjecture}. To explain, the {\bf lower satisfiability threshold} $r^-_{\textrm{sat}}=r^-_{\textrm{sat}}(k)$ is the supremum over all $r$ such that 
$$\lim_{n\to\infty} \Prob[H_k(n, \lceil r n \rceil )\textrm{ is properly 2-colorable}] = 1.$$
The {\bf upper satisfiability threshold} $r^+_{\textrm{sat}}=r^+_{\textrm{sat}}(k)$ is the infimum over all $r$ such that
$$\lim_{n\to\infty} \Prob[H_k(n, \lceil r n \rceil)\textrm{ is properly 2-colorable}] = 0.$$
The satisfiability conjecture posits that $r^-_{\textrm{sat}} = r^+_{\textrm{sat}}$. It is still open.

Bounds on these thresholds were first obtained in \cite{MR2263010} as follows. Let $Z(G)$ be the number of proper 2-colorings of a hyper-graph $G$. A first moment computation shows that
$$f_k(r) = \lim_{n\to\infty} n^{-1} \log \E[Z(H_k(n, \lceil rn \rceil))]$$
where  $f_k(r):=\log(2) + r \log(1-2^{1-k})$. Let $r_{\textrm{first}} = r_{\textrm{first}} (k)$ be such that $f_k(r_{\textrm{first}})=0$.  If $r>r_{\textrm{first}}$ then $f_k(r)<0$. Therefore $r^+_{\textrm{sat}} \le r_{\textrm{first}}$. 

Let $r_{\textrm{second}}$ be the supremum over numbers $r\ge 0$ such that the second moment $\E[Z(H_k(n, \lceil rn \rceil))^2]$ is equal to $\E[Z(H_k(n, \lceil rn \rceil))]^2$ up to sub-exponential factors. The Paley-Zygmund inequality gives the bound $r_{\textrm{second}} \le r^-_{\textrm{sat}}$. 

In \cite{MR2263010}, it is shown that 
\begin{eqnarray*}
r_{\textrm{first}} &=&  \frac{\log(2)}{2} 2^k - \frac{\log(2)}{2} + O(2^{-k}),\\
r_{\textrm{second}} &=& \frac{\log(2)}{2} 2^k - \frac{\log(2)+1}{2} + O(2^{-k}).
\end{eqnarray*}
So there is a constant-sized gap between the two thresholds. 

A more detailed view of the second moment is illuminating. But before explaining, we need some terminology.  Let $[n]$ be the set of natural numbers $\{1,2,..,n\}$. A coloring $\chi$ of $[n]$ is {\bf equitable} if $|\chi^{-1}(0)|=|\chi^{-1}(1)|$. We will assume from now on that $n$ is even so that equitable colorings of $[n]$ exist. Let $Z_e(G)$ be the number of equitable proper colorings of a hyper-graph $G$. A computation shows that $\E[Z(H_k(n, \lceil rn \rceil))]$ equals $\E[Z_e(H_k(n, \lceil rn \rceil))]$ up to sub-exponential factors. This enables us to work with equitable proper colorings in place of all proper colorings. This reduces the computations because there is only one equitable coloring up to the action of the symmetric group $\sym(n)$. 

A computation shows that the second moment factorizes as 
$$\E[Z_e(H_k(n, m))^2] = \E[Z_e(H_k(n, m))]\E[Z_e(H_k(n,m)) | \chi \textrm{ is proper}]$$
where $\chi:[n] \to \{0,1\}$ is any equitable 2-coloring. Let $H^\chi_k(n, m)$ be the random hyper-graph chosen by conditioning $H_k(n,m)$ on the event that $\chi$ is a proper 2-coloring. This is called the {\bf planted model} and $\chi$ is the {\bf planted coloring}. So computing the second moment of $Z_e(H_k(n, m))$ reduces to computing the first moment of $Z_e(H^\chi_k(n,m))$. 

The {\bf normalized Hamming distance} between colorings $\chi,\chi': [n] \to \{0,1\}$ is
$$d_n(\chi,\chi') = n^{-1} \#\{v\in [n]:~\chi(v) \ne \chi'(v)\}.$$
Let $Z^\chi(\d)$ be the number of equitable proper colorings $\chi'$ with $d_n(\chi,\chi')=\d$. Then
$$Z_e(H^\chi_k(n,m)) = \sum_\d Z^\chi(\d).$$
In  \cite{MR2263010}, it is shown that $\E[Z^\chi(\d)| \chi \textrm{ is proper}]$ is equal to $\exp(n \psi(\d))$ (up to sub-exponential factors) where $\psi$ is an explicit function.

Note that $\psi(\d)=\psi(1-\d)$ (since if $\chi'$ is a proper equitable coloring then so is $1-\chi'$ and $d_n(\chi,1-\chi')=1-d_n(\chi,\chi')$). A computation shows $\psi(1/2)=f_k(r)$. If  $r<r_{\textrm{second}}$ then $\psi(\d)$ is uniquely maximized at $\d=1/2$. However, if $r>r_{\textrm{second}}$ then the maximum of $\psi$ is attained in the interval $\d \in [0,2^{-k/2}]$. In fact, $\psi(\d)$ is negative for $\d \in [2^{-k/2}, 1/2 - 2^{-k/2}]$. So with high probability, there are no proper equitable colorings $\chi'$ with $d_n(\chi,\chi') \in [2^{-k/2}, 1/2 - 2^{-k/2}]$. This motivates defining the {\bf local cluster}, denoted $\cC(\chi)$, to be the set of all proper equitable 2-colorings $\chi'$ with $d_n(\chi,\chi') \le 2^{-k/2}$. 

The papers \cite{coja-zdeb-hypergraph, MR3205212} obtain a stronger lower bound on the lower satisfiability threshold using an argument they call the {\em enhanced second moment method}. To explain, we need some terminology. We say a proper equitable coloring $\chi$ is {\bf good} if the size of the local cluster $|\cC(\chi)|$ is bounded by $\E[Z_e(H_k(n,m))]$. One of the main results of \cite{coja-zdeb-hypergraph, MR3205212} is that $\Pr[ \chi \textrm{ is good} | \chi \textrm{ is proper}]$ tends to 1 as $n\to\infty$ with $m  = rn + O(1)$ and $r < r_{\textrm{second}} + \frac{1-\log(2)}{2} + o_k(1)$. An application of the Paley-Zygmund inequality to the number of good colorings yields the improved lower bound
 $$r_{\textrm{second}} + \frac{1-\log(2)}{2} + o_k(1) \le r^-_{\textrm{sat}}.$$


 
 The argument showing  $\Pr[ \chi \textrm{ is good} | \chi \textrm{ is proper}]\to 1$ is combinatorial. It is shown that (with high probability) there is a set $R \subset [n]$ with cardinality $|R| \approx (1-2^{-k})n$ which is rigid in the following sense:  if $\chi':[n] \to \{0,1\}$ is any proper equitable 2-coloring then either: the restriction of $\chi'$ to $R$ is the same as the restriction of $\chi$ to $R$ or $d_n(\chi',\chi)$ is at least $cn/k^t$ for some constants $c,t>0$. This rigid set is constructed explicitly in terms of local combinatorial data of the coloring $\chi$ on $H^\chi_k(n,m)$.
 
 In summary, these papers study two random models $H_k(n,m)$ and $H^\chi_k(n,m)$. When $r=m/n$ is in the interval $(r_{\textrm{second}}, r_{\textrm{second}} + \frac{1-\log(2)}{2})$, the typical number of proper colorings of $H_k(n, m)$ grows exponentially in $n$ but is smaller (by an exponential factor) than the expected number of proper colorings of $H^\chi_k(n,m)$. It is these facts that we will generalize, by replacing $H_k(n,m), H^\chi(n,m)$ with random sofic approximations to a group $\G$ so that the exponential growth rate of the number of proper colorings roughly corresponds with sofic entropy. 
 
Although the models that we study in this paper are similar to the models in  \cite{MR2263010, coja-zdeb-hypergraph, MR3205212}, they are different enough that we develop all results from scratch. Moreover, although the strategies we employ are roughly same, the proof details differ substantially. The reader need not be familiar with these papers to read this paper.
 
\subsection{The action}
In the rest of this introduction, we introduce the action $\G \cc X$ in Theorem \ref{thm:main0} and outline the first steps of its proof. So fix positive integers $k,d$. Let
$$\G=\langle s_1,\ldots, s_d:~ s_1^k = s_2^k =\cdots = s_d^k =1 \rangle$$
be the free product of $d$ copies of $\Z/k\Z$. 

The {\bf Cayley hyper-tree} of $\G$, denoted $G=(V,E)$, has vertex set $V=\G$. The edges are the left-cosets of the generator subgroups. That is, each edge $e\in E$ has the form $e=\{gs_i^j:~0\le j \le k-1\}$ for some $g \in \G$ and $1\le i \le d$.  

\begin{remark}
It can be shown by considering each element of $\G$ as a reduced word in the generators $s_1,...,s_d$ that $G$ is a hyper-tree in the sense that there exists a unique ``hyper-path" between any two vertices.  More precisely, for any $v,w \in V$, there exists a unique sequence of edges $e_1,..,e_l$ such that $v \in e_1$, $w \in e_l$, $|e_i \cap e_{i+1}| = 1$, $e_i \neq e_j$ for any $i \neq j$, and $v \notin e_{2}$, $w \notin e_{l-1}$.  More intuitively, there are no ``hyper-loops" in $G$. 
\end{remark}

The group $\G$ acts on $\{0,1\}^\G$ by $(gx)_f = x_{g^{-1}f}$ for $g,f \in \G, x \in \{0,1\}^\G$. Let $X\subset \{0,1\}^\G$ be the subset of proper 2-colorings. It is a closed $\G$-invariant subspace. Furthermore, $\G \cc X$ is topologically mixing: 

\begin{claim}For any nonempty open sets $A$, $B$ in $X$, there exists $N$ such that for any $g \in \G$ with $|g| > N$, $gA \cap B \neq \emptyset$.  Here $|g|$ denotes the shortest word length of representations of $g$ by generators $s_1,...,s_d$. 
\end{claim}

\begin{proof}
It suffices to show the claim for $A$,$B$ being cylinder sets.  We make a further simplification by assuming each $A$, $B$ is a cylinder set on a union of hyper-edges, and a yet further simplification that each $A$, $B$ is a cylinder set on a connected union of hyper-edges.    Informally, by shifting the ``coordinates" on which $A$ depends so that they are far enough separated from the coordinates on which $B$ depends, we can always fill in the rest of the graph to get a proper coloring.  

More precisely, suppose $A = \{x \in X: x \resto F_A = \chi_A \}$, where $\cF_A \subset E$ such that for any $e_1,e_2 \in \cF_A$ there is a finite sequence of $f_1,\cdots, f_{\ell} \in \cF_A$ such that $e_1 \cap f_1 \neq \emptyset$, $f_i \cap f_{i+1} \neq \emptyset$, and $f_\ell \cap e_2 \neq \emptyset$, $F_A =\cup_{e\in \cF_A}e$, and $\chi_A : F_A \to \{0,1\}$ and similarly $B = \{x \in X: x \resto F_B = \chi_B \}$.  $\chi_A$ and $\chi_B$ must be bichromatic on each edge in their respective domains since $A$ and $B$ are nonempty.   

Let $N = \max\{|h|: h \in F_A\} + \max\{|h|: h \in F_B\} + k$.  Then it can be shown for any $g$ with $|g| > N$, that $g^{-1}F_A \cap F_B = \emptyset$.  It follows from our earlier remark that there exists a unique hyper-path connecting $g^{-1}F_A$ to $F_B$ (otherwise there would be a hyperloop in $G$).  Thus for example one can recursively fill in a coloring on the rest of $\G$ by levels of hyperedges - first the hyperedges adjacent to $g^{-1}F_A$ and $F_B$, then the next layer of adjacent hyperedges, and so on.  At each step, most hyperedges only have one  vertex whose color is determined, so it is always possible to color another vertex of an edge to make it bichromatic.  Only along the hyper-path connecting $g^{-1}F_A$ to $F_B$ at some step there will be a hyperedge with two vertices whose colors are already determined, but since $k$ is large there is still another vertex to color to make the edge bichromatic.

\end{proof}
We will show that for certain values of $k,d$, the action $\G \cc X$ satisfies the conclusion of Theorem \ref{thm:main0}. 

\subsection{Sofic entropy of the shift action on proper colorings}

Given a homomorphism $\s:\G \to \sym(V)$, let $G_\s=(V,E_\s)$ be the hyper-graph with vertices $V$ and edges equal to the orbits of the generator subgroups. That is, a subset $e \subset V$ is an edge if and only if $e = \{\s(s^j_i)v\}_{j=0}^{k-1}$ for some $1\le i \le d$ and $v \in V$. 

Recall that a hyper-graph is {\bf $k$-uniform} if every edge has cardinality $k$. We will say that a homomorphism $\s:\G \to \sym(V)$ is {\bf uniform} if $G_\s$ is $k$-uniform.  Equivalently, this occurs if for all $1\le i \le d$, $\s(s_i)$ decomposes into a disjoint union of $k$-cycles. 

A 2-coloring $\chi:V \to \{0,1\}$ of a hyper-graph $G$ is {\bf $\eps$-proper} if the number of monochromatic edges is $\le \eps |V|$. Using the formulation of sofic entropy in \cite{MR4138907} (which was inspired by \cite{MR3542515}), we show in \S \ref{sec:entropy} that if $\Si=\{\s_n\}_{n\ge 1}$ is a sofic approximation to $\G$ by uniform homomorphisms then the $\Si$-entropy of $\G\cc X$ is:
$$h_\Si(\G \cc X):=  \inf_{\eps>0} \limsup_{i\to\infty} |V_i|^{-1} \log \#\{\eps\textrm{-proper $2$-colorings of $G_{\s_i}$}\}.$$

\subsection{Random hyper-graph models}
\begin{defn}
Let $\Hom_{\rm{unif}}(\G,\sym(n))$ denote the set of all uniform homomorphisms from $\G$ to $\sym(n)$. Let $\P^u_n$ be the uniform probability measure on $\Hom_{\rm{unif}}(\G,\sym(n))$ and let $\E^u_n$ be its expectation operator. The measure $\P^u_n$ is called the {\bf uniform model}. We will always assume $n \in k\Z$ so that $\Hom_{\rm{unif}}(\G,\sym(n))$ is non-empty. In \S \ref{sec:reduction} we show that $\{\P^u_n\}_{n\ge 1}$ is a random sofic approximation. We will use the uniform model to obtain the sofic approximation $\Si_1$ which appears in Theorem \ref{thm:main0}.
\end{defn}


Recall that if $V$ is a finite set, then a 2-coloring $\chi:V \to \{0,1\}$ is {\bf equitable} if $|\chi^{-1}(0)|=|\chi^{-1}(1)|$.  We assume from now on that $n$ is even so that equitable colorings of $[n]$ exist.

\begin{defn}
Fix an equitable coloring $\chi:[n] \to \{0,1\}$. Let $\Hom_{\chi}(\G,\sym(n))$ be the set of all uniform homomorphisms $\s:\G \to \sym(n)$ such that $\chi$ is proper as a coloring on $G_\s$. Let $\P^\chi_n$ be the uniform probability measure on $\Hom_{\chi}(\G,\sym(n))$ and let $\E^\chi_n$ be its expectation operator. The measure $\P^\chi_n$ is called the {\bf planted model} and $\chi$ is the {\bf planted coloring}. When $\chi$ is understood, we will write $\P^p_n$ and $\E^p_n$ instead of $\P^\chi_n$ and $\E^\chi_n$. In \S \ref{sec:reduction} we show that $\{\P^p_n\}_{n\ge 1}$ is a random sofic approximation. We will use the planted model to obtain the sofic approximation $\Si_2$ which appears in Theorem \ref{thm:main0}.
\end{defn}

\begin{remark}
If $\chi$ and $\chi'$ are both equitable $2$-colorings then there are natural bijections from $\Hom_{\chi}(\G,\sym(n))$ to $\Hom_{\chi'}(\G,\sym(n))$ as follows. Given a permutation $\pi \in \sym(n)$ and $\s:\G \to \sym(n)$, define $\s^\pi: \G \to \sym(n)$ by $\s^\pi(g) = \pi \s(g) \pi^{-1}$. Because $\chi$ and $\chi'$ are equitable, there exists $\pi \in \sym(n)$ such that $\chi = \chi' \circ \pi$. The map $\s \mapsto \s^\pi$ defines a bijection from $\Hom_{\chi}(\G,\sym(n))$ to $\Hom_{\chi'}(\G,\sym(n))$. Moreover $\pi$ defines an hyper-graph-isomorphism from $G_\s$ to $G_{\s^\pi}$. Therefore, any random variable on $\Hom(\G,\sym(n))$ that depends only on the hyper-graph $G_\s$ up to hyper-graph-isomorphism has the same distribution under $\P^\chi_n$ as under $\P^{\chi'}_n$. This justifies calling $\P^\chi_n$ {\em the} planted model. 
\end{remark}

\subsection{The strategy and a key lemma}

The idea behind the proof of Theorem \ref{thm:main0} is to show that for some choices of $(k,d)$, the uniform model admits an exponential number of proper 2-colorings, but it has exponentially fewer proper 2-colorings than the expected number of proper colorings of the planted model (with probability that decays at most sub-exponentially in $n$).

 To make this strategy more precise, we introduce the following notation.  Let $Z(\eps;\s)$ denote the number of $\eps$-proper 2-colorings of $G_\s$. A coloring is {\bf $\s$-proper} if it is $(0,\s)$-proper. Let $Z(\s)=Z(0;\s)$ be the number of $\s$-proper 2-colorings.

In \S \ref{sec:reduction}, the proof of Theorem \ref{thm:main0} is reduced to the Key Lemma: 
\begin{lem}[Key Lemma]\label{lem:key}
Let $f(d,k):=\log(2) + \frac{d}{k} \log(1-2^{1-k})$. Also let $r=d/k$. Then
\begin{eqnarray}
f(d,k)&=& \lim_{n\to\infty} n^{-1}\log \E_n^{u}[  Z(\s)]  = \inf_{\eps>0} \limsup_{n\to\infty} n^{-1}\log \E_n^{u}[  Z(\eps;\s)] \label{eqn:1}.
\end{eqnarray}
Moreover, for any
$$0<\eta_0 < \eta_1 < (1-\log 2)/2$$
there exists $k_0$ (depending on $\eta_0,\eta_1$) such that for all $k\ge k_0$ if 
$$r=d/k= \frac{\log(2)}{2} \cdot 2^k - (1+\log(2))/2 +\eta$$
for some $\eta \in [\eta_0,\eta_1]$ then
\begin{eqnarray}
f(d,k) &<& \liminf_{n\to\infty} n^{-1}\log\E_n^{p}[ Z(\s)] \label{eqn:2}.
\end{eqnarray}
Also,
\begin{eqnarray}
0&=& \inf_{\eps>0} \liminf_{n\to\infty} n^{-1} \log \left(\P_n^{u}\left(  \left| n^{-1}\log Z(\s) - f(d,k) \right| < \eps \right)\right). \label{eqn:3}
\end{eqnarray}
In all cases above, the limits are over $n \in 2\Z \cap k\Z$.

\end{lem}

Equations (\ref{eqn:1}) and (\ref{eqn:2}) are proven in \S \ref{sec:1}  and \S \ref{sec:2} using first and second moment arguments respectively. This part of the paper is similar to the arguments used in \cite{MR2263010}.

Given $\s: \G \to \sym(n)$ and $\chi: [n] \to \{0,1\}$, let $\cC_\s(\chi)$ be the set of all proper equitable colorings $\chi':[n]\to\{0,1\}$ with $d_n(\chi,\chi')\le 2^{-k/2}$. In section \S \ref{sec:reduction2}, second moment arguments are used to reduce  equation (\ref{eqn:3}) to the following:

\noindent {\bf Proposition \ref{prop:cluster}}. {\it Let $0<\eta_0 < (1-\log 2)/2$. Then for all sufficiently large $k$ (depending on $\eta_0$), if 
$$r := d/k= \frac{\log(2)}{2} \cdot 2^k - (1+\log(2))/2 +\eta$$
for some $\eta$ with $\eta_0\le \eta <  (1-\log 2)/2$ then with high probability in the planted model, $|\cC_\s(\chi)| \leq  \E^u_n(Z_e)$. In symbols,
$$\lim_{n\to\infty} \P^\chi_n\big(|\cC_\s(\chi)| \leq  \E^u_n(Z_e)\big)=1.$$}

In \S \ref{sec:strategy}, Proposition \ref{prop:cluster} is reduced as follows. First, certain subsets of vertices are defined through local combinatorial constraints. There are two main lemmas concerning these subsets; one of which bounds their density and the other proves they are `rigid'. Proposition \ref{prop:cluster} is proven in \S \ref{sec:strategy} assuming these lemmas. 

The density lemma is proven in \S \ref{sec:Markov} using a natural Markov model on the space of proper colorings that is the local-on-average limit of the planted model. Rigidity is proven in \S \ref{sec:rigidity} using an expansivity argument similar to the way random regular graphs are proven to be good expanders. This completes the last step of the proof of Theorem \ref{thm:main0}.

{\bf Acknowledgements}. L.B. would like to thank Tim Austin and Allan Sly for helpful conversations. We would like to thank the anonymous reviewer for many corrections which have greatly improved the paper.

\section{Topological sofic entropy}\label{sec:entropy}

This section defines topological sofic entropy for subshifts using the formulation from \cite{MR3542515}. The main result is:

\begin{lem}\label{reduce:epsilon}
For any sofic approximation $\Si=\{\s_n\}$ with $\s_n \in \Hom_{\rm{unif}}(\G,\sym(n))$, 
$$h_\Si(\G \cc X) = \inf_{\eps>0} \limsup_{n\to\infty} n^{-1}\log Z(\eps;\s_n).$$
\end{lem}

Let $\G$ denote a countable group, $\cA$ a finite set (called the {\bf alphabet}). Let $T = (T^g)_{g\in \G}$ be the shift action on $\cA^\G$ defined by $T^gx(f)= x(g^{-1}f)$ for $x \in A^\G$.   Let $X \subset \cA^\G$ be a closed $\G$-invariant subspace. We denote the restriction of the action to $X$ by $\G \cc X$. Also let $\Si=\{\s_i:\G \to \Sym(V_i)\}_{i \in \N}$ be a sofic approximation to $\G$. 

Given $\s:\G \to \Sym(V)$, $v \in V$ and $x:V \to \cA$ the {\bf pullback name of $x$ at $v$} is defined by
$$\Pi^\s_v(x) \in \cA^\G,\quad \Pi^\s_v(x)(g) = x_{\s(g^{-1})v} \quad \forall g \in \G.$$
For the sake of building some intuition, note that when $\s$ is a homomorphism, the map $v \mapsto \Pi^\s_{v}(x)$ is $\G$-equivariant (in the sense that $\Pi^\s_{\s(g)v}(x) = g \Pi^\s_{v}(x)$). In particular $\Pi^\s_{v}(x) \in \cA^\G$ is periodic. In general, we think of $\Pi^\s_v(x)$ as an approximate periodic point.


Given an open set $\cO \subset \cA^\G$ containing $X$ and an $\eps>0$, a map $x:V \to \cA$ is called an {\bf $(\cO,\eps,\s)$-microstate} if
$$\#\{ v \in V:~ \Pi^\s_v(x) \in \cO\} \ge (1-\eps)|V|.$$
Let $\Omega(\cO,\eps,\s) \subset \cA^V$ denote the set of all $(\cO,\eps,\s)$-microstates. Finally, the {\bf $\Si$-entropy} of the action is defined by
$$h_\Si(\G \cc X):= \inf_\cO \inf_{\eps>0} \limsup_{i\to\infty} |V_i|^{-1} \log \#\Omega(\cO,\eps,\s_i)$$
where the infimum is over all open neighborhoods of $X$ in $\cA^\G$. This number depends on the action $\G \cc X$ only up to topological conjugacy. It is an exercise in \cite{MR4138907} to show that this definition agrees with the definition in \cite{kerr-li-sofic-amenable}.  We include a proof in Appendix \ref{appendix:sofic-entropy} for completeness.





\begin{proof}[Proof of Lemma \ref{reduce:epsilon}]
Let $\eps >0$ be given.  Let $S(\eps;\s_n) \subset 2^{V_n}$ be the set of $(\eps,\s_n)$-proper 2-colorings. Let $\cO_0 \subset 2^\G$ be the set of all  2-colorings $\chi:\G \to \{0,1\}$ such that for each generator hyper-edge $e \subset \G$, $\chi(e) = \{0,1\}$. A {\bf generator hyper-edge} is a subgroup of the form $\{s_i^j:~0\le j <k\}$ for some $i$. Note $\cO_0$ is an open superset of $X$.

 We claim that $\Omega(\cO_0, k\eps/d, \s_n) \subset S(\eps;\s_n)$.  To see this, let $\chi \in \Omega(\cO_0, k\eps/d, \s_n)$. Then $\Pi^{\s_n}_v(\chi) \in \cO_0$ if and only if all hyper-edges of $G_\s$ containing $v$ are bi-chromatic (with respect to $\chi$). So if $\Pi_v^{\s_n} \notin \cO_0$, then $v$ is contained in up to $d$ monochromatic hyperedges.  On the other hand, each monochromatic hyperedge contains exactly $k$ vertices whose pullback name is not in $\cO_0$. It follows that  $\chi \in S(\eps; \s_n)$. This implies $h_\Si(\G \cc X) \le \inf_{\eps>0} \limsup_{n\to\infty} n^{-1}\log Z(\eps;\s_n).$
 
 Given a finite subset $\cF$ of hyper-edges of the Cayley hyper-tree, let $\cO_\cF$ be the set of all $\chi \in 2^\G$ with the property that $\chi(e) = \{0,1\}$ for all $e \in \cF$. If $\cO'$ is any open neighborhood of $X$ in $2^\G$ then $\cO'$ contains $\cO_\cF$ for some $\cF$. To see this, suppose that there exist elements $\chi_\cF \in \cO_\cF\setminus \cO' $ for every finite $\cF$. Let $\chi$ be a cluster point of $\{\chi_\cF\}$ as $\cF$ increases to the set $E$ of all hyper-edges. Then $\chi \in X \setminus \cO'$, a contradiction. It follows that
 $$h_\Si(\G \cc X)= \inf_\cF \inf_{\eps>0} \limsup_{i\to\infty} |V_i|^{-1} \log \#\Omega(\cO_\cF,\eps,\s_i).$$

Next, fix a finite subset $\cF$ of hyper-edges of the Cayley hyper-tree. We claim that $S\left(\frac{\eps}{k|\cF|}; \s_n\right) \subset \Omega(\cO_\cF,\eps,\s_n)$.  To see this, let $\chi \in S\left(\frac{\eps}{k|\cF|}; \s_n\right)$ and $B(\chi,\s_n) \subset V_n$ be the set of vertices contained in a monochromatic edge of $\chi$. Now for $v \in V_n$, $\Pi_v^{\s_n}(\chi) \notin \cO_\cF$ if and only if $\Pi_v^{\s_n}(\chi)$ is monochromatic on some edge in $\cF$. This occurs if and only if there is an element $f \in \G$ in the union of $\cF$ such that  $\s_n(f^{-1})v \in B(\chi,\s_n)$.  There are at most $k |\cF||B(\chi,\s_n)|$ such vertices.  But $|B(\chi,\s_n)| \leq (\frac{\eps}{|\cF|})n$, so there are at most $k \eps n$ such vertices.  It follows that $\chi \in \Omega(\cO_\cF, \eps, \s_n)$. Therefore,
$$\inf_{\eps>0} \limsup_{n\to\infty} n^{-1}\log Z(\eps;\s_n) \le \inf_\cF \inf_{\eps>0} \limsup_{i\to\infty} |V_i|^{-1} \log \#\Omega(\cO_\cF,\eps,\s_i)= h_\Si(\G \cc X).$$


    \end{proof}

\section{Reduction to the key lemma}\label{sec:reduction}

The purpose of this section is to show how Lemma \ref{lem:key} implies Theorem \ref{thm:main0}. This requires replacing the (random) uniform and planted models with (deterministic) sofic approximations. The next lemma facilitates this replacement.



 

\begin{lem}\label{lem:super-exponential}
Let $D \subset \G$ be finite and $\d>0$. Then there are constants $\eps,N_0>0$ such that for all $n > N_0$ with $n \in 2\Z \cap k\Z$,
$$\P^u_n\{\s: ~\s ~\rm{is}~\rm{not}~ (D,\d)\rm{-sofic} \} \le n^{-\eps n},$$
$$\P^p_n\{\s: ~\s ~\rm{is}~\rm{not}~ (D,\d)\rm{-sofic} \} \le n^{-\eps n}.$$
\end{lem}

\begin{proof}
The proof given here is for the uniform model. The planted model is similar.

The proof begins with a series of four reductions. By taking a union bound, it suffices to prove the special case in which $D=\{w\}$ for $w \in \G$ nontrivial. (This is the first reduction).

Let $w = s_{i_l}^{r_l}\cdots s_{i_1}^{r_1}$ be the {\bf reduced form} of $w$. This means that $i_j \in \{1,\ldots, d\}$, $i_j \ne i_{j+1}$ for all $j$ with indices mod $l$ and $1\le r_j <k$ for all $j$. Let $|w|=r_1+\cdots + r_l$ be the {\bf length} of $w$. 

For any $g \in \G$, the fixed point sets of $\s(gwg^{-1})$ and $\s(w)$ have the same size. So after conjugating if necessary, we may assume that either $l=1$ or $i_1 \ne i_l$. 

For $1\le j\le l$, the {\bf $j$-th beginning subword} of $w$ is the element $w_j = s_{i_j}^{r_j}\cdots s_{i_1}^{r_1}$. Given a vertex $v \in V_n$ and $\s \in \Hom_{\rm{unif}}(\G,\sym(n))$, let $p(v,\s)=(e_1,\ldots, e_l)$ be the path defined by: for each $j$, $e_j$ is the unique hyper-edge of $G_\s$ labeled $i_j$ containing $\s(w_j)v$.  A vertex $v \in V_n$ {\bf represents a $(\s,w)$-simple cycle} if $\s(w)v=v$ and for every $1\le a<b \le l$, either 
\begin{itemize}
\item $e_a \cap e_b =\emptyset$,
\item $b=a+1$ and $|e_a \cap e_b| = 1$,
\item or $(a,b)=(1,l)$ and $|e_a \cap e_b|=1$.  
\end{itemize}

We say that $v$ {\bf represents a $(\s,w)$-simple degenerate cycle} if $\s(w)v = v$ and $l = 2$ and $|e_1 \cap e_2| \ge 2$.

If $\s(w)v = v$ then either 
\begin{itemize}
\item $v$ represents a $(\s,w)$-simple cycle,
\item there exists nontrivial $w' \in \G$ with $|w'| \le |w|+k$ such that some vertex $v_0 \in \cup_j e_j$ represents a $(\s,w')$-simple cycle,
\item or there exists nontrivial $w' \in \G$ with $|w'| \le |w|+k$ such that some vertex $v_0 \in \cup_j e_j$ represents a $(\s,w')$-simple degenerate cycle.
\end{itemize}
  So it suffices to prove there are constants $\eps,N_0>0$ such that for all $n > N_0$,
$$\P^u_n\big\{\s: ~\#\{v \in [n]:~v \textrm{ represents a $(\s,w)$-simple cycle} \}\ge \d n \big\}  \le n^{-\eps n}.$$ and $$\P^u_n\big\{\s: ~\#\{v \in [n]:~v \textrm{ represents a $(\s,w)$-simple degenerate cycle} \}\ge \d n \big\}  \le n^{-\eps n}.$$
(This is the second reduction).

Two vertices $v,v' \in V_n$ {\bf represent vertex-disjoint $(\s,w)$-cycles} if $p(v,\s)=(e_1,\ldots, e_l), p(v',\s)=(e'_1,\ldots, e'_l)$ and $e_i \cap e'_j =\emptyset$ for all $i,j$. 

Let $G_n(\d,w)$ be the set of all $\s \in \Hom_{\rm{unif}}(\G,\sym(n))$ such that there exists a subset $S \subset [n]$ satisfying
\begin{enumerate}
\item $|S| \ge \d n$,
\item every $v \in S$ represents a $(\s,w)$-simple cycle,
\item the cycles $p(v,\s)$ for $v \in S$ are pairwise vertex-disjoint.
\end{enumerate}

If $v$ represents a simple $(\s,w)$-cycle then there are at most $(kl)^2$ vertices $v'$ such that $v'$ also represents a simple $(\s,w)$-cycle but the two cycles are not vertex-disjoint.  Since this bound does not depend on $n$, it suffices to prove there exist $\eps>0$ and $N_0$ such that 
$$\P^u_n(G_n(\d,w)) \le n^{-\eps n}$$
for all $n \ge N_0$.  (This is the third reduction.  The argument is similar for simple degenerate cycles).

Let $m=\lceil \d n \rceil$ and $v_1,\ldots, v_m$ be distinct vertices in $[n]=V_n$. For $1\le i \le m$, let $F_i$ be the set of all $\s \in \Hom_{\rm{unif}}(\G,\sym(n))$ such that for all $1\le j \le i$
\begin{enumerate}
\item $v_j$ represents a $(\s,w)$-simple cycle,
\item the cycles $p(v_1,\s),\ldots, p(v_i,\s)$ are pairwise vertex-disjoint.
\end{enumerate}

By summing over all subsets of size $m$, we obtain
$$\P^u_n(G_n(\d,w)) \le {n \choose m } \P^u_n(F_m).$$
Since ${n \choose m} \approx e^{H(\d,1-\d)n}$ grows at most exponentially, it suffices to show  there exist $\eps>0$ and $N_0$ such that $\P^u_n(F_m) \le n^{-\eps n}$ for all $n \ge N_0$. (This is the fourth reduction.  The argument is similar for simple degenerate cycles).

Set $F_0= \Hom_{\rm{unif}}(\G,\sym(n))$. By the chain rule
$$\P^u_n(F_m) = \prod_{i=0}^{m-1} \P^u_n(F_{i+1} | F_i ).$$

In order to estimate $\P^u_n(F_{i+1} | F_i )$, $F_i$ can be expressed a disjoint union over the cycles involved in its definition. To be precise, define an equivalence relation $\cR_i$ on $F_i$ by: $\s,\s'$ are $\cR_i$-equivalent if for every $1\le j \le i$, $1\le q \le l$ and $r>0$ 
$$\s(s_{i_q}^rw_q)v_j = \s'(s_{i_q}^rw_q)v_j.$$
In other words, $\s,\s'$ are $\cR_i$-equivalent if they define the same paths according to all vertices up to $v_i$ (so $p(v_j,\s)=p(v_j,\s')$) and their restrictions to every edge in these paths agree. Of course, $F_i$ is the disjoint union of the $\cR_i$-classes.  Note that $\cR_0$ is trivial (everything is equivalent). 

In general, if $A, B_1,\ldots, B_m$ are measurable sets and the $B_i$'s are pairwise disjoint then $\P(A| \cup_i B_i)$ is a convex combination of $\P(A | B_i)$ (for any probability measure $\P$). Therefore, $\P^u_n(F_{i+1} | F_i)$ is a convex combination of probabilities of the form $\P^u_n(F_{i+1} | B_i)$ where $B_i$ is an $\cR_i$-class. 

Now fix a $\cR_i$-class $B_i$ (for some $i$ with $0\le i < m$). Let $K$ be the set of all vertices covered by the cycles defining $B_i$. To be precise, this means $K$ is the set of all $u \in [n]=V_n$ such that there exists an edge $e$ with $u\in e$ such that $e$ is contained in a path $p(v_j,\s)$ with $1\le j \le i$ and $\s \in B_i$. Since each path covers at most $kl$ vertices, $|K| \le ikl$. 

  If $l > 1$ (the case $l=1$ is similar), fix subsets $e_1,\ldots, e_{l-1} \subset [n]$ of size $k$. Conditioned on $B_i$ and the event that the first $(l-1)$ edges of $p(v_{i+1},\s)$ are $e_1,\ldots, e_{l-1}$, the $\P^u_n$-probability that $v_{i+1}$ represents a simple $(\s,w)$-cycle vertex-disjoint from $K$ is bounded by the probability that a uniformly random $k-1$-element subset of 
$$[n] \setminus \left(\bigcup_{2\le j \le l-1} e_j \cup K\right)$$
contains $v_{i+1}$.  Since 
$$ \left|\bigcup_{2\le j \le l-1} e_j \cup K\right| \le (i+1)kl \le m kl = kl \lceil \d n \rceil,$$
this probability is bounded by $C/n$  where $C=C(w,d,k,\d)$ is a constant not depending on $n$ or the choice of $B_i$. It follows that $\P^u_n(F_{i+1}|F_i) \le C/n$ for all $0\le i \le m-1$ and therefore
$$\P^u_n(F_m) \le (C/n)^m \le (C/n)^{\d n}.$$
This implies the lemma (the argument is similar for simple degenerate cycles).

\end{proof}


\begin{proof}[Proof of Theorem \ref{thm:main0} from Lemma \ref{lem:key}]


Choose $d,k,\eta_0, \eta_1, \eta$ according to the hypotheses of Lemma \ref{lem:key} so that $\eta \in [\eta_0,\eta_1]$, $k\ge k_0$ and all of the conclusions to Lemma \ref{lem:key} hold. Construct $\G \cc X$ according to Section 1.3.  We will always take $n \in 2\Z \cap k\Z$.  We will first show the existence of a sofic approximation $\Sigma_1 = \{\s_n\}$ to $\G$ such that 
$$\inf_{\eps>0} \limsup_{n\to\infty} n^{-1}\log Z(\eps;\s_n) = f(d,k).$$ 
Then by Lemma \ref{reduce:epsilon}, we will have $h_{\Sigma_1}(\G \cc X) = f(d,k) > 0$.

Observe that Lemma \ref{lem:key}(1) implies the following: for every $\beta>0$ there exists $\eps(\beta)>0$ (with $\eps$ decreasing as $\beta$ decreases), $N$, and $\xi(\beta,\eps,N) > 0$ such that for $n > N$, $\P^u_n(n^{-1} \log (Z(\eps;\s_n))  - f(d,k)\geq \beta) < e^{-\xi n}$.  In other words, the number of $\eps$-proper colorings can only exponentially exceed $f(d,k)$ with exponentially small probability.  This is because the negation of the above would imply that $\inf_{\eps>0} \limsup_{n\to\infty} n^{-1}\log \E_n^{u}[  Z(\eps;\s)] > f(d,k)$, contradicting (1).  Let $F_{n,\beta,\eps}$ be the event that $n^{-1}\log [Z(\eps;\s_n)]  - f(d,k)\geq \beta$.


Let $H_{n,\beta}$ be the event that $|n^{-1}\log Z(\s_n) - f(d,k)|<\beta$.  Lemma \ref{lem:key}(3) implies $\P_n^{u}(H_{n,\beta})$ decays at most subexponentially in $n$. Precisely, for any $c>0$ there exists $N=N(c,\beta)$ such that $n>N$ implies  $\P_n^{u}(H_{n,\beta}) > e^{-cn}$.

Let $I_{n,\beta,\eps}$ be the event that $|n^{-1}\log Z(\eps;\s_n) - f(d,k)| < \beta$.  Notice that since $Z(\eps;\s_n) \geq Z(\s_n)$ for any $\eps$, $F_{n,\beta,\eps}^c \cap H_{n,\beta} \subset I_{n,\beta,\eps}$.

Consider a decreasing sequence $\beta_m \to 0$, and $\eps_m$ and $\xi_m$ depending on $\beta_m$ as discussed above.  Since $\P_n^{u}(H_{n,\beta})$ is decaying only subexponentially, we can choose an increasing sequence $K_m$ satisfying, for each $m$, $\P_n^{u}(H_{n,\beta_m}) > 2\sum_{i=1}^m e^{-\xi_in}$ for all $n > K_m$.  Now it follows that $n>K_m$ implies $\P^u_n(\cap_{j=1}^m F^c_{n,\beta_j,\eps_j} \cap H_{n,\beta_j}) \geq \P_n^{u}(H_{n,\beta_m})  - \sum_{i=1}^m e^{-\xi_in} \geq e^{-\xi_1 n}$. By the observation in the preceding paragraph, $\P^u_n(\cap_{j=1}^m I_{n,\beta_j,\eps_j}) \geq e^{-\xi_1 n}$.  

Given $\d>0$ and $D \subset \G$ finite, let $J_{D,\d,n}$ be the event that $\s_n$ is $(D,\d)$-sofic.  Let $\d_m \to 0$ be a decreasing sequence and $D_m \subset \G$ be an increasing sequence of finite subsets of $\G$.  

Lemma $\ref{lem:super-exponential}$ implies that $\P^u_n(J_{D,\d,n}^c)$ decays super-exponentially in $n$ for any $D$ and $\d$, so there exists an increasing sequence $N_m$ such that for $n>N_m$,  $\P^u_n(\cap_{j=1}^m I_{n,\beta_j,\eps_j} \cap J_{D_m,\d_m,n}) \geq 0.5e^{-\xi_1 n}$.  It follows that we can choose a deterministic sofic approximation sequence $\Sigma_1 = \{\s_n\}$ such that $$\inf_{\eps>0} \limsup_{n\to\infty} n^{-1}\log Z(\eps;\s_n) = f(d,k).$$

We next show the existence of $\Sigma_2$.  Equation (2) of Lemma \ref{lem:key} implies the existence of a number $f_p$ with
$$f(d,k) < f_p < \liminf_{n\to\infty} n^{-1}\log\E_n^{p}[ Z(\s)].$$
Since $Z(\s) \le 2^n$ for every $\s$, there exist constants $c,N_0>0$ such that 
\begin{eqnarray}\label{eqn:exponential}
\P^p_n\{\s:~ Z(\s) \ge \exp(n f_p)\} \ge \exp( -cn)
\end{eqnarray}
for all $n \ge N_0$. 

Now let $\d>0$ and $D \subset \G$ be finite. Then there exists $N_2$ such that if $n>N_2$ and $\s_{n}$ is chosen at random with law $\P^p_{n}$, then with positive probability,
\begin{enumerate}
\item $\s_{n}$ is $(D,\d)$-sofic,
\item $n^{-1} \log Z(\s_n)  \ge  f_p$.
\end{enumerate}
This is implied by Lemma \ref{lem:super-exponential} and equation (\ref{eqn:exponential}). So there exists a sofic approximation $\Si_2=\{\s'_n\}$ to $\G$ such that
$$ \limsup_{n\to\infty} n^{-1}\log Z(\s'_n) \ge f_p.$$
Since $Z(\s'_n) \le Z(\beta; \s'_n)$, Lemma \ref{reduce:epsilon} implies $h_{\Si_2}(\G \cc X) \ge f_p > f(d,k) = h_{\Si_1}(\G \cc X)$. 

\end{proof}

\section{The first moment}\label{sec:1}

To simplify notation, we assume throughout the paper that $n \in 2\Z \cap k\Z$ without further mention. This section proves (\ref{eqn:1}) of  Lemma \ref{lem:key}. The proof is in two parts. Part 1, in \S \ref{sec:almost}, establishes:

\begin{thm}\label{thm:almost}
$$\lim_{\eps \searrow 0} \limsup_{n\to\infty} (1/n) \log \E^u_n[Z(\eps;\s)]  = \limsup_{n\to\infty} (1/n) \log \E^u_n[Z(\s)].$$
\end{thm}

Part 2 has to do with equitable colorings, where a 2-coloring $\chi:[n] \to \{0,1\}$ is {\bf equitable} if
$$|\chi^{-1}(0)| = |\chi^{-1}(1)| = n/2.$$
Let $Z_e(\s)$ be the number of proper equitable colorings of $G_\s$. \S \ref{sec:equitable} establishes

\begin{thm}\label{thm:1moment}
$$\lim_{n\to\infty} \frac{1}{n} \log \E^u_n[ Z(\s) ] = \lim_{n\to\infty} \frac{1}{n} \log \E^u_n[ Z_e(\s) ].$$
Moreover,
$$ \frac{1}{n} \log \E^u_n[ Z_e(\s) ] =   f(d,k) +O(n^{-1}\log(n))$$
where $f(d,k)= \log(2) + \frac{d}{k} \log(1-2^{1-k}).$
\end{thm}

Combined, Theorems \ref{thm:almost} and \ref{thm:1moment} imply (\ref{eqn:1}) of Lemma \ref{lem:key}.

\begin{remark}
If $r:=(d/k)$ then the formula for $\lim_{n\to\infty} \frac{1}{n} \log \E^u_n[ Z(\s) ]$ above is the same as the formula found in \cite{MR2263010, coja-zdeb-hypergraph, MR3205212} for the exponential growth rate of the number of proper 2-colorings of $H_k(n,m)$. 
\end{remark}

\begin{remark}
When we write an error term, such as $O(n^{-1}\log(n))$, we always assume that $n \ge 2$ and the implicit constant is allowed to depend on $k$ or $d$.
\end{remark}

\subsection{Almost proper 2-colorings}\label{sec:almost}

For $0 < x \le 1$, let $\eta(x)=-x\log(x)$. Also let $\eta(0)=0$. If $\vT=(T_i)_{i\in I}$ is a collection of numbers with $0 \le T_i\le 1$, then let
$$H(\vT) := \sum_{i \in I} \eta(T_i)$$
be the {\bf Shannon entropy} of $\vT$. 
\begin{defn}
A {\bf $k$-partition} of $[n]$ is an unordered partition of $[n]$ into sets of size $k$. Of course, such a partition exists if and only if $n/k \in \N$ in which case there are 
\begin{eqnarray}\label{eqn:part3}
\frac{n!}{k!^{n/k} (n/k)!}
\end{eqnarray}
such partitions. By Stirling's formula,
\begin{eqnarray}\label{eqn:part}
\frac{1}{n} \log (\# \{\textrm{$k$-partitions} \}) = (1-1/k)(\log (n) -1) -(1/k)\log (k-1)! + O(n^{-1}\log(n)).
\end{eqnarray}
\end{defn}

\begin{defn}
The {\bf orbit-partition} of a permutation $\rho \in \sym(n)$ is the partition of $[n]$ into orbits of $\rho$. Fix a $k$-partition $\pi$. Then the number of permutations $\rho$ whose orbit partition is $\pi$ equals $(k-1)!^{n/k}$. 

Given $\s\in \Hom_{\rm{unif}}(\G,\sym(n))$, define the $d$-tuple $(\pi^\s_1,\ldots, \pi^\s_d)$ of $k$-partitions by: $\pi^\s_i$ is the orbit-partition of $\s(s_i)$. Fix a $d$-tuple of $k$-partitions $(\pi_1,\ldots, \pi_d)$. Then the number of uniform homomorphisms $\s$ such that $\pi^\s_i = \pi_i$ for all $i$ is $[(k-1)!^{n/k}]^d$. Combined with (\ref{eqn:part3}), this shows the number of uniform homomorphisms into $\sym(n)$ is
$$\left[\frac{n!(k-1)!^{n/k}}{k!^{n/k} (n/k)!} \right]^d.$$
By Stirling's formula,
\begin{eqnarray}\label{eqn:part2}
\frac{1}{n} \log \# \Hom_{\rm{unif}}(\G,\sym(n))= d(1-1/k)(\log n -1)  + O(n^{-1}\log(n)).
\end{eqnarray}

\end{defn}

\begin{defn}
Let $\pi$ be a $k$-partition, $\chi:[n] \to \{0,1\}$ a 2-coloring and $\vt = (t_j)_{j=0}^k \in [0,1]^{k+1}$ a vector with $\sum_j t_j = 1/k$. The pair $(\pi,\chi)$ has {\bf type $\vt$} if for all $j$,
$$\#\left\{e\in \pi:~ |e \cap \chi^{-1}(1)| = j\right\} = n t_j.$$
\end{defn}

\begin{lem}\label{lem:numberofpartitions}
Let $\vec{t}=(t_0,t_1,\ldots, t_k) \in [0,1]^{k+1}$ be such that $\sum_j t_j=1/k$ and $nt_j \in \Z$. Let $p=\sum_j j t_j$.  Let $\chi: [n] \to \{0,1\}$ be a map such that $|\chi^{-1}(1)|=pn$. Let $f( \vec{t})$ be the number of $k$-partitions $\pi$ of $[n]$ such that $(\pi,\chi)$ has type $\vt$. Then
$$(1/n)\log f(\vec{t}) = (1-1/k)(\log(n)-1) -H(p,1-p)+H(\vec{t})-\sum_{j=0}^k t_j \log(j!(k-j)!) + O(n^{-1}\log(n)).$$
\end{lem}

\begin{proof}
The following algorithm constructs all such partitions with no duplications:
\begin{itemize}
\item[Step 1.] Choose an unordered partition of the set $\chi^{-1}(1)$ into $t_j n$ sets of size $j$ ($j=0,\ldots, k$).
\item[Step 2.] Choose an unordered partition of the set $\chi^{-1}(0)$ into $t_j n$ sets of size $k-j$ ($j=0,\ldots, k$).
\item[Step 3.] Choose a bijection between the collection of subsets of size $j$ constructed in part 1 with the collection of subsets of size $k-j$ constructed in part 2.
\item[Step 4.] The partition consists of all sets of the form $\alpha \cup \beta$ where $\alpha \subset \chi^{-1}(1)$ is a set of size $j$ constructed in Step 1 and $\beta \subset \chi^{-1}(0)$ is a set of size $(k-j)$ constructed in Step 2 that it is paired with under Step 3.
\end{itemize}
The number of choices in Step 1 is $\frac{(pn)!}{ \prod_{j=1}^k (j)!^{t_jn}(t_jn)!}$. The number of choices in Step 2 is $\frac{((1-p)n)!}{ \prod_{j=0}^{k-1} (k-j)!^{t_jn}(t_jn)!}$. The number of choices in Step 3 is $\prod_{j=1}^{k-1} (t_jn)!$. So 
\begin{eqnarray}\label{eqn:q-thing}
f(\vec{t}) = \frac{(pn)!((1-p)n)!}{ \prod_{j=0}^{k} j!^{t_jn}(k-j)!^{t_jn} (t_jn)!}.
\end{eqnarray}
The lemma follows from this and Stirling's formula.
\end{proof}

Let $\sM$ be the set of all matrices $\vT=(T_{ij})_{1\le i \le d, 0\le j \le k}$ such that
\begin{enumerate}
\item $T_{ij} \ge 0$ for all $i,j$,
\item $\sum_{j=0}^k T_{ij}=1/k$ for all $i$,
\item there exists a number, denoted $p(\vT)$, such that $p(\vT) = \sum_{j=0}^k j T_{ij}$ for all $i$.
\end{enumerate}
Let $\sM_n$ be the set of all $\vT \in \sM$ such that $n\vT$ is integer-valued.

\begin{lem}\label{lem:diophantine}
Let $A$ be an integer-valued $q\times p$ matrix (for some $p,q \in \N$),  $b \in \R^q$ and $\sK \subset \R^p$ the set of all $x\in \R^p$ such that $Ax = b$ and $x_i \ge 0$ for all $i$. For $n\in \N$, let $\sK_n$ be the set of all $x \in \sK$ such that $nx$ is integer-valued. 

Assume $\sK$ is compact and there exists $x \in \sK$ with $x_i>0$ for all $i$. Then there is a constant $C>0$ such that if $\sK_n$ is non-empty then it is $C/n$-dense in $\sK$ in the following sense. For any $x \in \sK$ there exists $x' \in \sK_n$ such that $\|x - x'\|_\infty < C/n$. Moreover, the constant $C$ can be chosen to depend continuously on the vector $b \in \R^q$. 
\end{lem}

\begin{proof}
To begin, we will define several constants which will enable us to choose $C>0$. Because $A$ is integer-valued, its kernel, denoted $\ker(A) \subset \R^p$, is such that $\ker(A) \cap  \Z^{p}$ has rank equal to the dimension of $\ker(A)$. Therefore, $\ker(A)\cap \Z^{p}$ is cocompact in $\ker(A)$. So there is a constant $C_1>0$ such that for any $z \in \ker(A)$ there is an element $z' \in\ker(A) \cap  \Z^{p}$ with $\|z - z'\|_\infty < C_1$.

By hypothesis, there is a constant $C_2>0$ and an element $y \in \sK$ such that $y_i > C_2$ for all $i$. Because $\sK$ is compact there is another constant $C_3>0$ such that $\|x-y\|_\infty <C_3$ for all $x\in \sK$. Let $C=\frac{C_1C_3}{C_2} + C_1$. Now let $x \in \sK$ be arbitrary and suppose $\sK_n$ is non-empty. We will show there exists $x'' \in \sK_n$ such that $\|x-x''\|_\infty \le C/n$. 

Let 
$$x'=\left(1-\frac{C_1}{C_2n}\right)x + \frac{C_1}{C_2n}y.$$
Then $x'_i \ge (C_1/C_2n)y_i > C_1/n$ for all $i$. Also $\|x-x'\|_\infty = (C_1/C_2n)\|x-y\|_\infty < \frac{C_1C_3}{C_2n}$.

Because $\sK_n$ is non-empty, there exists $x_n \in \sK_n$. By linearity, $\sK$  is the intersection of the hyperplane $x_n + \ker(A)$ with the positive orthant. Thus we can write $x'=x_n+(1/n)z' + z''$ where $z' \in \ker(A) \cap \Z^{p}$ and $z'' \in \ker(A)$ satisfies   $\|z''\|_\infty < C_1/n$. Let $x''=x_n + (1/n)z'$. Note $\|x''-x'\|_\infty =\|z''\|_\infty < C_1/n$. Since $x'_{i}>C_1/n$ this implies $x''_{i}>0$ for all $i$. It is now straightforward to check that $x'' \in \sK_n$. By the triangle inequality
$$\|x - x''\|_\infty \le \|x - x'\|_\infty + \|x' - x''\|_\infty < \frac{C_1C_3}{C_2n} + C_1/n = C/n.$$
Because $x$ and $n$ are arbitrary, this implies the Lemma with $C=\frac{C_1C_3}{C_2} + C_1$. Moreover $C_1$ does not depend on the vector $b$; while $C_2,C_3$ can be chosen to depend continuously on $b$. 

\end{proof}

\begin{lem}\label{lem:first}
Given a matrix $\vT \in \sM$ define
\begin{eqnarray*}
F(\vT)&:=&H(\vT) + (1-d)H(p,1-p) - (d/k)\log k +\sum_{i=1}^d \sum_{j=0}^k T_{ij} \log {k \choose j}
\end{eqnarray*}
where $p=p(\vT)$. Then for any $\eps\ge 0$,
$$(1/n) \log \E^u_n[Z(\eps;\s)] = \sup\left\{ F(\vT):~\vT \in \sM_n \textrm{ and } \sum_{i=1}^d \sum_{j=0,k} T_{ij} \le \eps\right\} + O(n^{-1}\log(n))$$
where the constant implicit in the error term does not depend on $\eps$. 
\end{lem}

\begin{proof}
Given $\s\in \Hom_{\rm{unif}}(\G,\sym(n))$ and $1\le i \le d$, let $\pi^\s_i$ be the orbit-partition of $\s(s_i)$. For $\vT$ as above,  let $Z_\s(\vT)$ be the number of $\eps$-proper colorings $\chi:[n] \to \{0,1\}$ such that $(\pi^\s_i,\chi)$ has type $\vT_i = (T_{i,0},\ldots, T_{i,k})$. It suffices to show that
$$(1/n) \log \E^u_n[Z_\s(\vT)] = F(\vT) + O(n^{-1}\log(n))$$
for all $n\ge 2$ such that $\vT \in \sM_n$.  This is because the size of $\sM_n$ is a polynomial (depending on $k,d$) in $n$ so the supremum above determines the exponential growth rate of $\E^u_n[Z(\eps;\s)]$.


To prove this, fix a $\vT$ as above and let $n$ be such that $n \vT$ is integer-valued. Fix a coloring $\chi:[n] \to \{0,1\}$ such that $|\chi^{-1}(1)|=pn$. By symmetry,
$$\E^u_n[Z_\s(\vT)] = {n \choose pn} \P^u_n[ (\pi^\s_i,\chi) \textrm{ has type } \vT_i ~\forall i ].$$
The events $\{(\pi^\s_i,\chi) \textrm{ has type } \vT_i\}_{i=1}^d$ are jointly independent. So
\begin{eqnarray}\label{eqn:part4}
\E^u_n[Z_\s(\vT)] = {n \choose pn} \prod_{i=1}^d \P^u_n[ (\pi^\s_i,\chi) \textrm{ has type } \vT_i ].
\end{eqnarray}
By symmetry, $\P^u_n[ (\pi^\s_i,\chi) \textrm{ has type } \vT_i ]$ is the number of $k$-partitions $\pi$ such that $(\pi,\chi)$ has type $\vT_i$ divided by the number of $k$-partitions. By Lemma \ref{lem:numberofpartitions} and (\ref{eqn:part}),
\begin{eqnarray*}
&& \frac{1}{n} \log \P^u_n[ (\pi^\s_i,\chi) \textrm{ has type } \vT_i ] \\
&=&  -H(p,1-p)+H(\vec{T}_i)-\sum_{j=0}^k T_{ij} \log(j!(k-j)!)  +(1/k) \log (k-1)! + O(n^{-1}\log(n)).
\end{eqnarray*}
Combine this with (\ref{eqn:part4})  to obtain 
\begin{eqnarray*}
&&(1/n) \log \E^u_n[Z_\s(\vT)]\\
 &=& (1-d) H(p,1-p) +H(\vec{T}) -  \sum_{i=1}^d \sum_{j=0}^k T_{ij}\log(j!(k-j)!) +(d/k) \log (k-1)! + O(n^{-1}\log(n)).
\end{eqnarray*}
This simplifies to the formula for $F(\vT)$ using the assumption that $\sum_{j=0}^k T_{ij}=1/k$ for all $i$. 

 \end{proof}
 
\begin{proof}[Proof of Theorem \ref{thm:almost}]
By Lemma \ref{lem:diophantine} applied to $\sM$, continuity of $F$ and compactness of $\sM$, 
\begin{eqnarray*}
&&\lim_{n\to\infty} \sup\left\{ F(\vT):~\vT \in \sM_n \textrm{ and } \sum_{i=1}^d \sum_{j=0,k} T_{ij} \le \eps\right\} \\
&&=\sup\left\{ F(\vT):~\vT \in \sM \textrm{ and } \sum_{i=1}^d \sum_{j=0,k} T_{ij} \le \eps\right\}.
\end{eqnarray*}
Theorem \ref{thm:almost} now follows from Lemma \ref{lem:first} by continuity of $F$ and compactness of $\sM$.
\end{proof}

\subsection{Equitable colorings}\label{sec:equitable}

\begin{proof}[Proof of Theorem \ref{thm:1moment}]
Let $\sM_0$ be the set of all $\vT \in \sM$ such that $T_{ij}=0$ whenever $j \in \{0,k\}$.  By Lemma \ref{lem:first}, it suffices to show that $F$ admits a unique global maximum on $\sM_0$ and moreover if $\vT\in \sM_0$ is the global maximum then $p(\vT)=1/2$ and $F(\vT)=f(d,k)$. 

The function $F$ is symmetric in the index $i$. To exploit this, let $\sM'$ be the set of all vectors $\vt=(t_j)_{j=1}^{k-1}$ such that $t_{j} \ge 0$ for all $j$ and  $\sum_{j=1}^{k-1} t_{j}=1/k$. Let
\begin{eqnarray*}
p(\vt) &=& \sum_{j=1}^{k-1} j t_{j} \\
F(\vt) &=& dH(\vt) + (1-d)H(p,1-p) - (d/k)\log k + d \sum_{j=1}^{k-1} t_{j} \log {k \choose j}.
\end{eqnarray*}
Note that $F(\vt)=F(\vT)$ if $\vT$ is defined by $\vT_{ij}=\vt_j$ for all $i,j$. Moreover, since Shannon entropy is strictly concave, for any $\vT \in \sM_0$, if $\vt$ is defined to be the average: $\vt_j = d^{-1} \sum_{i=1}^d \vT_{ij}$ then $F(\vt) \ge F(\vT)$ with equality if and only if $\vt_j = \vT_{ij}$ for all $i,j$. So it suffices to show that $F$  admits a unique global maximum on $\sM'$ and moreover if $\vt \in \sM'$ is the global maximum then $p(\vt)=1/2$ and $F(\vt)=f(d,k)$. 

Because $\frac{\partial H(\vt)}{\partial t_j} = -[\log(t_j)+1]$, $\frac{\partial p}{\partial t_j} = j$, and $\frac{\partial H(p,1-p)}{\partial t_j} = j \log\left( \frac{1-p}{p}\right)$,
\begin{eqnarray*}
\frac{\partial F}{\partial t_j}= -d[\log(t_j)+1] + (1-d)j \log\left( \frac{1-p}{p}\right) + d \log {k \choose j}.
\end{eqnarray*}
Since this is positive infinity whenever $t_j = 0$, it follows that every maximum of $F$ occurs in the interior of $\sM'$. The method of Lagrange multipliers implies that, at a critical point, there exists $\l\in \R$ such that 
$$\nabla F = \l \nabla \left( \vt \mapsto \sum_j t_j\right) = (\l,\l,\ldots,\l).$$
So at a critical point,
$$\frac{\partial F}{\partial t_j} = -d[\log(t_j)+1] + (1-d)j \log\left( \frac{1-p}{p}\right) + d \log {k \choose j} = \l.$$
Solve for $t_j$ to obtain
$$t_j = \exp(-\l/d - 1){k \choose j} \left( \frac{1-p}{p}\right)^{j(1-d)/d}.$$
Note
$$1 = k \sum_{j=1}^{k-1} t_j$$
$$1 = 1/p \sum_{j=1}^{k-1} jt_j$$
implies
$$0 =\sum_{j=1}^{k-1}  (k-j/p) t_j =\sum_{j=1}^{k-1}  (pk-j){k \choose j} \left( \frac{1-p}{p}\right)^{j(1-d)/d}.$$
So define
$$g(x):=\sum_{j=1}^{k-1}  (kx-j){k \choose j} \left( \frac{1-x}{x}\right)^{j(1-d)/d}.$$
It follows from the above that $g(p(\vt))=0$ whenever $\vt$ is a critical point. 

We claim that $g(x)=0$ if and only if $x =1/2$ (for $x \in (0,1)$). The change of variables $j\mapsto k-j$ in the formula for $g$ shows that $g(1-x)=-\left( \frac{x}{1-x}\right)^{k(1-d)/d}g(x)$. So it is enough to prove that $g(x)<0$ for $x \in (0,1/2)$. 
 
To obtain a simpler formula for $g$, set $y(x) = \left( \frac{1-x}{x}\right)^{(1-d)/d}$. The binomial formula implies
\begin{eqnarray*}
g(x)&=&\sum_{j=1}^{k-1}  (kx-j){k \choose j} y^j \\
&=& kx[ (1+y)^k  - 1 - y^k] - ky[ (1+y)^{k-1}-y^{k-1}] \\
&=& k[ (x(1+y) - y)(1+y)^{k-1} -x + (-x+1)y^k].
\end{eqnarray*}
Because $0<x<1/2$, $y > \left( \frac{x}{1-x}\right)$ which implies that the middle coefficient $(x(1+y) - y) = x - y(1-x) < 0$. So
$$g(x)/k < (1-x)y^k  - x<0$$
where the last inequality holds because 
$$y^k=\left( \frac{x}{1-x}\right)^{k(d-1)/d} < \frac{x}{1-x}$$
assuming $k(d-1)/d >1$. This proves the claim. 

So if $\vt$ is a critical point then $p(\vt)=1/2$. Put this into the equation above for $t_j$ to  obtain
$$t_j = C {k \choose j}$$
where $C=\exp(-\l/d - 1)$. Because 
$$1/k = \sum_{j=1}^{k-1} t_j  = C \sum_{j=1}^{k-1} {k \choose j} = C(2^k - 2)$$
it must be that
\begin{eqnarray}\label{max-type}
t_j = \frac{1}{k(2^k-2)}{k \choose j}.
\end{eqnarray}
The formula $F(\vt)=f(d,k)$ now follows from a straightforward computation.  
\end{proof}

\section{The second moment}\label{sec:2}

This section gives an estimate on the expected number of proper colorings at a given Hamming distance from the planted coloring. This computation yields (\ref{eqn:2}) of Lemma \ref{lem:key} as a corollary. It also reduces the proof of (\ref{eqn:3}) to obtaining an estimate on the typical number of proper colorings near the planted coloring.

Before stating the main result, it seems worthwhile to review notation. Fix $n>0$ with $n \in 2\Z \cap k\Z$.  Fix an equitable 2-coloring $\chi:[n] \to \{0,1\}$. This is the {\bf planted coloring}. The planted model $\P^p_n$ is the uniform probability measure on the set $\Hom_{\chi}(\G,\sym(n))$ of all uniform homomorphisms $\s$ such that $\chi$ is $\s$-proper. Also let $Z_e:\Hom_{\chi}(\G,\sym(n)) \to \N$ be the number of equitable proper 2-colorings. For $\d \in [0,1]$, let $Z_\chi(\d;\cdot):  \Hom_{\chi}(\G,\sym(n))  \to \N$ be the number of equitable proper 2-colorings $\tchi$ such that $|d_n(\chi,\tchi) -\d|<1/2n$ where $d_n$ is the normalized Hamming distance defined by
$$d_n(\chi,\tchi) = n^{-1} \#\{ v\in [n]:~ \chi(v) \ne \tchi(v)\}.$$
We will also write $Z_\chi(\d;\sigma)=Z_\chi(\d)=Z(\d)$ when $\chi$ and/or $\sigma$ are understood.

The main result of this section is:
\begin{thm}\label{thm:cluster}
With notation as above, for any $0\le \delta \le 1$ such that $\d n/2$ is an integer,
$$ \frac{1}{n} \log \E^p_n[Z(\delta)] = \psi_0(\d) + O_\d(n^{-1}) + O(n^{-1}\log(n))$$
(for $n\ge 2$) where 
$$\psi_0(\d) = (1-d)H(\d,1-\d) + dH_0(\d,1-\d) + \frac{d}{k} \log\left( 1 - \frac{1-\d_0^k -(1-\d_0)^k }{2^{k-1}-1} \right),$$
$\d_0$ is defined to be the unique solution to
$$\d_0 \frac{ 1-2^{2-k} + (\d_0/2)^{k-1}  }{ 1-2^{2-k} + 2(\d_0/2)^k +2( (1-\d_0)/2)^k } = \d$$
and 
\begin{eqnarray*}
H(\d,1-\d) &:=& -\d\log\d - (1-\d)\log(1-\d),\\
H_0(\d,1-\d) &:=& -\d\log\d_0 - (1-\d)\log(1-\d_0).
\end{eqnarray*}
Moreover, the constant implicit in the error term $O(n^{-1}\log(n))$ may depend on $k$ but not on $\d$. The constant implicit in the $O_\d(n^{-1})$ term depends continuously on $\d$ for $\d \in (0,1/2]$. 
\end{thm}

\begin{remark}
If $\d_0=\d$ then $\d=1/2$. In the general case, $\d_0 = \d+O(2^{-k})$. Theorem \ref{thm:cluster} parallels similar results in  \cite{MR2263010, coja-zdeb-hypergraph} for the random hyper-graph $H_k(n,m)$. This is explained in more detail in the next subsection.
\end{remark}

The strategy behind the proof of Theorem \ref{thm:cluster} is as follows. We need to estimate the expected number of equitable colorings at distance $\d$ from the planted coloring. By symmetry, it suffices to fix another coloring $\tchi$ that is at distance $\d$ from the planted coloring and count the number of uniform homomorphisms $\s$ such that both $\chi$ and $\tchi$ are proper with respect to $G_\s$. This can be handled one generator at a time. Moreover, only the orbit-partition induced by a generator is used in this computation. So, for fixed $\chi,\tchi$, we need to estimate the number of $k$-partitions of $[n]$ that are bi-chromatic under both $\chi$ and $\tchi$. To make this strategy precise, we need the next definitions.

\begin{defn}
Let $\tchi$ be an equitable 2-coloring of $[n]$. An edge $P \subset [n]$ is {\bf $(\chi,\tchi)$-bichromatic} if $\chi(P) = \tchi(P) = \{0,1\}$. Recall that a $k$-partition is a partition $\pi=\{P_1,\ldots, P_{n/k}\}$ of $[n]$ such that every part $P \in \pi$ has cardinality $k$. A $k$-partition $\pi$ is {\bf $(\chi,\tchi)$-bichromatic} if every part $P \in \pi$ is $(\chi,\tchi)$-bichromatic.

Given a $(\chi,\tchi)$-bichromatic edge $P \subset [n]$ of size $k$, there is a $2\times 2$ matrix $\ve(\tchi,P)$ defined by
$$\ve_{i,j}(\tchi,P)  =  |P \cap \chi^{-1}(i) \cap \tchi^{-1}(j)|.$$
Let $\cE$ denote the set of all such matrices (over all $P,\tchi$). This is a finite set. To be precise, $\cE$ is the set of all $2\times 2$ matrices $\ve=(e_{ij})_{i,j=0,1}$ such that 
\begin{itemize}
\item $e_{ij} \in \{0,1,\ldots, k\}$ for all $i,j$
\item $0<e_{10} + e_{11} < k$
\item $0<e_{01}+e_{11}<k$
\item $\sum_{i,j} e_{ij} = k$.
\end{itemize}
If $\pi$ is a $(\chi,\tchi)$-bichromatic $k$-partition then it induces a function $t_{\tchi,\pi}:\cE \to [0,1]$ by
$$t_{\tchi,\pi}(\ve) = n^{-1}\#\left\{P \in \pi:~ \ve=\ve(\tchi,P) \right\}.$$
Let $\cT$ be the set of all functions $t:\cE \to [0,1]$ satisfying
\begin{itemize}
\item $\sum_{\ve\in \cE} t(\ve)=1/k$,
\item $\sum_{\ve \in \cE} (e_{10} + e_{11})t(\ve) = 1/2$,
\item $\sum_{\ve \in \cE} (e_{01} + e_{11})t(\ve) = 1/2$.
\end{itemize}
Also let $\cT_n$ be the set of $t \in \cT$ such that $nt(\ve)$ is integer-valued for each $\ve \in \cE$.
A $k$-partition $\pi$ has {\bf type $(\chi,\tchi,t)$} if $t=t_{\tchi,\pi}$.
\end{defn}

\begin{lem}
Given an equitable 2-coloring $\tchi:[n] \to \{0,1\}$, let $p^\tchi=(p^\tchi_{ij})$ be the $2\times 2$ matrix 
$$p^\tchi_{ij} = (1/n)| \chi^{-1}(i) \cap \tchi^{-1}(j)|.$$
Then
\begin{displaymath}
p^\tchi = \left[\begin{array}{cc}
1/2 - d_n(\chi,\tchi)/2 & d_n(\chi,\tchi)/2 \\
d_n(\chi,\tchi)/2 & 1/2- d_n(\chi,\tchi)/2
\end{array}\right]
\end{displaymath}
In particular, $p^\tchi$ is determined by the Hamming distance $d_n(\chi,\tchi)$.
\end{lem}

\begin{proof}
Let $p=p^\tchi$. The lemma follows from this system of linear equations:
\begin{eqnarray*}
1/2 &=& p_{01} + p_{11} \\
1/2 &=& p_{10} + p_{11} \\
d_n(\chi,\tchi) &=& p_{01} + p_{10}\\
1 &=& p_{00} + p_{01} + p_{10} + p_{11}.
\end{eqnarray*}
The first two occur because  both $\chi$ and $\tchi$ are equitable. The third follows from the definition of normalized Hamming distance and the last holds because $\{\chi^{-1}(i) \cap \tchi^{-1}(j)\}_{i,j \in \{0,1\}}$ partitions $[n]$.
\end{proof}

For $t \in \cT$, define the $2\times 2$ matrix $p^t=(p^t_{ij})$ by
\begin{eqnarray*}
p^t_{ij} := \sum_{\ve \in \cE} e_{ij} t(\ve).
\end{eqnarray*}
If  $\pi$ is a $k$-partition that has type $(\chi,\tchi,t)$ (for some equitable $\tchi$) then $p^\tchi = p^t$. This motivates the definition. 

The main combinatorial estimate we will need is:
\begin{lem}\label{lem:g}
Let $t\in \cT_n$ and $\tchi:[n]\to\{0,1\}$ be equitable. Suppose $p^t=p^\tchi$. Let $g(\tchi,t)$ be the number of $k$-partitions of type $(\chi,\tchi,t)$. Also let 
$$G(t):=(1-1/k)(\log (n)-1) - H(p^t) -(1/k)\log (k!) + H(t)  + \sum_{\ve} t(\ve) \log {k \choose \ve}$$
where ${k \choose \ve}$ is the multinomial $\frac{k!}{e_{00}!e_{01}!e_{10}!e_{11}!}$. Then
$$(1/n)\log g(\tchi,t) = G(t) + O(n^{-1}\log(n))$$
(for $n\ge 2$) where the constant implicit in the error term depends on $k$ but not on $\tchi$ or $t$.
\end{lem}

\begin{proof}

The following algorithm constructs all such partitions with no duplications:
\begin{itemize}
\item[Step 1.] Choose a partition $\{Q_{ij}^\ve:~ i,j \in \{0,1\}, \ve \in \cE\}$ of $\chi^{-1}(i) \cap \tchi^{-1}(j)$ such that
$$|Q_{ij}^\ve| = e_{ij}t(\ve) n.$$

\item[Step 2.] For $i,j\in \{0,1\}$ and $\ve \in \cE$, choose an unordered partition $\pi^\ve_{ij}$ of $Q_{ij}^\ve$  into $t(\ve) n$ sets of size $e_{ij}$. 

\item[Step 3.] For $i,j\in \{0,1\}$ with $(i,j) \ne (0,0)$ and $\ve \in \cE$, choose a bijection $\a^\ve_{ij}: \pi^\ve_{00} \to \pi^\ve_{ij}$. 

\item[Step 4.]  The $k$-partition consists of all sets of the form $P \cup \bigcup_{i,j\in \{0,1\}, (i,j) \ne (0,0)} \alpha^{\ve}_{ij}(P)$ over all $P \in \pi^\ve_{00}$ and $\ve \in \cE$. 
\end{itemize}

The number of choices in Step 1 is 
$$ \prod_{i,j \in \{0,1\} }  |\chi^{-1}(i) \cap \tchi^{-1}(j)|! \prod_{\ve \in\cE}  (e_{ij} t(\ve)n)!^{-1}.$$

The combined number of choices in Steps 1 and 2 is 
$$ \prod_{i,j \in \{0,1\} }  |\chi^{-1}(i) \cap \tchi^{-1}(j)|! \prod_{\ve \in\cE}  e_{ij}!^{-t(\ve)n} (t(\ve)n)!^{-1} = \left( \prod_{\ve \in\cE} (t(\ve)n)! \right)^{-4} \prod_{i,j \in \{0,1\} }  |\chi^{-1}(i) \cap \tchi^{-1}(j)|! \prod_{\ve \in\cE}  e_{ij}!^{-t(\ve)n}.$$

The number of choices in Step 3 is $\prod_{\ve \in\cE} (t(\ve)n)!^3.$ So
$$g(\tchi,t) = \left(\prod_{i,j \in \{0,1\}} |\chi^{-1}(i) \cap \tchi^{-1}(j)|! \right) \left( \prod_{\ve \in\cE} (t(\ve)n)! \right)^{-1} \left(\prod_{i,j\in \{0,1\}} \prod_{\ve \in\cE} e_{ij}!^{-t(\ve)n} \right).$$
An application of Stirling's formula gives
$$(1/n)\log g(\tchi,t) =(1-1/k)(\log (n)-1) - H(p^t) + H(t) - \sum_{\ve,i,j} t(\ve) \log (e_{ij}!) + O(n^{-1}\log(n))$$
(for $n\ge 2$) where the constant implicit in the error term depends on $k$ but not on $\tchi$ or $t$. 

Since $\sum_\ve t(\ve)=1/k$, 
$$\sum_{\ve} t(\ve) \log {k \choose \ve} = (1/k)\log (k!) - \sum_{\ve,i,j} t(\ve) \log (e_{ij}!).$$
Substitute this into the formula above to finish the lemma.

\end{proof}

Next we use Lagrange multipliers to maximize $G(t)$. To be precise, for $\d \in [0,1]$, let $\cT(\d)$ be the set of all $t \in \cT$ such that $p^t_{01} = \d/2$. Define $\cT_n(\d) = \cT(\delta) \cap \cT_n$. To motivate this definition, observe that if $\tchi$ is an equitable 2-coloring and $\d=d_n(\chi,\tchi)$ then $p^\tchi_{01}=\d/2$. So if $\pi$ is a $k$-partition with type $(\chi,\tchi,t)$ then $p^t_{01}=\d/2$.

\begin{lem}\label{lem:max}
Let $\d \in [0,1]$. Then there exists a unique $s_\d \in \cT(\d)$ such that 
$$\max_{t\in \cT(\d)} G(t) = G(s_\d).$$ 
Moreover, if $\d_0$, $C>0$ and $t_\d \in \cT(\d)$ are defined by
\begin{eqnarray*}
\frac{\d}{2} &=& \frac{\d_0}{2} \frac{ 1-2^{2-k} + (\d_0/2)^{k-1}  }{ 1-2^{2-k} + 2(\d_0/2)^k +2( (1-\d_0)/2)^k }\\
C &=& \frac{1}{k[1-2^{2-k} + 2(\d_0/2)^k +2( (1-\d_0)/2)^k] }\\
t_\d(\ve) &=& C \left(\frac{1-\d_0}{2} \right)^{e_{00}+e_{11}} \left(\frac{\d_0}{2} \right)^{e_{01}+e_{10}} {k\choose \ve}
\end{eqnarray*}
then $s_\d=t_\d$.
\end{lem}

\begin{proof}
Define $F:\cT \to \R$ by
$$F(t)=H(t) + \sum_{\ve} t(\ve) \log {k \choose \ve}.$$ 
For all $t \in \cT(\d)$, $G(t)-F(t)$ is constant in $t$. Therefore, it suffices to prove the lemma with $F$ in place of $G$. 

The function $F$ is concave over $t \in \cT(\d)$. This implies the existence of a unique $s_\d \in \cT(\d)$ such that 
$$\max_{t\in \cT(\d)} F(t) = F(s_\d).$$ 

By definition, $\cT(\d)$ is the set of all functions $t:\cE \to [0,1]$ satisfying
\begin{eqnarray*}
1/k &=&  \sum_{\ve\in \cE} t(\ve) \\
p_{ij} &=&  \sum_{\ve\in \cE} e_{ij} t(\ve),
\end{eqnarray*}
where $p=(p_{ij})$ is the matrix
\begin{displaymath}
p = \left[\begin{array}{cc}
1/2 - \d/2 & \d/2 \\
\d/2 & 1/2- \d/2
\end{array}\right].
\end{displaymath}

For any $\ve \in \cE$,
\begin{eqnarray}\label{f1}
\frac{\partial F}{\partial t(\ve)}  = -\log t(\ve) - 1 + \log {k \choose \ve}.
\end{eqnarray}
Since this is positive infinity when $t(\ve)=0$,  $s_\d$ must lie in the interior of $\cT(\d)$.  
By the method of Lagrange multipliers there exists $\l \in \R$ and a $2\times 2$ matrix $\vec{\mu}$ such that
\begin{eqnarray}\label{f2}
\frac{\partial F}{\partial t(\ve)}(s_\d) = \l + \vec{\mu} \cdot \ve.
\end{eqnarray}
Evaluate (\ref{f1}) at $s_\d$, use $(\ref{f2})$ and solve for $s_\d(\ve)$ to obtain
$$s_\d(\ve) = C_0 {k \choose \ve} x_{00}^{e_{00}} x_{01}^{e_{01}} x_{10}^{e_{10}}x_{11}^{e_{11}}$$
for some constants $C_0, x_{ij}$. In fact, since $F$ is concave, $s_\d$ is the unique critical point and so it is the only element of $\cT(\d)$ of this form. So it suffices to check that the purported $t_\d$ given in the statement of the lemma has this form and that it is in $\cT(\d)$ as claimed. The former is immediate while the latter is a tedious but straightforward computation. For example, to check that $\sum_\ve t_\d(\ve)=1/k$, observe that, by the multinomial formula for any $(x_{ij})_{i,j\in\{0,1\}}$,
\begin{eqnarray*}
&&\sum_{\ve \in \cE}  {k \choose \ve} x_{00}^{e_{00}} x_{01}^{e_{01}} x_{10}^{e_{10}}x_{11}^{e_{11}} \\
&=& \Big[ (x_{00} + x_{01} + x_{10} + x_{11})^k \\
&~~~& - (x_{00} + x_{01})^k  -  (x_{00} + x_{10})^k  - (x_{11} + x_{01})^k -  (x_{11} + x_{10})^k  + x_{00}^k  + x_{01}^k + x_{10}^k + x_{11}^k \Big].
\end{eqnarray*}
Substitute $x_{00}=x_{11}=\frac{1-\d_0}{2}$ and $x_{01}=x_{10}=\d_0/2$ to obtain
\begin{eqnarray*}
\sum_{\ve \in \cE} t_\d(\ve) &=& C\left[ 1 - 4(1/2)^k  + 2\left(\frac{1-\d_0}{2}\right)^k  + 2\left(\frac{\d_0}{2}\right)^k \right] = 1/k.
\end{eqnarray*}
The rest of the verification that $t_\d \in \cT(\d)$ is left to the reader.
\end{proof}


\begin{proof}[Proof of Theorem \ref{thm:cluster}]

 Let $\sE(\d)$ be the set of all equitable 2-colorings $\tchi:[n] \to \{0,1\}$ such that $d_n(\tchi,\chi)=\d$. Also let $F_{\tchi} \subset \Hom_{\rm{unif}}(\G,\sym(n))$ be the set of all $\s$ such that $\tchi$ is a proper 2-coloring of the hyper-graph $G_\s$.  By linearity of expectation,
$$\E^p_n[Z_\chi(\d)]  = \sum_{\tchi \in \sE(\d)} \P^u_n(F_{\tchi} | F_{\chi}).$$
The cardinality of $\sE(\d)$ is ${ n/2 \choose \d n/2}^2$. By Stirling's formula
\begin{eqnarray}\label{E(d)}
n^{-1} \log { n/2 \choose \d n/2}^2 = H(\d, 1-\d) + O(n^{-1}\log(n)).
\end{eqnarray}
We have $\P^u_n(F_{\tchi} | F_{\chi})$ is the same for all $\tchi \in \sE(\d)$. This follows by noting that the distribution of hyper-graphs in the planted model is invariant under any permutation which fixes $\chi$. If $\eta, \eta'$ are two configurations with $d_n(\eta, \chi) = d_n(\eta', \chi) = \delta$ then there is a permutation $\pi \in \sym(n)$ which fixes $\chi$ and such that $\eta\circ \pi = \eta'$. To see this note that we simply need to find a $\pi \in \sym(n)$ which maps the sets $\chi^{-1}(i) \cap \eta^{-1}(j)$ to $\chi^{-1}(i) \cap \eta'^{-1}(j)$ for each $i,j \in \{0,1\}$. Such a map exists since for each $i,j$ the two sets have the same size. It follows that
\begin{eqnarray}\label{symmetry2}
n^{-1} \log \E^p_n[Z_\chi(\d)]  =  H(\d, 1-\d) + n^{-1} \log \P^u_n(F_{\tchi} | F_{\chi}) + O(n^{-1}\log(n))
\end{eqnarray}
for any fixed $\tchi \in \sE(\d)$. 

For $1\le i \le d$, let $F_{\chi, i}$ be the set of uniform homomorphisms $\s$ such that the orbit-partition of $\s(s_i)$ is $\chi$-bichromatic in the sense that $\chi(P) = \{0,1\}$ for every $P$ in the orbit-partition of $\s(s_i)$. Then the events $\{F_{\chi,i} \cap F_{\tchi,i}  \}_{i=1}^d$ are i.i.d. and 
$$F_{\chi} \cap F_{\tchi} = \bigcap_{i=1}^d F_{\chi,i} \cap F_{\tchi,i}.$$
 Therefore,
\begin{eqnarray}\label{symmetry}
\P^u_n(F_{\tchi} | F_{\chi}) = \frac{\P^u_n(F_{\tchi,1}  \cap F_{\chi,1})^d  }{ \P^u_n(F_{\chi})}.
\end{eqnarray}

Note $\P^u_n(F_{\chi,1} \cap F_{\tchi,1})$ is, up to sub-exponential factors, equal to the maximum of $g(\tchi,t)$ over $t \in \cT_n(\d)$ divided by the number of $k$-partitions of $[n]$. So equation (\ref{eqn:part}) implies
$$ \frac{1}{n} \log \P^u_n(F_{\tchi,1}  \cap F_{\chi,1}) =  \max_{t\in \cT_n(\d)} \frac{1}{n} \log g(\tchi,t) - (1-1/k)(\log(n)-1) +(1/k)\log ((k-1)!) + O(n^{-1}\log n).$$
So Lemma \ref{lem:g} implies
$$ \frac{1}{n} \log \P^u_n(F_{\tchi,1}  \cap F_{\chi,1}) =  \max_{t\in \cT_n(\d)} G(t) - (1-1/k)(\log(n)-1) +(1/k)\log ((k-1)!) + O(n^{-1}\log n).$$
We apply Lemma \ref{lem:diophantine} to $\cT(\delta)$ to obtain the existence of $s^{(n)}_\d \in \cT_n(\d)$ with $\|s^{(n)}_\d - t_\d\|_\infty < O_\d(n^{-1})$ where the constant implicit in the $O_\d(n^{-1})$ term depends continuously on $\d$ for $\d \in (0,1/2]$. Since $G$ is differentiable in a neighborhood of $t_\d$, we have 
$$\max_{t\in \cT_n(\d)} G(t)=  G(t_\d) + O_\d(n^{-1}).$$
 So Lemma \ref{lem:max}  implies
\begin{eqnarray*}
 \frac{1}{n} \log \P^u_n(F_{\tchi,1}  \cap F_{\chi,1})  &=& -H(\vp) + H(t_\d)  - \sum_{i,j,\ve} t_\d(\ve) \log(e_{ij}!) + (1/k)\log (k-1)! \\
 && + O_\d(n^{-1}) + O(n^{-1}\log(n)).
\end{eqnarray*}
Since $\vp=(\d/2,\d/2,(1-\d)/2,(1-\d)/2)$, $H(\vp) = H(\d,1-\d) + \log(2).$ So
\begin{eqnarray}\label{woody}
 \frac{1}{n} \log \P^u_n(F_{\tchi,1}  \cap F_{\chi,1}) &=&\notag -H(\d,1-\d) - \log(2) + H(t_\d)  - \sum_{i,j,\ve} t_\d(\ve) \log(e_{ij}!) \\
 &&  + (1/k)\log (k-1)! +  O_\d(n^{-1}) + O(n^{-1}\log(n)).
\end{eqnarray}
On the other hand, Theorem \ref{thm:1moment} implies
\begin{eqnarray*}
\frac{1}{n} \log \P^u_n(F_{\chi}) &=& \frac{1}{n} \log \left( {n\choose n/2}^{-1} \E^u_n[Z_e(\s)]\right) \\
&=& (d/k)\log(1-2^{1-k}) + O(n^{-1}\log(n)).
\end{eqnarray*}
Combine this result with (\ref{symmetry2}), (\ref{symmetry}) and (\ref{woody}) to obtain 
\begin{eqnarray*}
 n^{-1} \log \E^p_n[Z_\chi(\d)] &=& (1-d)H(\d,1-\d) - d\log(2) + dH(t_\d)- d\sum_{i,j,\ve} t_\d(\ve) \log(e_{ij}!)\\
&&   + (d/k)\log (k-1)! - (d/k)\log(1-2^{1-k}) + O_\d(n^{-1}) + O(n^{-1}\log(n)).
\end{eqnarray*}
Since $ \sum_{\ve}  t_\d(\ve) =1/k$,
$$\sum_{\ve}  t_\d(\ve) \log {k \choose \ve } = (1/k) \log k! - \sum_{i,j,\ve} t_\d(\ve) \log(e_{ij}!).$$
Substitute this into the previous equation to obtain 
$$n^{-1} \log \E^p_n[Z_\chi(\d)] = \psi_0(\d) + O_\d(n^{-1}) + O(n^{-1}\log(n))$$
where 
\begin{eqnarray*}
\psi_0(\d)&=&(1-d)H(\d,1-\d) - d\log(2) + dH(t_\d)  + d \sum_{\ve \in \cE}  t_\d(\ve) \log {k \choose \ve } \\
&&- (d/k)\log k  -(d/k)\log(1-2^{1-k}).
\end{eqnarray*}
Observe that in every estimate above, the constant implicit in the error term does not depend on $\d$. To finish the lemma, we need only simplify the expression for $\psi_0$. 

By Lemma \ref{lem:max},
\begin{eqnarray*}
H(t_\d) &=& -\sum_{\ve}t_\d (\ve)\log t_\d(\ve)\\ &=& -\sum_{\ve}t_\d (\ve) \left(\log C + (e_{00} + e_{11})\log\left(\frac{1-\d_0}{2}\right) + (e_{01} + e_{10})\log\left(\frac{\d_0}{2}\right) + \log {k \choose \ve}\right)\\
&=& -(1/k)(\log C) - (1-\d)\log (1-\d_0) - \d\log (\d_0) + \log 2 - \sum_\ve t_\d(\ve) \log{k \choose \ve} \\
&=&-(1/k)(\log C) + H_0(\d,1-\d)  + \log 2 - \sum_\ve t_\d(\ve) \log{k \choose \ve}.
\end{eqnarray*}
Combined with the previous formula for $\psi_0$, this implies
\begin{eqnarray*}
\psi_0(\d)&=&(1-d)H(\d,1-\d)    -(d/k)\log C + dH_0(\d,1-\d) - (d/k)\log k  -(d/k)\log(1-2^{1-k}).
\end{eqnarray*}

To simplify further, use the formula for $C$ in Lemma \ref{lem:max} to obtain
\begin{eqnarray*}
-(d/k)\left(\log C + \log k + \log(1-2^{1-k})\right) &=& (d/k)\log \frac{1 - 2^{2-k} + 2(\d_0/2)^k + 2((1-\d_0)/2)^k}{1-2^{1-k}} \\ &=& (d/k) \log \left(1- \frac{1 - \d_0^k - (1-\d_0)^k}{2^{k-1}-1}\right).
\end{eqnarray*}
Thus $\psi_0(\d) = (1-d)H(\d,1-\d) + dH_0(\d,1-\d) + \frac{d}{k}\log(1- \frac{1-\d_0^k-(1-\d_0)^k}{2^{k-1}-1})$.


\end{proof}

\subsection{Analysis of $\psi_0$ and the proof of Lemma \ref{lem:key} inequality (\ref{eqn:2})}

Theorem \ref{thm:cluster} reduces inequality (\ref{eqn:2}) to analyzing the function $\psi_0$. A related function $\psi$, defined by 
$$\psi(x):=H(x,1-x) + \frac{d}{k} \log \left( 1 - \frac{1 - x^k - (1-x)^k}{2^{k-1}-1}\right)$$
has been analyzed in \cite{MR2263010, coja-zdeb-hypergraph}. It is shown there $\psi(x)$ is the exponential rate of growth of the number of proper colorings at normalized distance $x$ from the planted coloring in the model $H_k(n,m)$. Moreover, if $r=d/k$ is close to $\frac{\log(2)}{2} \cdot 2^k - (1+\log(2))/2$ then the global maximum of $\psi(x)$ is attained at some $x\in (0,2^{-k/2})$. Moreover, $\psi$ has a local maximum at $x=1/2$ and is symmetric around $x=1/2$. It is negative in the region $(2^{-k/2}, 1/2 - 2^{-k/2})$. We will not need these facts directly, and mention them only for context, especially because we will obtain similar results for $\psi_0$.

The relevance of $\psi$ to $\psi_0$ lies in the fact that
\begin{eqnarray}\label{eqn:from psi to psi}
\psi_0(\d)=\psi(\d_0) - (H(\d_0,1-\d_0) - H_0(\d,1-\d)) + (d-1)\left[ H_0(\d,1-\d) - H(\d,1-\d)\right].
\end{eqnarray}
As an aside, note that $H_0(\d,1-\d) - H(\d,1-\d)$ is the Kullback-Leibler divergence of the distribution $(\d,1-\d)$ with respect to $(\d_0,1-\d_0)$. 

To prove inequality (\ref{eqn:2}), we first estimate the difference $\psi_0(\d)-\psi(\d_0)$ and then estimate $\psi(\d_0)$. Because the estimates we obtain are useful in the next subsection, we prove more than what is required for just inequality (\ref{eqn:2}).

\begin{lem}\label{lem:4things} 
Suppose $0 \leq \d_0 \leq 1/2$. Define $\varepsilon \ge 0$ by $\d = \d_0(1-\varepsilon)$. Then 
\begin{eqnarray*}
H(\d_0,1-\d_0) - H_0(\d,1-\d) &=& \d_0\varepsilon\log\left(\frac{1-\d_0}{\d_0}\right) \ge 0, \\
H_0(\d,1-\d) - H(\d,1-\d)&=&  O(\d_0\varepsilon^2),\\
\varepsilon &=& O(2^{-k}), \\
(1-\d)_0 &=& 1-\d_0.
 \end{eqnarray*}
  The last equation implies $\psi_0(1-\d) = \psi_0(\d)$.
\end{lem}

\begin{proof}
The first equality follows from:
\begin{eqnarray*}
H(\d_0,1-\d_0) - H_0(\d,1-\d) &=& -\d_0\log\d_0 - (1-\d_0)\log(1-\d_0) + \d\log\d_0 + (1-\d)\log(1-\d_0) \\
&=& (\d_0 - \d)\log(1/\d_0) + (\d_0 - \d)\log(1-\d_0)\\
 &=& \d_0\varepsilon\log\left(\frac{1-\d_0}{\d_0}\right).
 \end{eqnarray*}
The second estimate follows from:
\begin{eqnarray*}
 H_0(\d,1-\d) - H(\d,1-\d)&=& \d\left(\log\d - \log\d_0 \right) +  (1-\d) \left( \log(1-\d) - \log(1-\d_0) \right)\\
 &=& \d\log(1-\varepsilon) + (1-\d)\log\left(\frac{1-\d}{1-\d_0}\right)\\
  &=& -\d\varepsilon + (1-\d)\log\left(1+ \frac{\d_0\varepsilon}{1-\d_0}\right) + O(\d_0\varepsilon^2) \\
   &=& -\d\varepsilon + \d_0\varepsilon + O(\d_0\varepsilon^2)  = O(\d_0\varepsilon^2). 
 \end{eqnarray*}
The third estimate follows from:
\begin{eqnarray}
\varepsilon &=& 1- \frac{\d}{\d_0} \nonumber\\
&=& 1-\frac{1 - 2^{2-k} + (\d_0/2)^{k-1}}{1 - 2^{2-k} + 2(\d_0/2)^{k} + 2((1-\d_0)/2)^k} \nonumber\\
&=& \frac{2(\d_0/2)^{k} -  (\d_0/2)^{k-1}+ 2((1-\d_0)/2)^k}{1 - 2^{2-k} + 2(\d_0/2)^{k} + 2((1-\d_0)/2)^k} \nonumber\\
&=&  \frac{(\d_0/2)^{k-1}(\d_0-1)+ 2((1-\d_0)/2)^k}{1 - 2^{2-k} + 2(\d_0/2)^{k} + 2((1-\d_0)/2)^k}\nonumber\\
&=&  2^{1-k}\cdot (1-\d_0) \cdot \frac{(1-\d_0)^{k-1} - \d_0^{k-1} }{1 - 2^{2-k} + 2(\d_0/2)^{k} + 2((1-\d_0)/2)^k}.\label{vareps}
 \end{eqnarray}
The denominator is $1+O(2^{-k})$ and the numerator is $O(2^{-k})$. The result follows.  

The last equation follows from:
\begin{eqnarray*}
 1-\d &=& 1 - \d_0 \left(\frac{ 1-2^{2-k} + (\d_0/2)^{k-1}  }{ 1-2^{2-k} + 2(\d_0/2)^k +2( (1-\d_0)/2)^k }\right) \\ 
 &=& \frac{1-2^{2-k} + 2(\d_0/2)^k +2\left( (1-\d_0)/2)^k - \d_0(1-2^{2-k} + (\d_0/2)^{k-1} \right)}{1-2^{2-k} + 2(\d_0/2)^k +2( (1-\d_0)/2)^k} \\
 &=& \frac{(1-\d_0)\left(1-2^{2-k}+((1-\d_0)/2)^{k-1}\right)}{1-2^{2-k} + 2(\d_0/2)^k +2( (1-\d_0)/2)^k}.
\end{eqnarray*}

The last expression shows that $(1-\d)_0 = 1-\d_0$.

\end{proof}



\begin{lem}\label{lem:psi}
Let $0 \le \eta$ be $O_k(1)$. If 
$$r =d/k= \frac{\log(2)}{2} \cdot 2^k - (1+\log(2))/2 +\eta$$
then
\begin{eqnarray*}
f(d,k) &=& \psi(1/2)=\psi_0(1/2) = (1-2\eta)2^{-k} + O(2^{-2k}) \\
\psi(2^{-k}) &=& 2^{-k} + O(k^2 2^{-2k}).
\end{eqnarray*}
In particular, if $k$ is sufficiently large then $\psi(2^{-k}) > f(d,k)$. 
\end{lem}

\begin{proof}
By direct inspection $f(d,k)=\psi(1/2)=\psi_0(1/2)$. By Taylor series expansion, $\log(1-2^{1-k}) = -2^{1-k} - 2^{1-2k} + O(2^{-3k})$. So
\begin{eqnarray*}
f(d,k)&=&\log(2) + r \log(1-2^{1-k}) \\
&=& \log(2) + \left(  \frac{\log(2)}{2} \cdot 2^k - (1+\log(2))/2 +\eta \right) \left( - 2^{1-k} - 2^{1-2k} \right) + O(r 2^{-3k}) \\
&=&(1-2\eta)2^{-k} + O(2^{-2k}). 
\end{eqnarray*}
Next we estimate $\psi(2^{-k})$. For convenience, let $x=2^{-k}$. Then
$$1 - x^k - (1-x)^k = k\cdot 2^{-k} + O(k^2 2^{-2k}).$$
Since $\log(1-x)=-x - x^2/2 + O(x^3)$, 
$$\log \left( 1 - \frac{1 - x^k - (1-x)^k}{2^{k-1}-1}\right) =  -2k \cdot 2^{-2k}   + O(k^2 2^{-3k}).$$
So
$$r \log \left( 1 - \frac{1 - x^k - (1-x)^k}{2^{k-1}-1}\right) = -k\log(2)\cdot 2^{-k} + O(k^2 2^{-2k}). $$
Also,
$$H(x,1-x) = (k\log(2) + 1)\cdot 2^{-k} + O(2^{-2k}).$$
Add these together to obtain
$$\psi(2^{-k}) = 2^{-k} + O(k^22^{-2k}).$$

\end{proof}

\begin{cor}
Inequality (\ref{eqn:2}) of Lemma \ref{lem:key} is true. To be precise, let $0<\eta_0$ be constant with respect to $k$. Then for all sufficiently large $k$ (depending on $\eta_0$), if 
$$r =d/k= \frac{\log(2)}{2} \cdot 2^k - (1+\log(2))/2 +\eta$$
for some $\eta\ge \eta_0$ with $\eta=O_k(1)$ then
$$f(d,k) < \liminf_{n\to\infty} n^{-1}\log\E_n^{p}[ Z(\s)].$$
\end{cor}

\begin{proof}
By definition,
$$n^{-1}\log\E_n^{p}[ Z(\s)] \ge \max_{\d\in [0,1/2]} n^{-1}\log\E_n^{p}[ Z(\d)].$$
By Theorem \ref{thm:cluster},
$$n^{-1}\log\E_n^{p}[ Z(\s)] \ge \max_{\d \in C_n} \psi_0(\d) + O_\d(n^{-1}) + O(n^{-1}\log(n))$$
where $C_n$ is the set of $\d\in [0,1/2]$ such that $\d n/2$ is an integer. Because $\psi_0:[0,1/2] \to \R$ is continuous, $\lim_{n\to\infty} \max_{\d \in C_n} \psi_0(\d) = \max_{\d \in [0,1/2]} \psi_0(\d)$. The constant implicit in the $O_\d(n^{-1})$ depends continuously on $\d$ for $\d \in (0,1/2]$, while the constant implicit in the $O(n^{-1}\log(n))$ error term does not depend on $\d$. So
$$\liminf_{n\to\infty} n^{-1}\log\E_n^{p}[ Z(\s)] \ge \max_{\d \in [\upsilon,1/2]} \psi_0(\d)$$
for any $0<\upsilon<1/2$. Because $\psi_0$ is continuous, 
\begin{eqnarray}\label{eqn:obvious}
\liminf_{n\to\infty} n^{-1}\log\E_n^{p}[ Z(\s)] \ge \max_{\d \in [0,1/2]} \psi_0(\d).
\end{eqnarray}
Because $H_0(\d,1-\d) - H(\d,1-\d) \ge 0$ (since it is a Kullback-Liebler divergence), the first equality of Lemma \ref{lem:4things} implies
\begin{eqnarray*}
\psi_0(\d) &=& \psi(\d_0) - (H(\d_0,1-\d_0) - H_0(\d,1-\d)) + (d-1)\left[ H_0(\d,1-\d) - H(\d,1-\d)\right] \\
&\ge & \psi(\d_0) - \d_0\varepsilon\log\left(\frac{1-\d_0}{\d_0}\right).
\end{eqnarray*}
By Lemma \ref{lem:psi},  $\psi(2^{-k}) = f(d,k)+2\eta 2^{-k} + O(k^2 2^{-2k})$. By  Lemma \ref{lem:4things}, $\varepsilon = O(2^{-k})$. As $\d$ varies over $[0,1/2]$, $\d_0$ also varies over $[0,1/2]$, so there exists $\d$ such that $\d_0=2^{-k}$. For this value of $\d$,
$$\psi_0(\d) \ge \psi(2^{-k}) -2^{-k}\varepsilon\log\left(\frac{1-2^{-k}}{2^{-k}}\right) \ge f(d,k) + 2\eta 2^{-k} + O(k^2 2^{-2k}).$$
Combined with (\ref{eqn:obvious}) this implies the Corollary.
\end{proof}

In the next subsection, we will need the following result.

\begin{prop}\label{prop:5cases}
Let
$$0<\eta_0 < (1-\log 2)/2.$$
Then there exists $k_0$ (depending on $\eta_0$) such that for all $k\ge k_0$ if 
$$r=d/k= \frac{\log(2)}{2} \cdot 2^k - (1+\log(2))/2 +\eta$$
for some $\eta \in [\eta_0,(1-\log 2)/2)$ then in the interval $[2^{-k/2},1-2^{-k/2}]$, $\psi_0$ attains its unique maximum at 1/2. That is,
$$\max\{\psi_0(\d):~ 2^{-k/2} \le \d  \le 1-2^{-k/2}\} = \psi_0(1/2)=f(d,k) = \log(2) + r \log(1-2^{1-k})$$
and if $\d \in [2^{-k/2},1-2^{-k/2}]$ and $\d \ne 1/2$ then $\psi_0(\d)<\psi_0(1/2)$. 
\end{prop}

\begin{proof}
By Lemma \ref{lem:4things}, it suffices to restrict $\d$ to the interval $[2^{-k/2},1/2]$ (because $\psi_0(\d)=\psi_0(1-\d)$). So we will assume $\d \in [2^{-k/2},1/2]$ without further mention.

Define $\psi_1$ by
$$\psi_1(\d_0)=\frac{d}{k} \log \left( 1 - \frac{1 - \d_0^k - (1-\d_0)^k}{2^{k-1}-1}\right).$$
Observe
\begin{eqnarray*}
\psi_1(\d_0) &=& r \left( - \frac{1 - \d_0^k - (1-\d_0)^k}{2^{k-1}-1} + O(4^{-k}) \right) \\ 
&=& -\log(2)[1-(1-\d_0)^k] + O(2^{-k}). \label{eqn:psi1}
\end{eqnarray*}

By (\ref{eqn:from psi to psi}) and the first inequality of Lemma \ref{lem:4things},
\begin{eqnarray}
\psi_0(\d) 
&\leq& \psi(\d_0) + (d-1)[H_0(\d,1-\d) - H(\d,1-\d)] \nonumber\\
&=& H(\d_0,1-\d_0) + \psi_1(\d_0) + (d-1)[H_0(\d,1-\d) - H(\d,1-\d)]\label{estimateb}.
\end{eqnarray}
Moreover, $(d-1)=O(k2^k)$ and, by Lemma \ref{lem:4things}, $H_0(\d,1-\d) - H(\d,1-\d) = O(\d_04^{-k})$. Therefore,
\begin{eqnarray}\label{psi-simple}
\psi_0(\d) &\le & H(\d_0,1-\d_0) + \psi_1(\d_0) + O(\d_0 k 2^{-k})\nonumber \\
&\le& H(\d_0,1-\d_0) -\log(2)[1-(1-\d_0)^k] + O((\d_0 k+1) 2^{-k}).\label{psi-simple}
\end{eqnarray}

Observe that $\d_0 \ge \d$. We divide the rest of the proof into five cases depending on where $\d_0$ lies in the interval $[2^{-k/2},1/2]$. 

{\bf Case 1}. Suppose $2^{-k/2} \leq \d_0 \leq \frac{1}{2k}$. We claim that $\psi_0(\d) <0$. Note $-\log(\d_0) \le (k/2)\log(2)$ and $-(1-\d_0)\log(1-\d_0) = \d_0 + O(\d_0^2)$. So
\begin{eqnarray*}
H(\d_0,1-\d_0) &=& -\d_0 \log \d_0 - (1-\d_0)\log(1-\d_0) \\
&\le& \d_0(k/2)\log(2) + \d_0 + O(\d_0^{2}).
\end{eqnarray*}
By Taylor series expansion,
$$1-(1-\d_0)^k \ge k \d_0 - {k \choose 2}\d_0^2 \ge 3k\d_0/4.$$
So by (\ref{psi-simple}) 
\begin{eqnarray*}
\psi_0(\d) &\le& \d_0(k/2)\log(2) + \d_0 - 3k\d_0\log(2)/4+ O(\d_0^{2}) \\
&=& \d_0[1-k\log(2)/4] + O(\d_0^{2}).
\end{eqnarray*}
Thus $\psi_0(\d)<0$ if $k$ is sufficiently large.

{\bf Case 2}. Let $0<\xi_0<1/2$ be a constant such that  $H(\xi_0,1-\xi_0) < \log(2)(1-e^{-1/2})$. Suppose  $\frac{1}{2k} \leq \d_0 \leq \xi_0$. We claim that $\psi_0(\d)<0$ if $k$ is sufficiently large.

By monotonicity, $H(\d_0,1-\d_0) \le H(\xi_0,1-\xi_0)$. Since $1-x \le e^{-x}$ (for $x>0$),
$$[1-(1-\d_0)^k] \ge 1- e^{-k\d_0} \ge 1-e^{-1/2}.$$
By (\ref{psi-simple}),
\begin{eqnarray*}
\psi_0(\d) &\le& H(\xi_0,1-\xi_0) - \log(2)(1-e^{-1/2}) + O( k2^{-k}).
\end{eqnarray*}
This implies the claim. 

{\bf Case 3}. Let $\xi_1$ be a constant such that $\max(\xi_0,1/3)<\xi_1<1/2$. Suppose $\xi_0 \leq \d_0 \leq \xi_1$. We claim that $\psi_0(\d)<0$ for all sufficiently large $k$ (depending on $\xi_1$).

By (\ref{psi-simple}),
\begin{eqnarray*}
\psi_0(\d) &\le& H(\xi_1,1-\xi_1) - \log(2)[1-(1-\d_0)^k]  + O( k2^{-k}) \le  H(\xi_1,1-\xi_1) - \log(2) + O( (1-\xi_0)^k). 
\end{eqnarray*}
This proves the claim.


{\bf Case 4}. We claim that if $\xi_1 \le \d_0 \le 0.5-2^{-k}$ then $\psi_0(\d)<f(d,k)$ for all sufficiently large $k$ (independent of the choice of $\xi_1$). 

Recall that we define $\varepsilon$ by $\d = \d_0(1-\varepsilon)$. By (\ref{vareps}), 
\begin{eqnarray*}
\varepsilon &=& 2^{1-k}\cdot (1-\d_0) \cdot \frac{(1-\d_0)^{k-1} - \d_0^{k-1} }{1 - 2^{2-k} + 2(\d_0/2)^{k} + 2((1-\d_0)/2)^k} \\
&\le & 2^{1-k}(1-\xi_1)^{k} + O\left(4^{-k}\right) \le 2\cdot 3^{-k}
\end{eqnarray*}
since $\xi_1 >1/3$, assuming $k$ is sufficiently large. 

The assumption on $r$ implies $d=O\left(k2^k\right)$. So the second equality of Lemma \ref{lem:4things} implies
$$(d-1)[H_0(\d,1-\d) - H(\d,1-\d)] = O\left(k 4.5^{-k}\right).$$ Through elementary but messy calculus computations one may show using the fact that $r = O(2^k)$ that
\begin{align*}
\psi'(1/2) &= 0 \\
\psi''(1/2) &= -4 + O( k^2 2^{-k}) \\
\psi'''(x) &= \frac{1}{x^2} - \frac{1}{(1-x)^2} + O(k^3 (2/3)^k ) \text{ and for any } x \in [1/3,2/3].
\end{align*}
By Taylor's theorem,
\begin{eqnarray}\label{estimateCOZ}
\psi(\d_0)= \psi(1/2) - (2+o_k(1) )(1/2-\d_0)^2 + \frac{1}{6} \pr{\frac{1}{\upsilon^2} - \frac{1}{(1-\upsilon)^2} + o_k(1) } (1/2 - \d_0)^3.
\end{eqnarray}
for some $\d_0 \leq \upsilon \leq 1/2$. Since $\d_0 \geq 1/3$, we have $1/2 - \d_0 \leq 1/6$. Furthermore, since $x \mapsto \frac{1}{x^2} - \frac{1}{(1-x)^2}$ is monotone decreasing on $(0,1/2]$, we have and $\frac{1}{\upsilon^2} - \frac{1}{(1-\upsilon)^2} \leq \frac{1}{(1/3)^2} - \frac{1}{(2/3)^2} = \frac{27}{4}$. Thus
\begin{align*}
\frac{ \frac{1}{6} \pr{ \frac{1}{\upsilon^2} - \frac{1}{(1-\upsilon)^2}} (1/2 - \d_0)^3}{ 2 (1/2-\d_0)^2} & \leq \frac{1}{12} \cdot \frac{27}{4} \cdot \frac{1}{6} < 1
\end{align*}
and for sufficiently large $k$ we have $\frac{1}{6} \pr{\frac{1}{\upsilon^2} - \frac{1}{(1-\upsilon)^2} + o_k(1) } (1/2 - \d_0)^3 < (2+o_k(1)) (1/2 - \d_0)^2$.
Since $\psi(1/2)=f(d,k)$, (\ref{estimateb}) implies
$$\psi_0(\d) \le f(d,k) - (2+o_k(1))(1/2-\d_0)^2 + \frac{1}{6} \pr{\frac{1}{\upsilon^2} - \frac{1}{(1-\upsilon)^2} + o_k(1)} (1/2 - \d_0)^3 + O\left(k 4.5^{-k}\right)$$
is strictly less than $f(d,k)$ if $k$ is sufficiently large.

{\bf Case 5}.  Suppose $0.5-2^{-k} \le \d_0 < 0.5$.  Let $\g = 0.5 - \d_0$. By (\ref{vareps}), 
$$\varepsilon =   O\left( \left[(1/2 +\g)^{k-1} -(1/2 -\g)^{k-1}\right]2^{-k}\right).$$
Define $L(x):= (1/2 +x)^{k-1} -(1/2 -x)^{k-1}$. We claim that $L(\g) \le \g$. Since $L(0)=0$, it suffices to show that $L'(x)\le 1$ for all $x$ with $|x|\le 0.01$. An elementary calculation shows
$$L'(x) = (k-1)\left[ (1/2 +x)^{k-2} +(1/2 -x)^{k-2}\right].$$
So $L'(x) \le 1$ if $|x|\le 0.01$ and $k$ is sufficiently large. Altogether this proves  $\varepsilon  = O\left(\g 2^{-k}\right)$. So the second equality of Lemma \ref{lem:4things} implies
$$(d-1)\left[H_0(\d,1-\d) - H(\d,1-\d)\right] = O\left(k2^{-k} \g^2\right).$$
By (\ref{estimateCOZ}) and  (\ref{estimateb}),
$$\psi_0(\d) \le f(d,k) - (2+o_k(1))\g^2 + O\left(k2^{-k} \g^2\right).$$
This is strictly less than $f(d,k)$ if $k$ is sufficiently large. 

\end{proof}

\subsection{Reducing Lemma \ref{lem:key} inequality (\ref{eqn:3}) to estimating the local cluster}\label{sec:reduction2}

As in the previous section, fix an equitable coloring $\chi:V_n \to \{0,1\}$. Given a uniform homomorphism $\s \in \Hom_{\rm{unif}}(\G,\Sym(n))$, the {\bf cluster around $\chi$} is the set 
$$\cC_{\s}(\chi) := \left\{\tchi \in Z_e(\s): d_n(\chi,\tchi) \leq 2^{-k/2}\right\}.$$
We also call this the {\bf local cluster} if $\chi$ is understood. 

In \S \ref{sec:strategy} we prove:
\begin{prop}\label{prop:cluster}
Let $0<\eta_0 <\eta_1< (1-\log 2)/2$. Then for all sufficiently large $k$ (depending on $\eta_0, \eta_1$), if 
$$r := d/k= \frac{\log(2)}{2} \cdot 2^k - (1+\log(2))/2 +\eta$$
for some $\eta$ with $\eta_0\le \eta \le \eta_1$ then with high probability in the planted model, $|\cC_\s(\chi)| \leq  \E^u_n(Z_e)$. In symbols,
$$\lim_{n\to\infty} \P^\chi_n\big(|\cC_\s(\chi)| \leq  \E^u_n(Z_e)\big)=1.$$
\end{prop}
The rest of this section proves Lemma \ref{lem:key} inequality (\ref{eqn:3}) from Proposition \ref{prop:cluster} and the second moment estimates from earlier in this section. So we assume the hypotheses of Proposition \ref{prop:cluster} without further mention.

We say that a coloring $\chi$ is {\bf $\s$-good} if it is equitable and $|\cC_\s(\chi)| \leq \E^u_n(Z_e(\s))$.  Let $S_g(\s)$ be the set of all $\s$-good proper colorings and let $Z_g(\s) = |S_g(\s)|$ be the number of $\s$-good proper colorings. 

We will say a positive function $G(n)$ is {\bf sub-exponential in $n$} if  $\lim_{n \to \infty} n^{-1} \log G(n) = 0$.  Also we say functions $G$ and $H$ are {\bf asymptotic}, denoted by $G(n) \sim H(n)$, if $\lim_{n \to \infty}G(n)/H(n) = 1$. Similarly, $G(n) \lesssim H(n)$ if $\limsup_{n \to \infty}G(n)/H(n) \le 1$.

\begin{lem}\label{lem:g to e}
 $\E^u_n(Z_g) \sim \E^u_n(Z_e) = F(n)\E^u_n(Z)$ where $F(n)$ is sub-exponential in $n$.  
\end{lem}

\begin{proof}
For brevity, let $\cH = \Hom_{\rm{unif}}(\G,\Sym(n))$. Let $\P^\chi_n$ be the probability operator in the planted model of $\chi$. By definition, 
\begin{eqnarray*}
\E^u_n(Z_g) &=& |\cH|^{-1}\sum_{\s \in \cH} Z_g(\s) = |\cH|^{-1}\sum_{\s\in \cH} \sum_{\chi: V\to \{0,1\}}  1_{S_g(\s)}(\chi) \\
&=& \sum_{\chi}\P^u_n(\chi \in S_g(\s)) \\
&=& \sum_{\chi \ \text{equitable}}\P^u_n\left(|\cC_\s(\chi)| \leq \E^u_n(Z_e)|\chi \ \text{proper}\right)\P^u_n(\chi \ \text{proper}) \\
&=& \sum_{\chi \ \text{equitable}}\P^{\chi}_n\left(|\cC_\s(\chi)| \leq \E^u_n(Z_e))\P^u_n(\chi \ \text{proper}\right) \\
&\sim & \sum_{\chi \ \text{equitable}} \P^u_n\left(\chi \ \text{proper}\right)= \E^u_n(Z_e)
\end{eqnarray*}
where the asymptotic equality $\sim$ follows from Proposition \ref{prop:cluster}. The equality $\E^u_n(Z_e) = F(n)\E^u_n(Z)$ holds by Theorem \ref{thm:1moment}.
\end{proof}

\begin{lem}\label{lem:pz estimate}
$\E_n^u(Z_g^2) \leq C(n)\E_n^u(Z_g)^2$, where $C(n) = C(n,k,r)$ is sub-exponential in $n$.
\end{lem}
\begin{proof}
\begin{eqnarray}
\E^u_n(Z_g^2)  &=& |\cH|^{-1}\sum_{\s \in \cH} \left(\sum_\chi 1_{S_g(\s)}(\chi)\right)^2\nonumber\\
&=& |\cH|^{-1}\sum_{\s \in \cH} \sum_{\chi, \tchi}1_{S_g(\s)}(\chi)1_{S_g(\s)}(\tchi) \nonumber\\
&=& \sum_{\chi,\tchi}\P^u_n\left(\chi \in S_g \ \textrm{and } \tchi \in S_g\right) \nonumber\\
&=& \sum_{\chi,\tchi}\P^u_n(\chi \in S_g)\P^u_n(\tchi \in S_g | \chi \in S_g) \nonumber\\
&=& \sum_{\chi}\P^u_n(\chi \in S_g) \E^u_n(Z_g|\chi \in S_g). \label{plant me}
\end{eqnarray}
For a fixed $\chi \in S_g(\s)$ we analyze $\E^u_n(Z_g|\chi \in S_g)$ by breaking the colorings into those that are close (i.e. in the local cluster) and those that are far. So let $Z_g(\d): \Hom_{\chi}(\G,\sym(n)) \to \N$ be the number of good proper colorings such that $d_n(\chi,\tchi) = \d$. (We will also use $Z_e(\d) = Z_\chi(\d)$ for the analogous number of equitable proper colorings). Then 
\begin{eqnarray}
\E^u_n(Z_g|\chi \in S_g) \leq 2\E^u_n\left(\sum_{0\leq \d \leq 2^{-k/2}}Z_g(\d) \Big| \chi \in S_g\right) + 2\E^u_n\left(\sum_{2^{-k/2} < \d \leq 1/2}Z_g(\d)\Big| \chi \in S_g\right).\label{decompose}
\end{eqnarray}
The coefficient $2$ above accounts for the following symmetry: if $\tchi$ is a good coloring with $d_n(\chi,\tchi)=\d$ then $1-\tchi$ is a good coloring with $d_n(\chi,1-\tchi)=1-\d$. 
Note that 
\begin{eqnarray}
\E^u_n\left(\sum_{0\leq \d \leq 2^{-k/2}}Z_g(\d)\Big|\chi \in S_g\right)\leq \E^u_n\big(\#\cC_\s(\chi) |\chi \in S_g\big)  \le  \E^u_n(Z_e)\label{near}
\end{eqnarray}
where the last inequality holds by definition of $S_g$.

For colorings not in the local cluster, 
\begin{eqnarray*}
\E^u_n\left(\sum_{2^{-k/2} < \d \leq 1/2} Z_g(\d)\Big|\chi \in S_g\right) &\leq & \E^u_n\left(\sum_{2^{-k/2} < \d \leq 1/2}Z_e(\d)\Big|\chi \in S_g\right) \\
&\leq & \E^u_n\left(\sum_{2^{-k/2} < \d \leq 1/2}Z_e(\d)\Big|\chi \ \text{proper}\right) \frac{\P^u_n(\chi \ \text{proper})}{\P^u_n(\chi \in S_g)}
\end{eqnarray*}
where the sum is over all $\d \in \Z[1/n]$ in the given range. By definition and  Proposition \ref{prop:cluster},
$$\frac{\P^u_n(\chi \ \text{proper})}{\P^u_n(\chi \in S_g)} = \frac{1}{\P^u_n(\chi \in S_g |\chi \ \text{proper} )} = \frac{1}{\P^\chi_n\big(|\cC_\s(\chi)| \leq  \E^u_n(Z_e)\big)}\to 1$$
as $n\to\infty$. Since $\E^u_n(\cdot| \chi \ \text{proper})= \E^\chi_n(\cdot)$, the above inequality now implies
\begin{eqnarray}
\E^u_n\left(\sum_{2^{-k/2} < \d \leq 1/2} Z_g(\d)\Big|\chi \in S_g\right) &\lesssim & \sum_{2^{-k/2} < \d \leq 1/2} \E^{\chi}_n(Z_e(\d)) \leq  C_1\sum_{2^{-k/2}< \d \leq 1/2} e^{n\psi_0(\d)}\nonumber \\
&\le& C_1 n e^{nf(d,k)}\le C_2 \E^u_n(Z_e)\label{far}
\end{eqnarray}
where the second inequality holds by Theorem \ref{thm:cluster} for some function $C_1 = C_1(n,k,r)$ which is sub-exponential in $n$. The second-to-last inequality holds because the number of summands is bounded by $n$ since $\d$ is constrained to lie in $\Z[1/n]$ and by Proposition \ref{prop:5cases}, $\psi_0(\d) \le f(d,k)$. The last inequality holds for some function $C_2=C_2(n,k,r)$ that is sub-exponential in $n$ since by Theorem \ref{thm:1moment}, $n^{-1} \log \E^u_n(Z_e)$ converges to $f(d,k)$. 

Combine (\ref{decompose}), (\ref{near}) and (\ref{far}) to obtain
$$\E^u_n(Z_g|\chi \in S_g) \le 2(1+C_2) \E^u_n(Z_e).$$
Plug this into (\ref{plant me}) to obtain
\begin{eqnarray*}
\E^u_n(Z_g^2) \le 2(1+C_2) \E^u_n(Z_e)^2 \sim 2(1+C_2) \E^u_n(Z_g)^2
\end{eqnarray*}
where the asymptotic $\sim$ holds by Lemma \ref{lem:g to e}. This proves the lemma.

\end{proof}

\begin{cor}
Lemma \ref{lem:key} inequality (\ref{eqn:3}) is true. That is:
\begin{eqnarray*}
0&=& \inf_{\eps>0} \liminf_{n\to\infty} n^{-1} \log \left(\P_n^{u}\left(  \left| n^{-1}\log Z(\s) - f(d,k) \right| < \eps \right)\right). 
\end{eqnarray*}

\end{cor}

\begin{proof}
By Theorem \ref{thm:1moment},
$$n^{-1} \log \E_n^{u}(Z(\s)) \to f(d,k)$$
as $n\to\infty$. In particular, for every $\eps >0$, for large enough $n$, $\P^u_n\left(n^{-1}\log Z(\s) > f(d,k) + \eps\right) < \exp{\pr{-\frac{\epsilon}{2} n}}$. So it suffices to prove
\begin{eqnarray*}
0&=& \inf_{\eps>0} \liminf_{n\to\infty} n^{-1} \log \left(\P_n^{u}\left(  n^{-1}\log Z(\s) \ge f(d,k) - \eps \right)\right). 
\end{eqnarray*}
Since $Z(\s) \ge Z_g(\s)$, it suffices to prove the same statement with $Z_g(\s)$ in place of $Z(\s)$. By Lemma \ref{lem:g to e} and Theorem \ref{thm:1moment}, $n^{-1}\log\left(\E_n^{u}[Z_g(\s)]\right)$ converges to $f(d,k)$ as $n\to\infty$. So we may replace $f(d,k)$ in the statement above with $n^{-1}\log\left(\E_n^{u}[Z_g(\s)]\right)$. Then we may multiply by $n$ both sides and exponentiate inside the probability. So it suffices to prove
\begin{eqnarray}
0&=& \inf_{\eps>0} \liminf_{n\to\infty} n^{-1} \log \left(\P_n^{u}\left(  Z_g(\s) \ge \E_n^{u}[Z_g(\s)]e^{-n\eps} \right)\right).\label{last} 
\end{eqnarray}
By the Paley-Zygmund inequality and Lemma \ref{lem:pz estimate}
$$\P_n^{u}\left(Z_g(\s) > \E_n^{u}[Z_g(\s)]e^{-n\eps}\right) \ge (1-e^{-n\eps})^2 \frac{\E_n^{u}[Z_g(\s)]^2}{\E_n^{u}[Z_g(\s)^2]} \ge \frac{1}{C}$$
where $C=C(n)$ is sub-exponential in $n$. This implies (\ref{last}).
\end{proof}

\section{The local cluster}\label{sec:strategy}

To prove Proposition \ref{prop:cluster}, we show that with high probability in the planted model, there is a `rigid' set of vertices with density approximately $1-2^{-k}$. Rigidity here means that any proper coloring either mostly agrees with the planted coloring on the rigid set or it must disagree on a large density subset. Before making these notions precise, we introduce the various subsets, state precise lemmas about them and prove Proposition \ref{prop:cluster} from these lemmas which are proven in the next two sections.

So suppose $G=(V,E)$ is a $k$-uniform $d$-regular hyper-graph and $\chi:V \to \{0,1\}$ is a proper coloring. An edge $e \in E$ is {\bf $\chi$-critical} if there is a vertex $v \in e$ such that $\chi(v) \notin \chi(e\setminus \{v\})$. If this is the case, then we say {\bf $v$ supports $e$ with respect to $\chi$}. If $\chi$ is understood then we will omit mention of it. We will apply these notions both to the case when $G$ is the Cayley hyper-tree of $\G$ and when $G=G_\s$ is a finite hyper-graph. 

For $l \in \{0,1,2,\ldots\}$, define the {\bf depth $l$-core of $\chi$} to be the subset $C_l(\chi) \subset V$ satisfying
$$C_0(\chi) = V,$$
$$C_{l+1}(\chi) = \{v \in V:~ v \textrm{ supports at least 3 edges } e \textrm{ such that } e \setminus \{v\} \subset C_l(\chi) \}.$$
By induction, $C_{l+1}(\chi) \subset C_l(\chi)$ for all $l\ge 0$. Also let $C_\infty(\chi) = \cap_l C_l(\chi)$.

The set $C_{l}(\chi)$ is defined to consist of vertices $v$ so that if $v$ is re-colored (in some proper coloring) then this re-coloring forces a sequence of re-colorings in the shape of an immersed hyper-tree of degree at least 3 and depth $l$.  Re-coloring a vertex of $C_{\infty}(\chi)$ would force re-coloring an infinite immersed tree of degree at least 3. 

Also define the {\bf attached vertices} $A_l(\chi) \subset V$ by: $v \in A_l(\chi)$ if $v \notin C_l(\chi)$ but there exists an edge $e$, supported by $v$ such that $e \setminus \{v\} \subset C_{l-1}(\chi)$. Thus if $v \in A_l(\chi)$ is re-colored then it forces a re-coloring of some vertex in $C_{l-1}(\chi)$. In this definition, we allow $l=\infty$ (letting $\infty - 1 = \infty)$. 

In order to avoid over-counting, we also need to define the subset $A'_l(\chi)$ of vertices $v \in A_l(\chi)$ such that there exists a vertex $w \in A_l(\chi)$, with $w \ne v$, and edges $e_v$, $e_w$ supported by $v, w$ respectively such that 
\begin{enumerate}
\item $e_v \setminus \{v\}, e_w \setminus \{w\} \subset C_{l-1}(\chi)$,
\item $e_v \cap e_w \ne \emptyset$. 
\end{enumerate}
In this definition, we allow $l=\infty$. 

We will need the following constants:
$$\l_0=\frac{ 1 }{ 2^{k-1}-1}, \quad \l := d\l_0 =\frac{ d }{ 2^{k-1}-1}.$$
 The significance of $\l_0$ is: if $e$ is an edge and $v \in e$ a vertex then $\l_0$ is the probability $v$ supports $e$ in a uniformly random proper coloring of $e$. So $\l=d\l_0$ is the expected number of edges that $v$ supports. For the values of $d$ and $k$ used in the Key Lemma \ref{lem:key}, $\l=\log(2)k + O(k2^{-k})$.

For the next two lemmas, we assume the hypotheses of Proposition \ref{prop:cluster}. 

\begin{lem}\label{lem:density}
For any $\d>0$ there exists $k_0$ such that $k \ge k_0$ implies
$$\lim_{l \to \infty} \liminf_{n\to\infty} \P^\chi_n\left(\frac{|C_l(\chi) \cup A_l(\chi) \setminus A'_l(\chi)|}{n} > 1-e^{-\l}(1+\delta) \right) = 1.$$
\end{lem}
Lemma \ref{lem:density} is proven in \S \ref{sec:Markov}.

\begin{defn}
Fix a proper 2-coloring $\chi:V \to \{0,1\}$. Let $\rho>0$. A subset $R \subset V$ is {\bf $\rho$-rigid}  (with respect to $\chi$) if for every proper coloring $\chi': V \to \{0,1\}$, $|\{v\in R: \chi(v) \neq \chi'(v)\}|$ is either less than $\rho |V|$ or greater than $2^{-k/2}|V|$.
\end{defn}

\begin{lem}\label{lem:rigidity}
For any $\rho>0$,
$$\lim_{l \to \infty} \liminf_{n\to\infty} \P^\chi_n\left(C_l(\chi) \cup A_l(\chi) \setminus A'_l(\chi) \textrm{ is $\rho$-rigid} \right) = 1.$$
\end{lem}  

Lemma \ref{lem:rigidity} is proven in \S \ref{sec:rigidity}. We can now prove Proposition \ref{prop:cluster}:

\noindent {\bf Proposition \ref{prop:cluster}}. {\it Let $0<\eta_0 < \eta_1 < (1-\log 2)/2$. Then for all sufficiently large $k$ (depending on $\eta_0, \eta_1$), if 
$$r := d/k= \frac{\log(2)}{2} \cdot 2^k - (1+\log(2))/2 +\eta$$
for some $\eta$ with $\eta_0\le \eta \le \eta_1$ then with high probability in the planted model, $|\cC_\s(\chi)| \leq  \E^u_n(Z_e)$. In symbols,
$$\lim_{n\to\infty} \P^\chi_n\big(|\cC_\s(\chi)| \leq  \E^u_n(Z_e)\big)=1.$$}

\begin{proof}
Let $0<\rho,\delta$ be small constants satisfying
\begin{eqnarray}\label{small stuff}
\log(2)\d +H(\rho,1-\rho) + \log(2)\rho<(1-2\eta - \log(2))2^{-k}.
\end{eqnarray}
Let $l$ be a natural number. Also let $\s:\G \to \sym(n)$ be a uniform homomorphism and $\chi:[n] \to \{0,1\}$ a proper coloring. To simplify notation, let 
$$R= C_l(\chi) \cup A_l(\chi) \setminus A'_l(\chi).$$
By Lemmas \ref{lem:density} and \ref{lem:rigidity} it suffices to show that if 
$ |R|/n > 1-e^{-\l} - \d$ and $R$  is $\rho$-rigid then $|\cC_\s(\chi)| \leq  \E^u_n(Z_e)$  (for all sufficiently large $n$). So assume $ |R|/n > 1-e^{-\l} - \d$ and $R$  is $\rho$-rigid.

Let $\chi' \in \cC_\s(\chi)$. By definition, this means $d_n(\chi',\chi)\le 2^{-k/2}$. Since $R$ is $\rho$-rigid, this implies
$$|\{v\in R: \chi(v) \neq \chi'(v)\}| \le \rho n.$$
Since this holds for all $\chi' \in \cC_\s(\chi)$, it follows that
$$|\cC_\s(\chi)| \le {|R| \choose \rho n} 2^{\rho n} 2^{n - |R|}.$$
By Stirling's formula 
$$n^{-1} \log {|R| \choose \rho n} \le n^{-1} \log {n \choose \rho n} \le H(\rho,1-\rho) + O(n^{-1}\log(n)).$$
Since $|R|/n > 1-e^{-\l} - \d = 1-2^{-k} -\d + O(k2^{-2k})$, 
$$n^{-1}\log(2^{n-|R|}) \le \log(2)[2^{-k}+\d] + O(k2^{-2k}).$$
Thus,
$$n^{-1}\log |\cC_\s(\chi)| \le \log(2)2^{-k}+\log(2)\d +H(\rho,1-\rho) + \log(2)\rho + O(k2^{-2k} + n^{-1}\log(n)).$$
On the other hand,
$$n^{-1}\log \E^u_n(Z_e) = f(d,k) + O(n^{-1}\log(n))= (1-2\eta)2^{-k} + O(2^{-2k})+O(n^{-1}\log(n))$$
by Lemma \ref{lem:psi} and Theorem \ref{thm:1moment}. 

Therefore, the choice of $\rho,\delta$ in (\ref{small stuff}) implies $|\cC_\s(\chi)| \le \E^u_n(Z_e)$ for all sufficiently large $n$. This also depends on $k$ being sufficiently large, but the lower bound on $k$ is uniform in $n$. 

\end{proof}

\section{A Markov process on the Cayley hyper-tree}\label{sec:Markov}

Here we study a $\G$-invariant measure $\mu$ on $X$ which is, in a sense, the limit of the planted model. We will define it by specifying its values on cylinder sets. 

Let $D \subset \G$ be a connected finite union of hyper-edges.  Let $\xi \in \{0,1\}^{D}$ be a proper coloring. We define $C(\xi)\subset X$ to be the set of proper colorings $\xi' \in X$ with $\xi' \resto D = \xi$ (where $\resto$ means ``restricted to''). We set $\mu(C(\xi))$ equal to the reciprocal of the number of proper colorings of $D$. In particular, if $\xi' \in \{0,1\}^D$ is another proper coloring of $D$, then $\mu(C(\xi))=\mu(C(\xi'))$.

Because $X$ is totally disconnected, any Borel probability measure on $X$ is determined by its values on clopen subsets (since these generate the topology). Since every clopen subset is a finite union of cylinder sets of the form above, 
Kolmogorov's Extension Theorem implies the existence of a unique Borel probability $\mu$ on $X$ satisfying the aforementioned equalities. 

Note this measure has the following Markov property. Let $x = (x_g)_{g \in G}$ be a random element of $X$ with law $\mu$. Let $v \in \G$ and let $e \subset \G$ be a hyperedge containing $v$. Let $\past(e,v)$ be the set of all $g \in \G$ such that every path in the Cayley hyper-tree from $g$ to an element of $e$ passes through $v$. In particular, $e \cap \past(e,v) = \{v\}$.  Then the distribution of $(x_g)_{g \in e \setminus \{v\}}$ conditioned on $\{x_g:~g \in \past(e,v)\}$ is uniformly distributed on the set of all colorings $y:(e\setminus \{v\}) \to \{0,1\}$ such that there exists some $h \in e \setminus \{v\}$ with $y(h) \ne x(v).$

\subsection{Local convergence}

We will prove the following lemma. 

\begin{lem}\label{lem:bsconc}
Let $\chi : V \to \{0,1\}$ be an equitable coloring with $|V|=n$. If $B \subseteq X$ is clopen, then for every $\epsilon > 0$
\begin{align*}
\lim_{n \to \infty} \P^\chi_n \pr{\abs{\frac{1}{n} \sum_{v \in V} \indic{B}(\Pi_v^{\sigma_n}(\chi)) - \mu(B) } > \epsilon } = 0.
\end{align*}
\end{lem}

To prove this lemma we will first show that if $f_n: \Hom_{\chi}(\G,\sym(n)) \to \R$ is the function
 $$f_n(\sigma_n)  : = \frac{1}{n} \sum_{v\in V} \indic{B}(\Pi_v^{\sigma_n}(\chi))$$ 
 then $f_n$ concentrates about its expectation using Theorem \ref{thm:planted-concentration}, and then we will show that this expectation is given by $\mu(B)$.

\begin{prop}\label{prop:bsconc}
We have
\begin{align*}
\lim_{n \to \infty} \P^\chi_n\pr{|f_n - \E^\chi_n[f_n]| > \epsilon } = 0.
\end{align*}
\end{prop}
\begin{proof}
For $g \in \G$, let $\proj_g:X \to \{0,1\}$ be the projection map $\proj_g(x)=x_g$. For $D \subset \G$, let $\cF_D$ be the smallest Borel sigma-algebra such that $\proj_g$ is $\cF_D$-measurable for every $g \in D$. 

Since every clopen subset $B$ of $X$ is a finite union of cylinder sets, the function $\indic{B}$ is $\cF_D$-measurable for some finite set $D \subset \Gamma$. 

We will use the normalized Hamming metrics $d_{\Sym(n)}$ and $d_{\Hom}$ on $\Sym(n)$ and $\Hom_{\chi}(\G,\sym(n))$ respectively. These are defined in the beginning of Appendix \ref{sec:concentration}. We claim $f_n$ is $L$-Lipschitz for some $L < \infty$. Let $\s, \s' \in\Hom_{\chi}(\G,\sym(n))$. Because $\indic{B}$ is $\cF_D$-measurable,
\begin{align*}
|f_n(\sigma) - f_n(\sigma')| &\leq n^{-1} \#\{v \in [n] : \chi(\sigma(\gamma^{-1})(v)) \neq \chi(\sigma'(\gamma^{-1})(v)) \text{ for some } \gamma \in D\} \\
& \leq n^{-1} \# \{ v \in [n] : \sigma(\gamma^{-1})(v) \neq \sigma'(\gamma^{-1})(v) \text{ for some } \gamma \in D\} \\
& \leq  \sum_{\gamma \in D} d_{\Sym(n)}(\sigma(\gamma^{-1}),\sigma'(\gamma^{-1})).
\end{align*}
Now $d_{\Sym(n)}$ is both left and right invariant. So
$$d_{\Sym(n)}(gh,g'h') \leq d_{\Sym(n)}(gh,gh') + d_{\Sym(n)}(gh',g'h') = d_{\Sym(n)}(h,h') + d_{\Sym(n)}(g,g')$$
for any $g,g',h,h' \in \Sym(n)$. Furthermore we immediately have for any $1 \leq i \leq d$ that $d_{\Sym(n)}(\s(s_i),\s'(s_i)) \leq d_{\Hom}(\s,\s')$. Together these imply $d_{\Sym(n)}(\s(\g),\s'(\g)) \le |\g| d_{\Hom}(\s,\s')$ for any $\g \in \G$ where $|\gamma|$ is the distance from $\gamma$ to the identity in the word metric on $\Gamma$. Thus if we take $L = \sum_{\gamma \in D} |\gamma| < \infty$ we see that $|f_n(\sigma) - f_n(\sigma')|\le L d_{\Hom}(\s,\s')$ as desired.

The Proposition now follows from Theorem \ref{thm:planted-concentration}. 
\end{proof}

To finish the proof of Lemma \ref{lem:bsconc}, it now suffices to show the expectation of $f_n$ with respect to the planted model converges to $\mu(B)$ as $n\to\infty$. We will prove this by an inductive argument, the inductive step of which is covered in the next lemma. 

\begin{lem}\label{lem: approx}
Let $D \subset \G$ be either a singleton or a connected finite union of hyperedges containing the identity. Let $\xi \in \{0,1\}^{D}$ be a proper coloring. Let $F_{D,\xi,v}\subset \Hom_{\chi}(\G,\sym(n))$ be the event that $\Pi_{v}^{\sigma_n}(\chi) \resto D = \xi \resto D$. Let $Q(D)$ be the number of proper colorings of $D$. Then
\begin{align}\label{e:approx}
\lim_{n \to \infty} \sup_{v \in \chi^{-1}(\xi(1_\G))} \abs{\P^\chi_n\pr{F_{D,\xi,v} } - \frac{2}{Q(D)} } = 0.   
\end{align}

\end{lem}

\begin{remark}
The factor of $2$ in $\frac{2}{Q(D)}$ is there to account for the fact that we are requiring $v \in \chi^{-1}(\xi(1_\G))$. 
\end{remark}

\begin{proof}

The statement is immediate if $D=\{1_\G\}$ is a singleton. For induction we assume that the statement is true for some finite $D \subset \G$. Let $D' \supset D$ be a connected union of hyper-edges such that there exists a unique hyper-edge $e$ with $D' = D \cup e$. By induction, it suffices to prove the statement for $D'$. 

Note by symmetry that $\P^\chi_n(F_{D,\xi,v})$ is the same for all $v \in \chi^{-1}(\xi(1_\G))$. Let us denote this common value by $c_{D,\xi}$. Define $A_{D,\xi} = \{ (v,\sigma) : v \in V, \sigma \in F_{D,\xi,v}\}$. If $g \in D$ then $(v,\sigma) \mapsto (\sigma(g^{-1}) v,\sigma)$ is a bijection from $A_{D,\xi}$ to $A_{g^{-1}D,g^{-1}\xi}$. This implies
\begin{align*}
\frac{n}{2} c_{D,\xi} \abs{ \Hom_\chi(\G,\Sym(n))} &= \abs{A_{D,\xi}} \\
& = \abs{A_{g^{-1}D,g^{-1}\xi}} \\
& = \frac{n}{2} c_{g^{-1} D, g^{-1} \xi} \abs{\Hom_\chi(\G,\Sym(n))}.
\end{align*}
Thus $c_{g^{-1} D, g^{-1} \xi} = c_{D,\xi}$. After replacing $D'$ with $\g^{-1} D'$ where $\g \in e \cap D$ is the unique element in $e \cap D$, we may therefore assume without loss of generality that $e \cap D = \{1_\G\}$. By symmetry we may also assume $\xi(1_\G)=1$ and $e = \{s_1^i:~i=0,\cdots, k-1\}$ is the hyper-edge generated by $s_1$.

For any subset $W \subset \G$, let $E_{{W, v}}$ be the event that $\sigma_n(g^{-1})({v}) \neq \sigma_n(g'^{-1})({v})$ for any $g, g' \in W$ with $g \neq g'$. Since the planted model is a random sofic approximation, by Lemma \ref{lem:super-exponential} we have $\P^\chi_n(E_{D \cup e, v}) = 1 - o_n(1)$ uniformly in $v$.  We will show that  
\begin{align}\label{noloops}
\lim_{n \to \infty} \sup_{v \in \chi^{-1}(1)} \abs{\P^\chi_n\pr{F_{e,\xi,v} \bigg | F_{D,\xi,v}, E_{D \cup e, v} } - \frac{1}{2^{k-1}-1}} = 0.
\end{align}

Next we show how to finish the lemma assuming (\ref{noloops}). We claim that $Q(D\cup e)=Q(D)(2^{k-1}-1)$. To see this, observe that if $\xi'$ is a proper coloring of $D$ and $\xi''$ is any coloring of $e \setminus \{1_\G\}$ then the concatenation $\xi' \sqcup \xi''$ is a proper coloring of $D \cup e$ unless $\xi''(g)=\xi'(1_\G)$ for all $g \in e \setminus \{1_\G\}$. Since $|e \setminus \{1_\G\}|=k-1$, this implies that every proper coloring of $D$ admits $2^{k-1}-1$ proper extensions to $D \cup e$. So $Q(D\cup e)=Q(D)(2^{k-1}-1)$ as claimed.

Since $F_{D\cup e,\xi,v} = F_{e,\xi,v} \cap F_{D,\xi,v}$, 
\begin{eqnarray*}
\P^\chi_n\pr{F_{D\cup e,\xi,v} }  &=& \P^\chi_n\pr{F_{e,\xi,v} \bigg | F_{D,\xi,v} } \P^\chi_n\pr{F_{D,\xi,v} } \\
&=& \P^\chi_n\pr{F_{e,\xi,v} \bigg | F_{D,\xi,v}, E_{D \cup e, v} } \P^\chi_n\pr{F_{D,\xi,v} } + o_n(1).
\end{eqnarray*}
The lemma now follows from the inductive hypothesis, (\ref{noloops}) and $Q(D \cup e) = Q(D)(2^{k-1}-1)$.

It now suffices to prove (\ref{noloops}).  Let $D_j$ denote the union of all subsets $\{g s_j^i : i = 0, \cdots, k-1\}$ over all $g \in D$ such that $\{g s_j^i : i = 0, \cdots, k-1\} \subseteq D$. That is, $D_j$ is the union of all hyperedges generated by $s_j$ which lie completely within $D$.  Because $e = \{s_1^i:~i=0,\cdots, k-1\}$, $e \cap D_1 = \emptyset$. Multiplication by $s_j$ on the right preserves $D_j$.

Fix $v \in V=\{1,\ldots, n\}$ and consider an injective function $h : D \to V$ such that $h(1_\Gamma) = v$ and $\chi\pr{h\pr{e'}} = \{0,1\}$ for all hyper-edges $e' \in D$. Let $N(\chi,h)$ be the set of $\s \in \Hom_{\chi}(\G,\sym(n))$ such that $\s(g^{-1})v = h(g)$ for all $g \in D$. By definition,
\begin{align*}
\P^\chi_n \pr{\sigma_n(g^{-1})(v)  = h(g)~\forall g \in D} & =\frac{|N(\chi,h)|}{|\Hom_{\chi}(\G,\sym(n))|}.
\end{align*}
For each $j \in \{1,\ldots, d\}$, let $N_j(\chi,h)$ be the set of permutations $\pi_j \in \Sym(n)$ such that
\begin{enumerate}
\item the orbit-partition of $\pi_j$ is a $k$-partition which is properly colored by $\chi$,
\item $h(gs_j) = \pi_j^{-1} h(g)$ for all $g \in D \cap Ds_j^{-1}$.  
\end{enumerate}
We claim that the map 
$$N(\chi,h) \mapsto N_1(\chi,h) \times \cdots \times N_d(\chi,h)$$
that sends $\s$ to $(\s(s_1),\ldots, \s(s_d))$ is a bijection.

The fact that it is injective is immediate since $\G$ is generated by $s_1,\ldots, s_d$. In order to prove that it is surjective, fix elements $\pi_j \in N_j(\chi,h)$ for all $j$. Define $\s\in \Hom_{\chi}(\G,\sym(n))$ by $\s(s_j) = \pi_j$. It suffices to check that $\s(g^{-1})(v)  = h(g)~\forall g \in D$. By induction on the number of edges in $D$, it suffices to assume that the equation  $\s(g^{-1})(v)  = h(g)$ holds for some $g \in D \cap Ds_j^{-1}$ and prove   $\s((gs_j)^{-1})(v)  = h(gs_j)$. This follows from:
\begin{eqnarray*}
\s((gs_j)^{-1})(v) =\s(s_j^{-1}) \s(g^{-1})(v)  =  \pi_j^{-1} h(g) = h(gs_j).
\end{eqnarray*}
This proves the claim. Therefore,
\begin{align}\label{E:restriction}
\P^\chi_n \pr{\sigma_n(\cdot)^{-1}(v) \resto D = h} & = \frac{ \prod_{j =1 }^d |N_j(\chi,h)|}{|\Hom_{\chi}(\G,\sym(n))|}.
\end{align}

Let $H_D = \{ h : D \to V : h \text{ is injective, } \chi(h(g)) = \xi(g)~\forall g\in D\text{, and } h(1_\Gamma) = v\}$. We have
\begin{align*}
\P^\chi_n(F_{D,\xi,v}, E_{D, v}) & = \sum_{h \in H_D} \P^\chi_n(\sigma_n(\cdot)^{-1}(v) \resto D = h).
\end{align*}
Similarly
\begin{align*}
\P^\chi_n(F_{D\cup e,\xi,v}, E_{D \cup e, v}) & = \sum_{h' \in H_{D \cup e}} \P^\chi_n(\sigma_n(\cdot)^{-1}(v) \resto D \cup e = h').
\end{align*}
Each $h \in H_D$ is the restriction to $D$ of 
\begin{align}\label{E:extension}
(m-1)! \binom{n/2 - \# \xi^{-1}(1) \cap D}{m-1}(k-m)! \binom{n/2 - \#\xi^{-1}(0) \cap D}{k-m}
\end{align}
distinct $h' \in H_{D\cup e}$ where $m=|\xi^{-1}(1) \cap e|$. We can express
\begin{align*}
\P^\chi_n(F_{e,\xi,v} | F_{D,\xi,v}, E_{D \cup e, v}) & = \frac{\P^\chi_n(F_{D \cup e, \xi,v}, E_{D \cup e, v})}{\P^\chi_n(F_{D,\xi,v}, E_{D \cup e, v})} \\
& = \frac{\P^\chi_n(F_{D \cup e, \xi,v}, E_{D \cup e, v})}{\P^\chi_n(F_{D,\xi,v}, E_{D, v})} +o_n(1)\\
& = \frac{\sum_{h' \in H_{D\cup e}} \P^\chi_n(\sigma_n(\cdot)^{-1}(v) \resto D \cup e = h')}{\sum_{h \in H_D} \P^\chi_n(\sigma_n(\cdot)^{-1}(v)\resto D = h)} +o_n(1) \\
& = \frac{\sum_{h' \in H_{D\cup e}} \prod_{j=1}^d |N_j(\chi,h')|}{\sum_{h \in H_D} \prod_{j=1}^d |N_j(\chi,h)|} + o_n(1)
\end{align*}
where the $o_n(1)$ term follows from the inductive hypothesis and the fact that $\P^\chi_n(E_{D,v}) \ge \P^\chi_n(E_{D\cup e,v}) = 1-o_n(1)$.  


If $h_1, h_2 \in H_D$ then $\chi(h_1(g)) = \xi(g) = \chi(h_2(g))$ for all $g\in D$. Observe that $|N_j(\chi,h_1)| = |N_j(\chi,h_2)|$. This follows by conjugating by a permutation in $\Sym(n,k)$ which preserves the color classes of $\chi$ and which maps $h_1(g)$ to $h_2(g)$ for each $g \in D$. Similarly, if $h'_1, h'_2 \in H_{D\cup e}$ then $|N_j(\chi,h'_1)| = |N_j(\chi,h'_2)|$. Therefore,
\begin{align*}
\P^\chi_n(F_{e,\xi,v} | F_{D,\xi,v}, E_{D \cup e, v}) 
&= \frac{|H_{D\cup e}| \prod_{j=1}^d |N_j(\chi,h')|}{| H_D | \prod_{j=1}^d |N_j(\chi,h)|} + o_n(1)
\end{align*}
for any $h\in H_D$ and $h' \in H_{D \cup e}$.

Next we choose $h' \in H_{D\cup e}$ and $h \in H_D$ so that $h'$ extends $h$. Note $|N_j(\chi,h')| = |N_j(\chi,h)|$ for all $j \ne 1$ because $D_j = (D \cup e)_j$. Therefore,
\begin{align}
&~ \quad \P^\chi_n(F_{e,\xi,v} | F_{D,\xi,v}, E_{D \cup e, v}) \nonumber \\
 &= \frac{|H_{D\cup e}| |N_1(\chi,h')|}{| H_D | |N_1(\chi,h)|} + o_n(1)
\nonumber \\
& = \frac{|N_1(\chi,h')|}{ |N_1(\chi,h)|} (m-1)! \binom{n/2 - \# \xi^{-1}(1) \cap D}{m-1}(k-m)! \binom{n/2 - \#\xi^{-1}(0) \cap D}{k-m}  + o_n(1)
\nonumber \\
& \sim \frac{|N_1(\chi,h')|}{ |N_1(\chi,h)|} (n/2)^{k-1} \label{E:P}
\end{align}
where the second equality uses (\ref{E:extension}). Here we are using the notation $f(n) \sim g(n)$ to mean $\lim_{n\to\infty} \frac{f(n)}{g(n)}=1$.  

To compute $|N_1(\chi,h)|$, let
$$T_{D,n} =\left\{ \vec{t} \in [0,1]^{k+1} : t_0 = t_k = 0, \sum_i t_i = \frac{1}{k} , \sum_i i t_i = \frac{\# \chi^{-1}(1) \setminus h(D_1)}{\# V \setminus h(D_1)}, (n - \# D_1) t_i \in \Z\right\}$$
be the set of possible types of orbit-partitions of permutations of $V \setminus h(D_1)$ that contribute to the count $|N_1(\chi,h)|$. To be precise, if $\vec{t} \in T_{D,n}$ then there is a $k$-partition of $V \setminus h(D_1)$ such that the number of parts $P$ of the partition with $|P \cap \chi^{-1}(1)|=i$ is $t_i (n- \#D_1)$.

If we fix $\vt \in T_{D,n}$ then the number of permutations whose orbit partition has type $\vt$ is
$$(k-1)!^{ (n-\#D_1)/k}\frac{(n/2-\#\xi^{-1}(1)\cap D_1)!(n/2-\#\xi^{-1}(0)\cap D_1)!}{ \prod_{j=0}^{k} j!^{t_j(n-\#D_1)}(k-j)!^{t_j(n-\#D_1)} (t_j(n-\#D_1))!}$$
where we have used (\ref{eqn:q-thing}). It follows that
\begin{align*}
|N_1(\chi,h)| &= (k-1)!^{(n-\#D_1)/k} \frac{(n/2 - \#\xi^{-1}(1) \cap D_1 )! (n/2 - \#\xi^{-1}(0) \cap D_1)!}{\pr{\frac{n - \#D_1}{k}}!} \times \sum_{\vec{t} \in T_{D,n}} S_{D,n}(\vt)
\end{align*}
where
\begin{align}\label{E:SDn}
S_{D,n}(\vt) &= \frac{((n-\#D_1)/k)!}{\prod_{i=1}^{k-1} i!^{t_i (n - \# D_1)} (k-i)!^{t_i (n - \# D_1)} (t_i (n - \# D_1))!}.
\end{align}
So
\begin{align*}
\frac{|N_1(\chi,h')|}{ |N_1(\chi,h)|} &\sim \frac{(n/k)}{(k-1)!(n/2)^{k} } \frac{\sum_{\vec{t} \in T_{D\cup e,n}} S_{D \cup e,n}(\vt)}{\sum_{\vec{t} \in T_{D,n}} S_{D,n}(\vt)}.
\end{align*}
We plug this into (\ref{E:P}) to obtain
\begin{align}\label{E:almost-there}
\lim_{n \to \infty} \P_n^\chi(F_{e,\xi,v} | F_{D,\xi,v}, E_{D \cup e, v}) & =
\lim_{n \to \infty} \frac{2}{k! } \frac{\sum_{\vec{t} \in T_{D\cup e,n}} S_{D \cup e,n}(\vt)}{\sum_{\vec{t} \in T_{D,n}} S_{D,n}(\vt)}.
\end{align}

Define $t^* \in [0,1]^{k+1}$ by $t^*_0 = t^*_k = 0$ and $t^*_i = \frac{1}{k(2^k-2)} \binom{k}{i}$ for $0<i<k$. We establish asymptotic estimates of the sum of $S_{D,n}(\vec{t})$ in Lemma \ref{lemma:smallball} and show that it suffices to sum over a small ball around $t^*$.

Recall $m=|\xi^{-1}(1) \cap e|$. Let $\tT_{D,n} = \{\vt \in T_{D,n}: t_m >0\}.$ We define $f_n : \tT_{D,n} \to T_{D \cup e, n}$ by
\begin{align*}
f_n(\vec{t})_j = \begin{cases}
 \frac{(n-\# D_1)t_j-1}{n-\#D_1 - k} \text{ if } j = m \\
 \frac{n - \# D_1}{ n - \#D_1 - k} t_j \text{ if } j \neq m
\end{cases}.
\end{align*}
Because $e \cap D_1 = \emptyset$, if $\vec{t} \in \tT_{D,n}$ then $f_n(\vec{t}) \in T_{D\cup e,n}$.

We claim that for any $\d>0$ there exists an $N:=N(\delta)$ and $\delta' > 0$ such that for $n > N$, $T_{D\cup e,n} \cap B_{\delta'}(t^*) \subseteq f_n(\tT_{D,n} \cap B_\delta(t^*))$. This follows from observing that $f_n$ is invertible and that $d(f_n(\vt),\vt) \leq \frac{k^2}{n-\#D_1-k}$. By Lemma \ref{lemma:smallball} we have
\begin{align}\label{E:smallball}
\lim_{n \to \infty} \frac{\sum_{\vec{t} \in T_{D \cup e, n}} S_{D \cup e, n}(\vec{t})}{\sum_{\vec{t} \in T_{D,n}} S_{D,n}(\vec{t})} = \lim_{n \to \infty} \frac{\sum_{\vec{t} \in \tilde{T}_{D,n} \cap B_\delta(t^*)} S_{D \cup e, n}(f_n(\vec{t}))}{ \sum_{\vec{t} \in T_{D,n} \cap B_\delta(t^*)} S_{D,n}(\vec{t}) }.
\end{align}
Furthermore, since $1 \leq m \leq k-1$ and $t^*_m > 0$, for $\delta > 0$ sufficiently small we have $\tilde{T}_{D,n} \cap B_\delta(t^*) = T_{D,n} \cap B_\delta(t^*)$.
Because $\#(D \cup e)_1 =\#D_1 +k$, the ratio $\frac{S_{D \cup e, n}(f_n(\vec{t}))}{S_{D,n}(\vec{t})}$ simplifies to
\begin{align*}
\frac{k}{n-\#D_1} m!(k-m)! t_m(n-\#D_1) = \frac{k (k!)}{\binom{k}{m}} t_m.
\end{align*}
In particular, at $t^*$ we obtain
\begin{align*}
\frac{S_{D \cup e, n}(f_n(t^*))}{S_{D,n}(t^*)} = \frac{k!}{2^{k} - 2}.
\end{align*}
Suppose $\vec{t} \in B_\delta(t^*)$. Then $|t_m - t_m^*| < \delta$. So
\begin{align*}
\abs{ \frac{S_{D \cup e, n}(f_n(\vec{t}))}{S_{D,n}(\vec{t})} - \frac{k!}{2^{k} - 2}}= \abs{ \frac{S_{D \cup e, n}(f_n(\vec{t}))}{S_{D,n}(\vec{t})} - \frac{S_{D \cup e, n}(f_n(t^*))}{S_{D,n}(t^*)}} = \frac{k (k!)}{\binom{k}{m}} |t_m-t^*_m|\le k! \delta.
\end{align*}
Since $\delta$ is arbitrary, this and (\ref{E:smallball}) imply
\begin{align*}
\lim_{n \to \infty} \frac{\sum_{\vec{t} \in T_{D \cup e, n}} S_{D \cup e, n}(\vec{t})}{\sum_{\vec{t} \in T_{D,n}} S_{D,n}(\vec{t})} = \frac{k!}{2^k - 2}.
\end{align*}
We plug this into (\ref{E:almost-there}) to obtain
\begin{align*}
\lim_{n \to \infty} \P_n^\chi(F_{e,\xi,v} | F_{D,\xi,v}, E_{D \cup e, v}) &  = \frac{1}{2^{k-1}-1}.
\end{align*}

\end{proof}

\begin{lem} \label{lemma:smallball}
For any $\delta > 0$ we have
\begin{align*}
\lim_{n \to \infty} \frac{\sum_{\vec{t} \in T_{D,n} \cap B_{\delta}(t^*)} S_{D,n}(\vec{t})}{\sum_{\vec{t} \in T_{D,n}} S_{D,n}(\vec{t})} = 1.
\end{align*}
\end{lem}
\begin{proof}
We use the following general consequence of Stirling's formula: let $l \in \N$, $\vt \in [0,1]^l$ with $\sum_i t_i = 1$.  Let $\vt^{(n)}$ be any sequence in $[0,1]^l$ such that $n\vt^{(n)} \in \N^l$ and $\vt^{(n)} \to \vt$ as  $n\to\infty$.  If $\vt^{(n)} = (t_{n,i})_{i=1}^l$ then $\lim_{n \to \infty} n^{-1} \log \binom{n}{nt_{n,1},\ldots,nt_{n,l}} = H(\vt)$.  

By (\ref{E:SDn})
$$S_{D,n}(\vec{t}) = \binom{(n-\#D_1)/k}{t_{n,1}(n-\#D_1),\ldots,t_{n,k-1}(n-\#D_1)} \left(\prod_{i=1}^{k-1} \binom{k}{i}^{t_i(n-\#D_1)} \right) k!^{-(n-\#D_1)/k}.$$
By setting $N=(n-\#D_1)/k$, we obtain
\begin{eqnarray*}
&&\lim_{n\to\infty} n^{-1} \log \binom{(n-\#D_1)/k}{t_{n,1}(n-\#D_1),\ldots,t_{n,k-1}(n-\#D_1)} \\
&=& 
k^{-1}\lim_{n\to\infty} N^{-1} \log \binom{N}{t_{n,1}kN,\ldots,t_{n,k-1}kN} \\
&=& k^{-1} H(k\vt) = - \sum_{i=0}^k t_i \log(kt_i) = H(\vt) -k^{-1} \log(k).
\end{eqnarray*} 

We thus have:
\begin{align*}
\limsup_{n \to \infty} \sup_{\vec{t} \in T_{D,n}} \abs{\frac{1}{n} \log S_{D,n}(\vec{t})  - \psi(\vt)}= 0
\end{align*}
where $\psi$ is given by
\begin{align*}
\psi(\vt) :=  -\frac{1}{k} \log(k(k)!) + H(\vec{t}) + \sum_{i=0}^{k} t_i  \log \binom{k}{i} .
\end{align*}
Note that $\psi$ does not depend on $n$ or $D$. Furthermore $\psi$ is continuous and strictly concave. By the method of Lagrange multipliers (see Section \ref{sec:equitable}) its maximum on $M:= \{x \in [0,1]^{k+1} : x_0 = x_k = 0, \sum_i x_i = \frac{1}{k} \}$ occurs at $t^*$. In particular given $\delta > 0$ there is some $0 < \eta < \delta$ such that
\begin{align*}
\inf_{x \in M \cap B_{\eta}(t^*)} \psi(x) - \sup_{x \in M \setminus B_{\delta}(t^*)} \psi(x) > 0.
\end{align*}
We claim that $B_\eta(t^*) \cap T_{D,n} \subseteq B_\eta(t^*) \cap M$ is non-empty for sufficiently large $n$. Since $\# T_{D,n}  \leq n^k$ this claim and standard arguments involving the exponential growth of $S_{D,n}(t)$ imply the lemma.

To prove the claim, we exhibit a member of $B_\eta(t^*) \cap T_{D,n}$ in a series of three steps. First, let $t^{(1)} \in \R^{k+1}$ be defined by $t^{(1)}_i = \frac{\floor{t^*_i(n-\#D_1)}}{n-\#D_1}$ for each $0 \leq i \leq k$. This satisfies the integrality condition, that is $(n-\#D_1) t_i^{(1)} \in \Z$ for all $i$. Furthermore $|t^* - t^{(1)}| \leq \sum_{i = 0}^{k} |t^*_i - t^{(1)}_i| \leq \frac{k-1}{n - \# D_1}$. 

Second, define $t^{(2)} \in \R^{k+1}$ by
\begin{align*}
t^{(2)}_i = \begin{cases}
t^{(1)}_1 + \pr{\frac{1}{k} - \sum_{j=1}^{k-1} t^{(1)}_j} \text{ if } i =1 \\
t^{(1)}_i \text{ otherwise}
 \end{cases}.
\end{align*}
Since $\frac{n-\#D_1}{k} \in \Z$ and $(n-\#D_1) t_i^{(1)} \in \Z$ for each $i$ we maintain the integrality condition, and $\sum_i t_i^{(2)} = \frac{1}{k}$. Furthermore
\begin{align*}
\abs{t^{(2)} - t^{(1)}} \leq \abs{\frac{1}{k} - \sum_{i=1}^{k-1} t_i^{(1)}} & = \abs{\sum_{i=0}^k (t_i^* - t_i^{(1)})}  \leq \sum_{i=0}^k \abs{t_i^* - t_i^{(1)}} \leq \frac{k-1}{n - \#D_1}.
\end{align*}

Finally let $\Delta = \frac{\#\chi^{-1}(1) \setminus h(D_1)}{n - \# D_1} - \sum_{i=0}^k i t^{(2)}_i$.  Define $t^{(3)} \in \R^{k+1}$ by
\begin{align*}
t_i^{(3)} = 
\begin{cases}
t_1^{(2)} - \Delta \text{ if } i = 1 \\
t_2^{(2)} + \Delta \text{ if } i = 2 \\
t_i^{(2)} \text{ otherwise}
\end{cases}.
\end{align*}
We claim $t^{(3)} \in B_\eta(t^*) \cap T_{D,n}$ for $n$ sufficiently large. First $t^{(3)}_0 = t^{(2)}_0 = t^{(1)}_0 = t^*_0 = 0$, with the same equalities holding for $t^{(3)}_k$. We have $t^{(3)}_i (n-\#D_1) \in \Z$ since $t^{(2)}_i(n-\#D_1) \in \Z$ for each $i$ and $\Delta (n-\#D_1) \in \Z$. We also have $\sum_i t^{(3)}_i = \sum_i t^{(2)}_i = \frac{1}{k}$. Furthermore $\sum_i i t_i^{(3)} = \sum_i i t_i^{(2)} + 2\Delta - \Delta = \frac{\chi^{-1}(1) \setminus h(D_1)}{n - \#D_1}$. Finally by repeated application of the triangle inequality
\begin{align*}
|t^* - t^{(3)}| \leq \frac{k-1}{n - \#D_1} + \frac{k-1}{n-\#D_1} +  2|\Delta|.
\end{align*}

We show that $|\Delta|$ is small for $n$ sufficiently large. This will not only imply that $t^{(3)} \in B_\eta(t^*)$, but also that $t^{(3)}_i \in [0,1]$ for each $i$.  Note that
$$
\frac{n/2 - \#D_1}{n} \leq \frac{\#\chi^{-1}(1) \setminus h(D_1)}{n - \# D_1} \leq \frac{n/2}{n - \#D_1} = \frac{(n-\#D_1)/2 + \#D_1/2}{n-\#D_1}.
$$
This implies $\abs{\frac{\#\chi^{-1}(1) \setminus h(D_1)}{n - \#D_1} - \frac{1}{2}} \leq \max\{\frac{\#D_1}{n}, \frac{\#D_1}{2(n-\#D_1)}\}$. Thus
\begin{align*}
\abs{\Delta} & \leq \abs{\frac{\#\chi^{-1}(1) \setminus h(D_1)}{n - \# D_1} - \frac{1}{2}} + \abs{\frac{1}{2} -  \sum_{i=0}^{k} i t^{(2)}_i} \\
& \leq  \max\left \{\frac{\#D_1}{n}, \frac{\#D_1}{2(n-\#D_1)}\right \} + \abs{\sum_{i=0}^k it_i^* - \sum_{i=0}^k i t_i^{(1)}} + \abs{\sum_{i=0}^k i t_i^{(1)} - \sum_{i=0}^k i t_i^{(2)}} \\
& \leq   \max\left \{\frac{\#D_1}{n}, \frac{\#D_1}{2(n-\#D_1)}\right \} + \frac{2(k-1)^2}{n-\#D_1}.
\end{align*}
As $k$ and $\# D_1$ are fixed, we have $|t^*-t^{(3)}| < \eta$ for large enough $n$. Since $t^*_i \in (0,1)$ for each $1 \leq i \leq k-1$ we have $|t^{(3)}_i - t^*_i| \leq |t^{(3)} - t^*|$. Thus for sufficiently large $n$ we will also have $t^{(3)}_i \in (0,1)$ for each $1 \leq i \leq k-1$. Therefore $t^{(3)} \in [0,1]^{k+1}$ and we have $t^{(3)} \in B_\eta(t^*) \cap T_{D,n}$ for $n$ sufficiently large.

\end{proof}

\begin{proof}[Proof of Lemma \ref{lem:bsconc}]

Let $\mu_n^\chi$ be the Borel probability measure on $X$ defined by 
$$\mu_n^\chi(B) = \E_n^\chi\pr{ \frac{1}{\# V} \sum_{v \in V} \indic{B}\pr{\Pi_v^{\sigma_n}(\chi)}}$$
for any Borel set $B \subset X$. By Proposition \ref{prop:bsconc}, it suffices to show that $\mu_n^\chi(B) \to \mu(B)$ as $n\to\infty$ for any clopen set $B \subset X$. Because clopen sets are finite unions of cylinder sets, it suffices to show  that if  $D \subset \Gamma$ is a finite subset and $\xi \in \{0,1\}^D$ then $\lim_{n \to \infty} \mu_n^{\chi}([\xi])  = \mu([\xi])$ where $[\xi]$ is the cylinder set $\{x \in X : x\resto D = \xi\}$.  We can further assume $D$ to be a connected finite union of hyperedges with $1_\Gamma \in D$. The proof now follows from Lemma \ref{lem: approx} since 
\begin{align*}
\mu_n^{\chi}([\xi]) &= n^{-1} \sum_{v\in V}\P^\chi_n\pr{F_{D,\xi,v} } = n^{-1} \sum_{v\in \chi^{-1}(\xi(1_\G))}\P^\chi_n\pr{F_{D,\xi,v} } = (1/2) \P^\chi_n\pr{F_{D,\xi,v_0} }
\end{align*}
where, in the last equality, $v_0$ is any vertex in $\chi^{-1}(\xi(1_\G))$. These equations are justified as follows.  By symmetry, if $v_0,v_1 \in \chi^{-1}(\xi(1_\G))$, then $\P^\chi_n\pr{F_{D,\xi,v_0} }= \P^\chi_n\pr{F_{D,\xi,v_1} }$. On the other hand, if $v \notin \chi^{-1}(\xi(1_\G))$, then $\P^\chi_n\pr{F_{D,\xi,v} }=0$. Lastly, $|\chi^{-1}(\xi(1_\G))|=n/2$.

\end{proof}

\subsection{The density of the rigid set}

This subsection proves Lemma \ref{lem:density}.  So we assume the hypotheses of Proposition \ref{prop:cluster}. An element $x \in X$ is a 2-coloring of the Cayley hyper-tree of $\G$. Interpreted as such, $C_l(x), A_l(x), A'_l(x)$ are well-defined subsets of $\G$ (see \S \ref{sec:strategy} to recall the definitions). 

For $l \in \N \cup \{\infty\}$, let 
\begin{eqnarray*}
\tC_l&=&\{x\in X:~ 1_\G \in C_l(x)\}, \\
\tA_l&=&\{x\in X:~ 1_\G \in A_l(x)\}, \\
\tA'_l&=&\{x\in X:~ 1_\G \in A'_l(x)\}.
\end{eqnarray*}


Recall that $\l_0 = \frac{1}{2^{k-1}-1}$ and $\l=d\l_0$. Since we assume the hypothesis of Prop. \ref{prop:cluster}, $\l$ is asymptotic to $\log(2)k$ as $k \to \infty$. 
\begin{prop}\label{prop:core}
\begin{eqnarray*}
\mu(\tC_\infty) &\ge& 1 - \l^2 e^{-\l} + O(k^6 2^{-2k}), \\
\mu(\tC_\infty \cup \tA_\infty) &\ge& 1- e^{-\l} + O(k^4 2^{-2k}).
\end{eqnarray*}
\end{prop}

\begin{proof}

For brevity, let $e_i \subset \G$ be the subgroup generated by $s_i$. So $e_i$ is a hyper-edge of the Cayley hyper-tree. Let $F^i_l \subset X$ be the set of all $x$ such that 
\begin{enumerate}
\item $1_\G$ supports the edge $e_i$ with respect to $x$ and
\item $e_i \setminus \{1_\G\} \subset C_l(x)$. 
\end{enumerate}
Since $C_{l+1}(x) \subset C_l(x)$, it follows that $F^i_{l+1} \subset F^i_l$. The events $F^i_l$ for $i=1,\ldots, d$ are i.i.d. Let $p_l = \mu(F^i_l)$ be their common probability.

We write $\Prob(\Bin(n,p) =m) = {n \choose m}p^m(1-p)^{n-m}$ for the probability that a binomial random variable with $n$ trials and success probability $p$ equals $m$. Since the events $F^1_{l-1},\ldots, F^d_{l-1}$ are i.i.d., $\tA_{l}$ is the event that either 1 or 2 of these events occur and $\tC_l$ is the event that at least 3 of these events occur, it follows that
$$\mu(\tA_{l}) = \Prob(\Bin(d,p_{l-1}) \in \{1,2\}).$$
$$\mu(\tC_l) = \Prob(\Bin(d,p_{l-1})\ge 3).$$
Thus
$$\mu(\tC_l \cup \tA_{l}) = \Prob(\Bin(d,p_{l-1}) >0).$$

\begin{claim} 
$p_0=\l_0$ and for $l\ge 0$, $p_{l+1}=f(p_l)$ where 
$$f(t) = \l_0 \Prob(\Bin(d-1,t) \ge 3)^{k-1}.$$
\end{claim}

\begin{proof}
To reduce notational clutter, let $F_l=F^1_l$. 
Note that $p_0 =\mu(F_0)=\l_0$ is the probability that the edge $e_1$ is critical. So
$$p_{l+1} = \mu(F_0) \mu(F_{l+1}~  |~ F_0) = \l_0 \mu(F_{l+1}~  |~ F_0).$$
Conditioned on $F_0$, $F_{l+1}$ is the event that $e_1 \setminus \{1_\G\} \in C_{l+1}(x)$. By symmetry and the Markov property $\mu(F_{l+1}~  |~ F_0)$ is the $(k-1)$-st power of the probability that $s_1 \in C_{l+1}(x)$ given that $1_\G$ supports $e_1$. By translation invariance, that probability is the same as the probability that $1_\G \in C_{l+1}(x)$ given that $1_\G$ does not support the edge $e_1$. By definition of $C_{l+1}(x)$ and the Markov property, this is the same as the probability that a binomial random variable with $(d-1)$ trials and success probability $p_l$ is at least 3. This implies the claim.
\end{proof}

The next step is to bound $\Prob(\Bin(d-1,t) \ge 3)$ from below:
\begin{align}
\Prob(\Bin(d-1,t) \ge 3) &= 1  - (1-t)^{d-1} - (d-1)t(1-t)^{d-2} - {d - 1 \choose 2}t^2(1-t)^{d-3} \nonumber \\
&\ge 1 - e^{ -(d-1)t}\left(1 + \frac{(d-1)t}{1-t} + \frac{(d-1)^2t^2}{2(1-t)^2}\right).\label{binomial}
\end{align}
The last inequality follows from the fact that $(1-t)^{d-1} \le e^{-(d-1)t}$. This motivates the next claim:

\begin{claim}\label{claim2}
Suppose $t$ is a number satisfying $\l_0\left(1 - \l^2 e^{1-\l} \right)^{k-1} \le t \le \l_0$. Then for all sufficiently large $k$,
\begin{eqnarray*}\label{assumptions}
0 \le \l-(d-1)t &\le& 1 \\
1 + \frac{(d-1)t}{1-t} + \frac{(d-1)^2t^2}{2(1-t)^2} &\le& (d-1)^2t^2.
\end{eqnarray*}
\end{claim}
\begin{proof}
The first inequality follows from:
$$\l-(d-1)t \ge \l-(d-1)\l_0 = \l_0 >0.$$
The second inequality follows from:
\begin{align*}
\l-(d-1)t &\le \l-(d-1)\l_0\left(1 - \l^2 e^{1-\l} \right)^{k-1}\\
& = d\l_0 - (d-1)\l_0\left(1 - \l^2 e^{1-\l} \right)^{k-1} \\
&\le d\l_0 - (d-1)\l_0( 1- (k-1) \l^2 e^{1-\l})\\
& \le \l_0 + (d-1)\l_0 (k-1) \l^2 e^{1-\l} \le \l_0 + k\l^3 e^{1-\l} \to_{k\to\infty} 0. 
\end{align*}
The third line follows from the general inequality $(1-x)^{k-1} \ge 1-(k-1)x$ valid for all $x \in [0,1]$. To see the limit, observe that under the hypotheses of Proposition \ref{prop:cluster}, $d \sim (\log(2)/2)k2^k$. So $\l \sim \log(2) k$. In particular, $k\l^3 e^{1-\l} \to 0$ and $\l_0 \to 0$ as $k\to\infty$. The implies the limit. 
Thus if $k$ is large enough then the second inequality holds.


To see the last inequality, observe that since $t \le \l_0$, $t \to 0$ as $k \to \infty$. On the other hand, $(d-1)t \sim \l \sim \log(2)k$. Thus $\frac{(d-1)t}{1-t}$ and $(d-1)t$ are asymptotic to $\log(2)k$. Since  $1 + \log(2)k + \frac{\log(2)^2k^2}{2} \le \log(2)^2k^2$ for all sufficiently large $k$, this proves the last inequality assuming $k$ is sufficiently large. 
\end{proof}

Now suppose that $t$ is as in Claim 3. Then
\begin{eqnarray*}
f(t) &\ge & \l_0\left(1 - e^{ -(d-1)t} \left(1 + \frac{(d-1)t}{1-t} + \frac{(d-1)^2t^2}{2(1-t)^2}\right)\right)^{k-1} \\
&\ge& \l_0\left(1 - e^{1-\l} (d-1)^2t^2 \right)^{k-1} \ge \l_0\left(1 - \l^2 e^{1-\l} \right)^{k-1}.
\end{eqnarray*}
The first inequality is implied by (\ref{binomial}). The second and third inequalities follow from Claim 3. For example, since $\l -(d-1)t \le 1$, $e^{-(d-1)t} \le e^{1-\l}$. 

Therefore, if $p_l$ satisfies the bounds $\l_0\left(1 - \l^2 e^{1-\l} \right)^{k-1} \le p_l \le \l_0$ then $f(p_l)=p_{l+1}$ satisfies the same bounds. Since $p_\infty = \lim_{l\to\infty} f^l(\l_0)$, it follows that
\begin{eqnarray}\label{E:p-infinity}
\l_0 \ge p_\infty \ge \l_0\left(1 - \l^2 e^{1-\l} \right)^{k-1} = \l_0 + O(k^3 2^{-2k}).
\end{eqnarray}
Because $(1-\frac{t}{n})^n \le e^{-t}$ for any $t,n>0$,
\begin{eqnarray*}
\mu(\tC_\infty \cup \tA_\infty) &= & \lim_{l \to \infty} \mu(\tC_l \cup \tA_{l}) = \lim_{l \to \infty} \Prob(\Bin(d, p_{l-1}) > 0 ) = \Prob(\Bin(d, p_\infty) > 0 ) \\
& = & 1 - (1-p_\infty)^d \ge 1 - \exp(-p_\infty d) = 1 - e^{-\l} + O(k^4 2^{-2k}).
\end{eqnarray*}
The first equality occurs because $\tC_l\cup \tA_l$ decreases to $\tC_\infty \cup \tA_\infty$. By (\ref{binomial}) and Claim \ref{claim2} (with $d$ in place of $d-1$),
\begin{eqnarray*}
\mu(\tC_\infty ) &=& \Prob(\Bin(d, p_\infty) \ge 3 ) \ge \Prob(\Bin(d, \l_0 + O(k^3 2^{-2k})) \ge 3 ) \\
&\ge & 1 - \exp(-\l_0 d)\left(1 + \frac{d\l_0 }{1-\l_0 } + \frac{d^2\l_0 ^2}{2(1-\l_0 )^2}\right) + O(k^6 2^{-2k})\\
&\ge& 1 - \l^2 e^{-\l} + O(k^6 2^{-2k}).
\end{eqnarray*}

\end{proof}

\begin{lem}\label{lem:A'}
$\mu(\tA'_{\infty}) = o(e^{-\l})$ where the implied limit is as $k\to\infty$ and $\eta$ is bounded.
\end{lem}
\begin{proof}

As in the previous proof, let $e_i \subset \G$ be the subgroup generated by $s_i$. So $e_i$ is a hyper-edge of the Cayley hyper-tree.

Let $x\in X$. We say that an edge $e$ is {\bf attaching} (for $x$) if it is supported by a vertex $v \in A_\infty(x)$ and $e \setminus \{v\} \subset C_\infty(x)$. Let $F(x)=0$ if $1_\G \notin C_\infty(x)$. Otherwise, let $F(x)$ be the number of attaching edges containing $1_\G$. Then by translation invariance,
\begin{eqnarray}\label{A'}
\mu(\tA'_\infty) \le \sum_{m=2}^d m \mu(F(x)=m).
\end{eqnarray}
Let $G \subset X$ be the set of all $x$ such that 
\begin{enumerate}
\item $e_1$ is a critical edge supported by some vertex $v \neq 1_\G$, 
\item $e_1 \setminus \{v,1_\G\} \subset C_\infty(x)$,
\item $v \in A_\infty(x)$.
\end{enumerate}
By the Markov property and symmetry,
\begin{eqnarray}\label{F}
\mu(F(x)=m) \le {d \choose m}\mu(G)^m(1-\mu(G))^{d-m}.
\end{eqnarray}
Let 
\begin{itemize}
\item $G_1 \subset X$ be the set of all $x$ such that $e_1$ is supported by $s_1$,
\item $G_2 \subset X$ be the set of all $x$ such that $e_1 \setminus \{s_1,1_\G\} \subset C_\infty(x)$,
\item $G_3 \subset X$ be the set of all $x$ such that $s_1 \in A_\infty(x)$.
\end{itemize}
By symmetry 
 $$\mu(G) = (k-1)\mu(G_3|G_2 \cap G_1)\mu(G_2|G_1)\mu(G_1).$$
Conditioned on $G_1 \cap G_2$, if $G_3$ occurs then there are no more than $1$ attaching edge $e$ supported by $s_1$ with $e\ne e_1$. By the Markov property and symmetry, 
 $$\mu(G_3|G_2\cap G_1) \leq \Prob(\Bin(d-1,p_\infty) \le 1) = O(\l e^{-\l})$$
 where we have used (\ref{E:p-infinity}). Also $\mu(G_1)=\l_0$. Thus $\mu(G) \leq O(k^2e^{-2\l})$. So (\ref{A'}) and (\ref{F}) along with straightforward estimates imply $\mu(\tA'_\infty) = o(e^{-\l})$.

\end{proof}


\begin{lem}\label{lem:double-prime}
$\limsup_{l\to\infty} \mu(\tA'_l) \le \mu(\tA'_\infty)$.
\end{lem}

\begin{proof}
Given a coloring $\chi:\G \to \{0,1\}$ of the Cayley hyper-tree and $l \in \N$, define $A''_l(\chi) = \cup_{m\ge l} A'_m(\chi)$. Also define 
$\tA''_l=\{x\in X:~ 1_\G \in A''_l(x)\}.$ Since $\tA''_l \supset \tA'_l$ and the sets $\tA''_l$ are decreasing in $l$, it suffices to prove that $\cap_{l\ge 0} \tA''_l \subset \tA'_\infty$. 

Suppose $x \in \cap_{l\ge 0} \tA''_l$. Then there exists an infinite set $S \subset \N$ such that $x \in \tA'_l$ $(\forall l \in S)$. So $1_\G \in A'_l(x)$ $(\forall l \in S)$. So for each $l \in S$, there exist $g_l \in A_l(x) \setminus \{1_\G\}$ and hyper-edges $e_l, f_l \subset \G$ such that
\begin{enumerate}
\item $1_\G$ supports $e_l$ (with respect to $x$),
\item $g_l$ supports $f_l$ (with respect to $x$),
\item $e_l \cup f_l \setminus \{1_\G, g_l\} \subset C_{l-1}(x)$,
\item $e_l \cap f_l \ne \emptyset$. 
\end{enumerate}
Because $e_l \cap f_l \ne \emptyset$, $g_l$ is necessarily contained in the finite set $\{s^{p_i}_is^{p_j}_j:~1 \le i,j\le d, 0 \le p_i \le k\}$. So after passing to an infinite subset of $S$ if necessary, we may assume there is a fixed element $g \in \G$ such that $g=g_l$ $(\forall l \in S)$. Similarly, we may assume there are edges $e,f \subset \G$ such that $e_l=e$ and $f_l=f$ $(\forall l \in S)$. 

Observe that $1_\G \notin C_\infty(x)$ because $1_\G \in A_l(x)$  implies $1_\G \notin C_l(x)$ $(\forall l \in S)$. Similarly, $g \notin C_\infty(x)$. Because $e_l \cup f_l \setminus \{1_\G, g_l\} \subset C_l(x)$ $(\forall l \in S)$ and the sets $C_l(x)$ are decreasing in $l$, it follows that $e \cup f \setminus \{1_\G, g\} \subset C_\infty(x)$. Therefore $\{1_\G,g\} \subset A_\infty(x)$. This verifies all of the conditions showing that $1_\G \in A'_\infty(x)$ and therefore $x \in \tA'_\infty$ as required. 

\end{proof}

We can now prove Lemma \ref{lem:density}:

\begin{proof}[Proof of Lemma \ref{lem:density}]
Observe that the sets $\tC_l,\tA_l, \tA'_l$ are clopen for finite $l$. By Lemma \ref{lem:bsconc}, 
\begin{eqnarray}\label{finiteconc}
\lim_{\delta\searrow 0} \liminf_{n\to\infty} \P^\chi_n\left(\left|\frac{|C_l(\chi) \cup A_l(\chi) \setminus A'_l(\chi)|}{n} - \mu\left(\tC_l \cup \tA_l \setminus \tA'_l\right) \right| <\delta \right)=1.
\end{eqnarray}
for any finite $l$. Since $\tA_\infty \cup \tC_\infty$ is the decreasing limit of $\tA_l \cup \tC_l$, Lemma \ref{lem:double-prime} implies
$$\liminf_{l\to\infty} \mu(\tC_l \cup \tA_l \setminus \tA'_l) \ge \mu(\tC_\infty \cup \tA_\infty \setminus \tA'_\infty).$$
By Proposition \ref{prop:core} and Lemma \ref{lem:A'}, 
$$\mu(\tC_\infty \cup \tA_\infty \setminus \tA'_\infty) \geq  1 - e^{-\l} + o(e^{-\l}).$$
Together with (\ref{finiteconc}), this implies the lemma. 
\end{proof}

\section{Rigid vertices}\label{sec:rigidity}


This section proves Lemma  \ref{lem:rigidity}. So we assume the hypotheses of Proposition \ref{prop:cluster}.

As in the previous section, fix an equitable coloring $\chi:V \to \{0,1\}$. We assume $|V|=n$ and let $\s:\G \to \Sym(V)$ be a uniformly random uniform homomorphism conditioned on the event that $\chi$ is proper with respect to $\s$.


\begin{lem}[Expansivity Lemma]\label{lem:expansivity}
There is a constant $k_0>0$ such that the following holds. If $k\ge k_0$ then with high probability (with respect to the planted model), as $n\to\infty$, for any $T \subset V$ with $|T| \le  2^{-k/2} n$ the following is true. For a vertex $v$ let $E_v$ denote the set of hyperedges supported by $v$. Let $E_T$ be the set of all edges $e\in \cup_{v\in T} E_v$ such that $|e \cap T| \ge 2$. Then
$$\#E_T \le 2 \#T.$$ 
\end{lem}

\begin{proof}

\begin{claim}\label{claim:k}
There exists $k_0 \in \N$ such that $k\ge k_0$ implies 
\begin{itemize}
\item $k/2 \le \l \le k$,
\item $1/2-k2^{1-k/2} \ge \frac{1}{\sqrt{8}}$, 
\item and for any $0 < t\le 2^{-k/2}$ and $k/2 \le \l' \le k$
 $$H(t,1-t) + \l'H(2t/ \l', 1-2t/ \l' ) + 2t \log(4k) + 4t \log(t) \le 0.9t\log(t).$$
 \end{itemize}
\end{claim}

\begin{proof}
Recall that $\l = \log(2)k + O(k 2^{-k})$. So the first two requirements are immediate for $k_0$ large enough.

We estimate each of the first three terms on the left as follows. Because $1 = \lim_{t \searrow 0} \frac{H(t,1-t)}{-t\log(t)}$, there exists $k_0 \in \N$ such that $k \ge k_0$ implies $\frac{H(t,1-t)}{-t\log(t)} \le 1.01$.

Note,
\begin{eqnarray*}
\l'H(2t/ \l', 1-2t/ \l' )  &=& - 2t\log( 2t/\l') - (\l' - 2t)\log(1-2t/\l') \\ 
&=& - 2t\log( 2t/\l') + O(t) \le -2t\log(t) + 2t \log(\l') + O(t)\\
&\le& -2t\log(t) + 2t \log(k) + O(t).
\end{eqnarray*}
So by making $k_0$ larger if necessary, we may assume 
$$\frac{\l'H(2t/ \l', 1-2t/ \l' )}{-t\log(t)} \le 2.01.$$
Since
$$\frac{2t \log(4k)}{-t\log(t)} \le \frac{2\log(4k)}{(k/2)\log(2)}$$ 
we may also assume $\frac{2t \log(4k)}{-t\log(t)} \le 0.01$. Combining these inequalities, we obtain
\begin{eqnarray*}
&&H(t,1-t) + \l'H(2t/ \l', 1-2t/ \l' ) + 2t \log(4k) + 4t \log(t)\\
 &\le& (1.01 + 2.01 + 0.01 -4)(-t\log(t)) \le 0.9t\log(t).
\end{eqnarray*}
\end{proof}
From now on, we assume $k\ge k_0$ with $k_0$ as above. To simplify notation, let $\zeta=2^{-k/2}$. 
Given a $2d$-tuple $c=(c_{1,0}, c_{1,1} \ldots, c_{d,0},c_{d,1})$ of natural numbers, let $E_c$ be the event that there are exactly $c_{i,0}$ critical edges of the form $\{\s(s_i)^j(v) : 0 \leq j \leq k-1\}$ and supported by a vertex of color $0$, and $c_{i,1}$ critical edges of the form $\{\s(s_i)^j(v) : 0 \leq j \leq k-1\}$ and supported by a vertex of color $1$. We denote $|c| = \sum_{1\leq i \leq d, b\in \{0,1\}} c_{i,b}$. Let $\P_{c,n}^\chi$ be the planted model conditioned on $E_c$.



This measure can be constructed as follows. Let $I_c$ be the set of triples $(i,b,j)$ with $1\le i \le d$, $b \in \{0,1\}$ and $1\le j \le c_{i,b}$. First choose edges 
$\{e_{i,b,j}\}_{(i,b,j)\in I_c}$
uniformly at random subject to the conditions:
\begin{enumerate}
\item each $e_{i,b,j} \subset [n]$ has cardinality $k$ and $e_{i,b,j} \cap e_{i,b',j'} = \emptyset$ whenever $b \neq b'$ or $j \ne j'$,
\item each $e_{i,b,j}$ is critical and is supported by a vertex of color $b$ with respect to $\chi$.
\end{enumerate}
Next choose a uniformly random uniform homomorphism $\s$ subject to:
\begin{enumerate}
\item $\chi$ is a proper coloring with respect to $\s$,
\item each $e_{i,b,j}$ is of the form $\{\s(s_i)^j(v) : 0 \leq j \leq k-1\}$ with respect to $\s$,
\item the edges $\{e_{i,b,j}\}_{(i,b,j) \in I_c}$ are precisely the critical edges of $\chi$ with respect to $\s$.
\end{enumerate}
Then $\s$ is distributed according to $\P_{c,n}^\chi$.

For $1\le l \le n$, let $\cT_l$ be the collection of all subsets $T \subset V=[n]$ such that $|T| = l$ and $|E_T| > 2|T|$.  Note that $\cT_1$ is empty.

To prove the lemma, we claim it suffices to show the following:
\begin{claim}\label{claim:expansivity}
If $n$ is sufficiently large and $kn/2 \le |c| \le kn$ then 
$\sum_{2 \le l \le \zeta n} \E_{c,n}^\chi[\#\cT_l]$ tends to zero in $n$.
\end{claim}

We briefly return to the model $\P^\chi_n$ not conditioned on $E_c$.  Let $E$ be the set of all critical edges. By Lemma \ref{lem:bsconc}, with high probability in $\P^\chi_n$, $|E|$ is asymptotic to $\l n$ as $n\to\infty$. Let $E'_n$ be the event that $(k/2)n \le |E| \le kn$. So $\P^\chi_n(E'_n) \to 1$ as $n\to\infty$.

By a first moment argument and the above paragraph, to prove the lemma it suffices to show that $\sum_{2 \le l \leq \z n}\E^\chi_n(\#\cT_l|E_n')$ tends to $0$ in $n$.  Now $\E^\chi_n(\#\cT_l|E_n')$ is a convex combination of $\E^\chi_{c,n}(\#\cT_l)$ over those $c$ such that $kn/2 \le |c| \le kn$, so the lemma follows from Claim \ref{claim:expansivity}. 

Before proving the above claim we need to prove another claim, which needs the following setup.  For $s\in I_c$ and $T\subset V$, let $F_{T,s}$ be the event that $e_s$ is supported by a vertex in $T$ and $|e_s \cap T|\ge 2$. For $S \subset I_c$, let $F_{T,S} = \cap_{s\in S} F_{T,s}$. 
 Note that for any $T \subset [n]$, the event $\{ T \in \cT_l\}$ is contained in $\cup_S F_{T,S}$ where the union is over all $S \subset I_c$ with $|S|=2l$. So
\begin{eqnarray}\label{eqn:eqn}
\E^\chi_{c,n}[ \#\cT_l ] \le \sum_{S,T} \P^\chi_{c,n}(F_{T,S}).
\end{eqnarray}
where the sum is over all $T \subset [n]$ and $S \subset I_c$ with $|T|=l$ and $|S| = 2l$.

Before proving the claim above, we need to prove:

\begin{claim}\label{claim:chain}
For $1 \le l \le \zeta n$, any $T \subset V$ with cardinality $|T|=l$, any $S \subset I_c$ with $|S| \le 2l-1$, and any $s_0 \in I_c \setminus S$,
one has $\P^\chi_{c,n}(F_{T,s_0}| F_{T,S}) \le \frac{2kl^2}{n^2}.$  
\end{claim}

\begin{proof}[Proof of Claim \ref{claim:chain}]
For $s\in I_c$, let $e_s$ be a random edge in $[n]$ with cardinality $k$ as in the sampling algorithm above. Without loss of generality, we imagine that $e_s$ for $s\in S$ has been chosen before $e_{s_0}$. Let $s_0=(i_0,b_0,j_0)$. 
We say that an edge $e$ is of type $i$ if $e$ is of the form $\{\s(s_i)^j(v) : 0 \leq j \leq k-1\}$. Let $S_0 \subset S$ be those edges of type $i_0$, and let $V_0 = \cup_{e \in S_0} e$,

\begin{itemize}
\item $n_i$ be the number of vertices $v \in [n] \setminus V_0$ such that $\chi(v)=i$, and
\item $l_i$ be the number of vertices $v \in T \setminus V_0$ such that $\chi(v)=i$.
\end{itemize}
Let $e_0 = e_{s_0}$. The probability that $e_{0}$ is supported by a vertex $v$ in $T$ is $l_{b_0}/n_{b_0}$ (this is conditioned on the edges $e_s$ for $s\in S$). 



Suppose first that $e_0$ is supported by a vertex $v$ in $T$ with $\chi(v) = 0$ (so $b_0 = 0$). Then the probability that $|e_{0} \cap T| =1$ is
$$\frac{{n_1-l_1 \choose k-1}}{{n_1 \choose k-1}}.$$


It follows that for $b_0 = 0$
$$\P^\chi_{c,n}(F_{T,s_0}| F_{T,S}) 
= \frac{l_0}{n_0} \left(1-\frac{{n_1-l_1 \choose k-1}}{{n_1 \choose k-1}}\right) 
. $$
In order to bound this expression, consider
\begin{eqnarray*}
\frac{{n_1-l_1 \choose k-1}}{{n_1 \choose k-1}} &=& \left(\frac{n_1-l_1}{n_1}\right)\cdots \left(\frac{n_1-l_1-k+2}{n_1-k+2}\right)\\
 &\ge & \left(\frac{n_1-l_1-k+2}{n_1-k+2}\right)^{k-1} =  \left(1-\frac{l_1}{n_1-k+2}\right)^{k-1}  \\
 &\ge &1-\frac{(k-1)l_1}{n_1-k+2}  \ge 1 - \frac{k l_1}{n_1}.
 \end{eqnarray*}
Thus,
 \begin{eqnarray*}
 \P^\chi_{c,n}(F_{T,s_0}| F_{T,S})  &\le& \frac{kl_0l_1}{n_0n_1}.
 \end{eqnarray*}
A similar argument shows that the same bound above holds for the case $b_0 = 1$.

Because $l_0+l_1 \le |T|=l$, $(l_0+l_1)^2 - (l_0 - l_1)^2 \leq l^2$, so that $l_0l_1\le l^2/4$. Note 
$$n_0 \ge n/2 - 2kl \ge n(1/2 - k2^{1-k}) \ge n/\sqrt{8}$$
where we have used the assumption $l\le \zeta n = 2^{-k/2}n$ and Claim 4. Similarly, $n_1 \ge n/\sqrt{8}$. Substitute these inequalities above to obtain
 \begin{eqnarray*}
 \P^\chi_{c,n}(F_{T,s_0}| F_{T,S})  &\le&  \frac{2kl^2}{n^2}.
 \end{eqnarray*} 

\end{proof}
We now prove Claim 5. Apply the chain rule and Claim \ref{claim:chain} to obtain: if $S \subset I_c$ has $|S|=2l$ and $T \subset [n]$ with $|T| = l$ then
$$\P^\chi_{c,n}(F_{T,S}) \le \left(\frac{2kl^2}{n^2}\right)^{2l}.$$
By (\ref{eqn:eqn})
\begin{eqnarray*}
\E^\chi_{c,n}[\#\cT_l] &\le& \sum_{S,T} \P^\chi_{c,n}(F_{T,S}) \le {n \choose l}{|c| \choose 2l}\left(\frac{2kl^2}{n^2}\right)^{2l}
\end{eqnarray*}
where the sum is over all $T \subset [n]$ and $S \subset I_c$ with $|T|=l$, $|S| = 2l$. 
Define $t, \l'$ by $tn = l$ and $|c|=\l' n$. By hypothesis $k/2\le \l' \le k$. Consider the following cases:

\noindent \underline{Case 1}: $2 \leq l \leq n^{0.1}$.  Then we make the following estimates:
\begin{eqnarray*}
{n \choose l} &\leq&  n^l
,\\
{|c| \choose 2l} &\leq & (kn)^{2l}. 
\end{eqnarray*}

It follows that $\E^\chi_{c,n}[\#\cT_l] \le \left(\frac{4k^4l^4}{n}\right)^l$, which is bounded by $n^{-1.1}$ for large enough $n$.

\noindent \underline{Case 2}: $l > n^{0.1}$.  
We make the following estimates: 
\begin{eqnarray*}
{n \choose l} &=&  \exp(nH(t, 1-t)+ 0.5\log(n) + O(1))
,\\
{|c| \choose 2l} &\le& \exp(\l' n H(2t/\l', 1-2t /\l' )+ 0.5\log(kn) + O(1))
\end{eqnarray*}
so that 
$$\E^\chi_{c,n}[\#\cT_l] \le Ckn\exp(n(H(t,1-t) + \l'H(2t/ \l', 1-2t/ \l' ) + 2t \log(2k) + 4t \log(t)))$$ for some constant $C$.
For $n$ sufficiently large and $t \le 2^{-k/2}$, this is bounded above by $Ckn\exp(0.9nt\log(t)) \le Ckn2^{-0.45kn^{0.1}}$ by Claim \ref{claim:k} and the choice of $k_0$ and $l$.  It follows that in this range of $l$,  $\E^\chi_{c,n}[\#\cT_l]$ decays super-polynomially.     

This proves Claim \ref{claim:expansivity} and finishes the lemma.

\end{proof}

\begin{lem}\label{lem:core-rigidity}
Let $\rho>0$. Then there exists $L$ such that $l>L$ implies $C_l(\chi) \subset [n]$ is $\rho$-rigid (with high probability in the planted model as $n\to\infty$).
\end{lem}
\begin{proof}
Without loss of generality, we may assume that $0<\rho< \mu(\tC_\infty)$. 

Observe that the sets $\tC_l$ are clopen for finite $l$. By Lemma \ref{lem:bsconc}, 
$$\lim_{\eta\searrow 0} \liminf_{n\to\infty} \P^\chi_n\left(\left|\frac{|C_l(\chi)|}{n} - \mu\left(\tC_l\right) \right| <\eta \right)=1.$$
Since the sets $C_l(\chi)$ are decreasing with $l$, this implies the existence of $L$ such that $l>L$ implies 
$$\liminf_{n\to\infty} \P^\chi_n\left(\left|\frac{|C_l(\chi)|}{n} - \frac{|C_{l+1}(\chi)|}{n}\right| <\rho/3 \right)=1.$$
Choose $l >L$. Let $\psi: V \to \{0,1\}$ be a $\s$-proper coloring.  Let 
$$T_l = \{v\in C_l(\chi):~ \chi(v) \neq \psi(v)\}.$$
Define $T_{l+1}$ similarly. Since $|C_l(\chi) \setminus C_{l+1}(\chi)| < \rho n/3$ (with high probability) and $T_l \setminus T_{l+1} \subset  C_l(\chi) \setminus C_{l+1}(\chi)$, it follows that $|T_l \setminus T_{l+1}| < \rho n/3$ (with high probability).


For every $v \in T_{l+1}$, let $F_v \subset E_v$ be the subset of $\chi$-critical edges $e$ such that $e \subset C_l(\chi)$.  

We claim that if $v \in T_{l+1}$ then $F_v \subset E_{T_l}$ where 
$$E_{T_l}=\{e \in \cup_{v\in T_l} E_v:~ |e\cap T_l | \ge 2\}.$$
Since $v \in T_{l+1}$, $\psi(v) \ne \chi(v)$. If $e \in F_v$ then $v$ supports $e$ with respect to $\chi$. Therefore because $\psi:[n] \to \{0,1\}$ is a proper coloring, there must exist a vertex $w \in e \setminus \{v\}$ such that $\psi(w) \ne \psi(v)$. This, combined with $\chi(w) \neq \chi(v)$, implies $\psi(w) \neq \chi(w)$ since there are only two possible colors. Since $\{v,w\} \subset e \subset C_l(\chi)$, this means that $|e \cap T_l | \ge 2$ and therefore $e \in E_{T_l}$, which proves the claim.

For every $v \in T_{l+1}$, $|F_v| \ge 3$ by the definition of the sets $C_l(\chi)$. Since edges can only be supported by one vertex, the sets $F_v$ are pairwise disjoint. So 
$$|E_{T_l}| \geq \left|\bigcup_{v\in T_{l+1} } F_v\right| \geq 3|T_{l+1}| \ge 3|T_l| - \rho n.$$
If $|T_l| > \rho n$ then $ |E_{T_l}|\geq 3|T_l| - \rho n > 2|T_l|$. So it follows from Lemma \ref{lem:expansivity} that (with high probability), $|T_l| > 2^{-k/2} n$. Thus $C_l$ is $\rho$-rigid.     
\end{proof}

We can now prove Lemma \ref{lem:rigidity}.
\begin{proof}[Proof of Lemma \ref{lem:rigidity}]
Let $\rho>0$. By Lemma \ref{lem:core-rigidity}, there exists $L$ such that $l>L$ implies $C_l(\chi)$ is $(\rho/3)$-rigid with high probability in the planted model as $n\to\infty$.  So without loss of generality we condition on the event that $C_l(\chi)$ is $(\rho/3)$-rigid.



Now let $l-1>L$. Let $\psi: V \to \{0,1\}$ be a $\s$-proper coloring.  Let 
$$T_{l-1} = \{v\in C_{l-1}(\chi):~ \chi(v) \neq \psi(v)\}.$$
$$T_l = \{v\in C_l(\chi):~ \chi(v) \neq \psi(v)\}.$$
$$T' = \{v\in A_l(\chi) \setminus A'_l(\chi):~ \chi(v) \neq \psi(v)\}.$$
We claim that $|T_{l-1}| \ge |T'|$. To see this, let $v \in T'$. Then there exists an edge $e$ supported by $v$ (with respect to $\chi$) with $e \setminus \{v\} \subset C_{l-1}(\chi)$. Since $\psi$ is proper and $\psi(v)\ne \chi(v)$, there must exist a vertex $w \in e \setminus \{v\}$ with $\psi(w) \ne \chi(w)$. Necessarily, $w \in T_{l-1}$. So there exists a function $f:T' \to T_{l-1}$ such that $f(v)$ is contained in an edge $e$ supported by $v$ with $e \setminus \{v\} \subset C_{l-1}(\chi)$. Because $v \notin A'_l(\chi)$, $f$ is injective. This proves the claim.

Now suppose that $|T_l\cup T'| > \rho n$. Since $T_l$ and $T'$ are disjoint, either $|T_l| > (\rho/3)n$ or $|T'| > (2\rho/3)n$.  If $|T_l| > (\rho/3)n$, we are done because by assumption that $C_l(\chi)$ is $(\rho/3)$-rigid, $|T_l| > 2^{-k/2}n$ and so $$|T_l \cup T'| > 2^{-k/2}n.$$
If $|T'| > (2\rho/3)n$ the claim implies $|T_{l-1}| > (2\rho/3)n$.  Now in the proof of Lemma \ref{lem:core-rigidity} we have shown that $|T_{l-1}\setminus T_l| < (\rho/3)n$, so $|T_l| > (\rho/3)n$ and we are again done.
This proves the lemma.
\end{proof}

\appendix

\section{Topological sofic entropy notions}\label{appendix:sofic-entropy}

In this appendix, we recall the notion of topological sofic entropy from \cite{kerr-li-sofic-amenable} and prove that it coincides with the definition given in \S \ref{sec:entropy}.

Let $T$ be an action of $\G$ on a compact metrizable space $X$. So for $g\in \G$, $T^g:X \to X$ is a homeomorphism and $T^{gh}=T^gT^h$. We will also denote this action by $\G \cc X$.   Let $\s: \G \to \Sym(n)$ be a map, $\rho$ be a pseudo-metric on $X$, $F \Subset \G$ be finite and $\d>0$. For $x,y \in X^n$, let 
$$\rho_\infty (x,y) = \max_i \rho(x_i,y_i), \ \ \
\rho_2(x,y) = \left( \frac{1}{n}\sum_i\rho (x_i,y_i)^2 \right)^{1/2}$$
be pseudo-metrics on $X^n$.  Also let
$$\Map(T,\rho,F,\d,\s) = \{x\in X^n : \forall f\in F,\ \rho_2(T^fx,x\circ \s(f))< \d\}.$$ 
Informally, elements of $\Map(T,\rho,F,\d,\s)$ are ``good models" that approximate partial periodic orbits with respect to the chosen sofic approximation.

For a pseudo-metric space $(Y,\rho)$, a subset $S \subset Y$ is {\bf $(\rho,\eps)$-separated} if for all $s_1\neq s_2 \in S$, $\rho(s_1,s_2) \geq \eps$.  Let $N_\eps(Y,\rho) = \max\{|S|: S\subset Y, S\ \text{is}\ (\rho,\eps) \text{-separated}\}$ be the maximum cardinality over all $(\rho,\eps)$-separated subsets of $Y$.

Given a sofic approximation $\Si$ to $\G$, we define 
$$\th_\Sigma (\G \cc X,\rho) = \sup_{\eps>0}\inf_{F\Subset \G}\inf_{\d>0}\limsup_{i\to \infty}|V_i|^{-1}\log(N_\eps(\Map(T,\rho, F, \d, \s_i), \rho_\infty))$$
where the symbol $F\Subset \G$ means that $F$ varies over all finite subsets of $\G$.

We say that a pseudo-metric $\rho$ on $X$ is {\bf generating} if for every $x \neq y$ there exists $g \in \G$ such that $\rho(gx, gy) > 0$. By \cite[Proposition 2.4]{kerr-li-sofic-amenable}, if $\rho$ is continuous and generating,  $\th_\Sigma(T, \rho)$ is invariant under topological conjugacy and does not depend on the choice of $\rho$. So we define $\th_\Sigma(T)=\th_\Sigma(T, \rho)$ where $\rho$ is any continuous generating pseudo-metric. The authors of \cite{kerr-li-sofic-amenable} define the topological sofic entropy of $\G \cc X$ to be $\th_\Sigma (T)$. The main result of this appendix is:

\begin{prop}
Let $\cA$ be a finite set and $X \subset \cA^\G$ a closed shift-invariant subspace. Let $T$ be the shift action of $\G$ on $X$.  Then $h_\Sigma(\G \cc X) =\th_\Sigma(T)$ where $h_\Sigma(\G \cc X)$ is as defined in \S \ref{sec:entropy}.

\end{prop}

\begin{proof}
To begin, we choose a pseudo-metric on $\cA^\G$ as follows. For $x,y \in \cA^\G$, let $\rho(x,y) = 1_{x_{1_\G} \neq y_{1_\G}}$.  Then $\rho$ is continuous and generating. So $\th_\Sigma(\G \cc X) = \th_\Sigma(\G \cc X,\rho)$.

Let $\eps>0$, $\cO \subset \cA^\G$ be an open set. We first analyze $\Omega(\cO,\eps,\s)$ from the definition of $h_\Sigma(\G \cc X)$.  Note that the topology for $\cA^\G$ is generated by the base $\sB = \{[a]: a \in \cA^F, F\Subset \G\}$ where if $a \in \cA^F$ then  $[a] = \{x \in \cA^\G: x|_F = a\}$.  In other words, open sets of $\sB$ are those that specify a configuration on a finite subset of coordinates.  For $F\Subset \G$ let $\cO(F) = \{y \in \cA^\G: \exists x\in X, y|_F = x|_F\} = \cup_{\substack{a\in \cA^F \\ [a] \cap X \neq \emptyset}}[a]$ be the open set containing all elements containing some configuration that appears in $X$ in the finite window $F$. 

\begin{claim}
Every open superset $\cO \supset X$ contains some open set of the form $\cO(F)$.
\end{claim}
Because $\Omega(\cO,\eps,\s)$ decreases as $\cO$ decreases, it suffices to only consider open sets of the form $\cO(F)$ in the definition of $h_\Sigma(\G \cc X)$.
\begin{proof}
$\cO$ is a union of elements in $\sB$ and $X$ is compact, so that there exists $ X \subset \cO' \subset \cO$ with $\cO'$ containing only finitely many base elements.   Let $F$ be the union of all coordinates specified by base elements in $\cO'$.  It follows that $\cO'$ contains $\cO(F)$. 
\end{proof}

Without loss of generality and for convenience we can assume that $F$ is symmetric, i.e. $F = F^{-1}$, and contains the identity.  This is because we can replace any $F$ with the larger set $F \cup F^{-1} \cup \{1_\G\}$, and both $\Map(T,\rho,F,\d,\s_i)$ and $\Omega(\cO(F),\eps,\s_i)$ are monotone decreasing in $F$.  

Let $n = |V_i|$. We assume $\lim_{i \to \infty} |V_i| = \infty$.  Now for each $x \in \Omega(\cO(F),\eps,\s_i)$ we obtain an element $\tx \in X^n$ and then show that these partial orbits form a good estimate for $\th_\Sigma$. Let $G(x) = \{v\in V_i: \Pi_v^{\s_i}(x) \in \cO(F)\}$.  For every $v\in G(x)$, choose some $\tx_v \in X$ that agrees with $\Pi_v^{\s_i}(x)$ on $F$.  For $v\notin G(x)$ choose an arbitrary element $\tx_v \in X$. Thus $\tx \in X^n$ .

Now for $v \in G(x), f \in F$, $T^f\tx_v(1_\G) = \tx_v(f^{-1}) = x_{\s_i(f)v}$. On the other hand we also want $\tx_{\s_i(f)v}(1_\G) = x_{\s_i(f)v}$, which is true if $v \in \s_i(f)^{-1}G(x)$ and $\s_i(1_\G)\s_i(f) v = \s_i(f)v$.  It follows that $\rho_2(T^f\tx, \tx \circ \s_i(f))<\sqrt{2\eps}$. 

Now consider separation of $\{\tx:x\in \Omega(\cO(F),\eps,\s_i)\}$.  We will show that a slightly smaller subset is $(\rho_\infty,1)$-separated.  By the pigeonhole principle there exists a subset $\bar{V_i}$ of size at least $(1-\eps)n$ such that $ \Omega(\cO(F),\eps,\s_i, \bar{V_i}): = \{x \in \Omega(\cO(F),\eps,\s_i): G(x) = \bar{V_i}\}$ has cardinality at least $e^{-n(H(\eps,1-\eps)+o(1))}\#\Omega(\cO(F),\eps,\s_i)$. Furthermore, if $x,y \in \Omega(\cO(F),\eps,\s_i, \bar{V_i})$ then $\rho_\infty(\tx, \ty) = 1$ if $x(v) \neq y(v)$ for some $v \in \bar{V_i} \cap \Fix(1_\G) $, where $\Fix(1_\G) = \{v\in V_i : \s_i(1_\G)v = v\}$.  Since there are at most $|\cA|^{(\eps + o(1)) n}$ configurations in $\cA^{V_i}$ with some fixed configuration on $\bar{V_i} \cap \Fix(1_\G)$, there exists a $(\rho_\infty,1)$-separated subset of $\Omega(\cO(F),\eps,\s_i, \bar{V_i})$ of size at least $|\cA|^{-(\eps+o(1)) n}\#\Omega(\cO(F)),\eps,\s_i,\bar{V_i})$.  It follows that
$$N_1(\Map(T,\rho,F,\sqrt{2\eps},\s_i),\rho_\infty) \geq |\cA|^{-(\eps +o(1)) n}e^{-n(H(\eps,1-\eps)+o(1))}\#\Omega(\cO(F),\eps,\s_i).$$

On the other hand, suppose we have some $\tx \in \Map(T,\rho,F,\d,\s_i)$.  This means that for every $f \in F$, there exists a set $\tV_i(f)$ of size $>(1- \d^2)n$ such that for $v \in \tV_i(f)$, $\tx_{\s_i(f)v}(1_\G) = T^f\tx_v(1_\G) = \tx_v(f^{-1})$.  Let $\tV_i = \cap_{f\in F}\tV_i(f)$. Then $|\tV_i| > (1-|F|\d^2)n$ and for $v \in \tV_i$, for every $f \in F$, $\tx_{\s_i(f)v}(1_\G) = T^f\tx_v(1_\G) = \tx_v(f^{-1})$. 

Define $x \in \cA^{V_i}$ by $x_v = \tx_v(1_\G)$.  Then for any fixed $v\in \tV_i$, for every $f\in F$, $\Pi_v^{\s_i}(x)(f) = x_{\s_i(f^{-1})v} = \tx_{\s_i(f^{-1})v}(1_\G) = T^{f^{-1}}\tx_v(1_\G) = \tx_v(f)$.  Since $\tx_v \in X$, it follows that $x \in \Omega(\cO(F), \d^2|F|,\s_i)$.

Also note that $\tx,\ty \in \Map(T,\rho,F,\d,\s_i)$ are $(\rho,\eps)$-separated for any $\eps \leq 1$ if and only if $\tx_v(1_\G) \neq \ty_v(1_\G)$ for some $v \in V_i$, so that $x \neq y$.  It follows that 
$$N_\eps(\Map(T,\rho,F,\d,\s_i),\rho_\infty) \leq \#\Omega(\cO(F), \d^2|F|,\s_i).$$  

Note that in the definitions of $h_\Sigma$ and $\th_\Sigma$, $F$ is fixed with respect to $\d$.

\end{proof}

\section{Concentration for the planted model}\label{sec:concentration}

\begin{defn}[Hamming metrics]
Define the {\bf normalized Hamming metric} $d_{\Sym(n)}$ on $\Sym(n)$ by
$$d_{\Sym(n)}(\s_1,\s_2) = n^{-1} \#\{ i \in [n]:~ \s_1(i) \ne \s_2(i)\}.$$
Define the {\bf normalized Hamming metric} $d_\Hom$ on $\Hom(\G, \Sym(n))$ by
$$d_\Hom(\s_1,\s_2) = \sum_{i=1}^d d_{\Sym(n)}(\s_1(s_i),\s_2(s_i)).$$
\end{defn}

The purpose of this section is to prove:
\begin{thm}\label{thm:planted-concentration}
There exist constants $c,\l>0$ (depending only on $k,d$) such that for every $\d>0$ there exists $N_\d$ such that for all $n>N_\d$, for every $1$-Lipschitz $f:\Hom_\chi(\G,\Sym(n))\to \R$,
\begin{eqnarray*}
\P^\chi_n\left( \left| f - \E^\chi_n[f] \right| > \delta \right) \le c\exp(-\l \d^2n).
\end{eqnarray*}
\end{thm}

\subsection{General considerations}
To begin the proof we first introduce some general-purpose tools.

\begin{defn}
A {\bf metric measure space} is a triple $(X,d_X,\mu)$ where $(X,d_X)$ is a metric space and $\mu$ is a Borel probability measure on $X$.  We will say $(X,d_X,\mu)$ is {\bf $(c,\l)$-concentrated} if for any 1-Lipschitz function $f:X \to \R$,
$$\mu\left( \left|f  - \int f~d\mu\right| > \eps\right) <  ce^{-\l \eps^2}.$$ 
 
If $(X,d_X,\mu)$ is $(c,\l)$-concentrated and $f:X \to \R$ is $L$-Lipschitz, then since $f/L$ is $1$-Lipschitz,
\begin{eqnarray}\label{L}
\mu\left( \left|f  - \int f~d\mu\right| > \eps\right)  = \mu\left( \left|f/L  - \int f/L~d\mu\right| > \eps/L\right)  <   c\exp(-\l \eps^2/L^2).
\end{eqnarray}
\end{defn}

\begin{lem}\label{lem:lipschitz}
Let $(X,d_X,\mu)$ be $(c,\l)$-concentrated. If $\phi:X \to Y$ is an $L$-Lipschitz map onto a measure metric space $(Y,d_Y,\nu)$ and $\nu = \phi_*\mu$ is the push-forward measure, then $(Y,d_Y,\nu)$ is $(c,\l/L^2)$-concentrated.
\end{lem}

\begin{proof}
This follows from the observation that if $f:Y \to \R$ is 1-Lipschitz, then the pullback $f \circ \phi: X \to \R$ is $L$-Lipschitz. So  equation (\ref{L}) implies 
$$\nu\left( \left|f  - \int f~d\nu\right| > \eps\right) = \mu\left( \left|f\circ \phi  - \int f\circ \phi~d\mu\right| > \eps\right)   < c\exp(-\l \eps^2/L^2).$$
\end{proof}

The next lemma is concerned with the following situation. Suppose $X=\sqcup_{i\in I} X_i$ is a finite disjoint union of spaces $X_i$. Even if we have good concentration bounds on the spaces $X_i$, this does not imply  concentration on $X$ because it is possible that a $1$-Lipschitz function $f$ will have different means when restricted to the $X_i$'s. However, if most of the mass of $X$ is concentrated on a sub-union $\cup_{j \in J} X_j$ (for some $J \subset I$) and the sets $X_j$ are all very close to each other, then there is a weak concentration inequality on $X$.

\begin{lem}\label{lem:concentration-criterion}
Let $(X,d_X,\mu)$ be a measure metric space with diameter $\le 1$. Suppose $X=\sqcup_{i\in I} X_i$ is a finite disjoint union of spaces $X_i$, each with positive measure ($\mu(X_i)>0$). Let $\mu_i$ be the induced probability measure on $X_i$. Suppose there exist $J\subset I$ and constants $\eta,\d,\l,c>0$ satisfying:
\begin{enumerate}
\item $\mu( \cup_{j \in J} X_j) \ge 1-\eta \ge 1/2.$
\item For every $j,k \in J$, there exists a measure $\mu_{j,k}$ on $X_j \times X_k$ with marginals $\mu_j,\mu_k$ respectively such that 
$$\mu_{j,k}(\{(x_j,x_k):~d_X(x_j,x_k) \le \delta\}) = 1.$$
\item For each $j \in J$, $(X_j,d_X,\mu_j)$ is $(c,\l)$-concentrated.
\end{enumerate}
Then for every $1$-Lipschitz function $f:X \to \R$ and every $\eps>\d+2\eta$,
\begin{eqnarray*}
\mu\left( \left| f - \int f~d\mu\right| > \eps\right) \le \eta + c\exp\left( - \l\left( \eps -\delta - 2\eta \right)^2\right). 
\end{eqnarray*}

\end{lem}

\begin{proof}
Let $f:X \to \R$ be a 1-Lipschitz function. After adding a constant to $f$ if necessary, we may assume $\int f~d\mu=0$. Note that the mean of $f$ is a convex combination of its restrictions to the $X_i$'s:
\begin{eqnarray*}
0=\int f(x)~d\mu(x) &=&  \sum_{i\in I} \mu(X_i) \int f(x_i) ~d\mu_i(x_i) \\
&=& \sum_{i\in I \setminus J} \mu(X_i) \int f(x_i) ~d\mu_i(x_i) + \sum_{j\in J} \mu(X_j) \int f(x_j) d\mu_j(x_j).
\end{eqnarray*}
Since $f$ is 1-Lipschitz with zero mean, $|f| \le \textrm{diam}(X) \le 1$.  So
\begin{eqnarray*}
\left|\mu(\cup_{j\in J} X_j)^{-1} \sum_{j \in J} \mu(X_j) \int f(x_j)~d\mu_j(x_j) \right| &=& \left|\mu(\cup_{j\in J} X_j)^{-1} \sum_{i\in I \setminus J} \mu(X_i) \int f(x_i) ~d\mu_i(x_i) \right| \\
&\le& \frac{\eta}{1-\eta} \le 2\eta
\end{eqnarray*}
where the last inequality uses that $\mu(\cup_{j\in J} X_j) \ge 1-\eta$ and $\eta \le 1/2$. 

For any $j,k \in J$, the $\mu_j$ and $\mu_k$-means of $f$ are $\d$-close:
\begin{eqnarray*}
\left| \int f(x_j)~d\mu_j(x_j)  - \int f(x_k)~d\mu_k(x_k)\right| &=&\left| \int f(x_j) - f(x_k) ~d\mu_{j,k}(x_j,x_k)  \right| \\
&\le&  \int |f(x_j) - f(x_k)| ~d\mu_{j,k}(x_j,x_k)  \le \delta.
\end{eqnarray*}
So for any $j_0 \in J$,
$$\left| \int f(x_{j_0})~d\mu_{j_0}(x_{j_0}) - \mu(\cup_{j\in J} X_j)^{-1} \sum_{j \in J} \mu(X_j) \int f(x_j)~d\mu_j(x_j) \right| \le \delta.$$
Combined with the previous estimate, this gives
$$\left| \int f(x_{j_0})~d\mu_{j_0}(x_{j_0}) \right| \le \delta + 2\eta.$$

Now we estimate the $\mu$-probability that $f$ is $>\eps$ (assuming $\eps>\delta+2\eta$):

\begin{eqnarray*}
\mu\left( \left| f - \int f~d\mu\right| > \eps\right)  &=& \mu( | f | > \eps) \\
&\le&  \eta +  \sum_{j \in J}  \mu_j( | f|  > \eps) \mu(X_j) \\
&\le& \eta + \sum_{j \in J}  \mu_j\left( \left| f -\int f(x_{j})~d\mu_{j}(x_{j}) \right|  > \eps - \left| \int f(x_{j})~d\mu_{j}(x_{j}) \right| \right) \mu(X_j) \\
&\le& \eta + \sum_{j \in J}  \mu_j\left( \left| f -\int f(x_{j})~d\mu_{j}(x_{j}) \right|  > \eps -\delta -  2\eta\right) \mu(X_j) \\
&\le& \eta + c\exp\left( - \l\left( \eps -\delta - 2\eta \right)^2\right). 
\end{eqnarray*}

\end{proof}

The next lemma is essentially the same as \cite[Proposition 1.11]{MR1849347}. We include a proof for convenience.
\begin{lem}\label{lem:product-concentration}\cite{MR1849347}
Suppose $(X,d_X,\mu)$ is $(c_1,\l_1)$-concentrated and $(Y,d_Y, \nu)$ is $(c_2,\l_2)$-concentrated. Define a metric  on $X\times Y$  by $d_{X\times Y}( (x_1,y_1),(x_2,y_2)) = d_X(x_1,x_2)  + d_Y(y_1,y_2)$. Then $(X\times Y, d_{X\times Y}, \mu \times \nu)$ is $(c_1+c_2, \min(\l_1,\l_2)/4)$-concentrated.  
\end{lem}

\begin{proof}
Let $F:X \times Y \to \R$ be 1-Lipschitz. For $y \in Y$, define $F^y:X \to \R$ by $F^y(x)=F(x,y)$. Define $G:Y \to \R$ by $G(y) = \int F^y(x)~d\mu(x)$. Then $F^y$ and $G$ are $1$-Lipschitz. 

If $|F(x,y) - \int F~d\mu\times \nu| > \eps$ then either $|F^y(x) - \int F^y~d\mu| > \eps/2$ or $|G(y) - \int G~d\nu| > \eps/2$. Thus
\begin{eqnarray*}
&&\mu\times \nu\left( \left\{ (x,y):~ \left|F(x,y)-\int F~d\mu\times \nu\right| > \eps\right\} \right)\\
 &\le& \mu \times \nu\left( \left\{(x,y):~\left|F^y(x) - \int F^y~d\mu\right| > \eps/2 \right\}\right) + \nu\left( \left\{y:~\left| G(y) - \int G~d\nu\right| > \eps/2 \right\}\right) \\
&\le& c_1 e^{-\l_1 \eps^2/4} + c_2 e^{-\l_2 \eps^2/4} \le (c_1+c_2) \exp( - \min(\l_1,\l_2) \eps^2/4).
\end{eqnarray*}
\end{proof}

\begin{lem}\label{lem:base}
Let $(X,d_X,\mu)$ and $(Y,d_Y,\nu)$ be metric-measure spaces. Suppose 
\begin{enumerate}
\item $X,Y$ are finite sets, $\mu$ and $\nu$ are uniform probability measures,
\item there is a surjective map $\Phi:X \to Y$ and a constant $C>0$ such that $|\Phi^{-1}(y)|=C$ for all $y \in Y$,
\item $(Y,d_Y, \nu)$ is $(c_1,\l_1)$-concentrated,
\item for each $y \in Y$, the fiber $\Phi^{-1}(y)$ is $(c_2,\l_2)$-concentrated (with respect to the uniform measure on $\Phi^{-1}(y)$ and the restricted metric),
\item for each $y_1,y_2 \in Y$ there is a probability measure $\mu_{y_1,y_2}$ on $\Phi^{-1}(y_1) \times \Phi^{-1}(y_2)$ with marginals equal to the uniform measures on $\Phi^{-1}(y_1)$ and  $\Phi^{-1}(y_2)$ such that
$$\mu_{y_1,y_2}(\{(x_1,x_2):~ d_X(x_1,x_2) \le d_Y(y_1,y_2)\})=1.$$
\end{enumerate}
Then  $(X,d_X,\mu)$ is $(c_1+c_2, \min(\l_1,\l_2)/4)$-concentrated. 
\end{lem}

\begin{proof}
Let $f:X \to \R$ be 1-Lipschitz. Let $\E[f|Y]:Y \to \R$ be its conditional expectation defined by
$$\E[f|Y](y) = | \Phi^{-1}(y)|^{-1} \sum_{x \in \Phi^{-1}(y)} f(x).$$
Also let $\E[f] = |X|^{-1}\sum_{x\in X} f(x)$ be its expectation. 

We claim that $\E[f|Y]$ is $1$-Lipschitz. So let $y_1,y_2 \in Y$. By hypothesis (5)
\begin{eqnarray*}
\E[f|Y](y_1) - \E[f|Y](y_2) &=& \int f(x_1)-f(x_2)~d\mu_{y_1,y_2}(x_1,x_2) \\
&\le& \int d_X(x_1, x_2)~d\mu_{y_1,y_2}(x_1,x_2) \\
&\le& d_Y(y_1, y_2).
\end{eqnarray*}
The first inequality holds because $f$ is 1-Lipschitz and the second by hypothesis (5). This proves $\E[f|Y]$ is $1$-Lipschitz. 

Let $\eps>0$. Because $\Phi$ is $C$-to-1, it pushes forward the measure $\mu$ to $\nu$. Because $(Y,d_Y, \nu)$ is $(c_1,\l_1)$-concentrated, 
\begin{eqnarray}\label{base}
\mu\left( \left|\E[f|Y]\circ \Phi  - \int f~d\mu\right| > \eps/2\right) = \nu\left( \left|\E[f|Y]  - \E[f] \right| > \eps/2\right) <  c_1e^{-\l_1 \eps^2/4}.
\end{eqnarray}
Because each fiber $\Phi^{-1}(y)$ is $(c_2,\l_2)$-concentrated, for any $y \in Y$,
$$|\Phi^{-1}(y)|^{-1}\#\left\{x \in \Phi^{-1}(y):~\left|f(x) - \E[f|Y](y)  \right| > \eps/2\right\} <  c_2e^{-\l_2 \eps^2/4}.$$ 
Average this over $y \in Y$ to obtain
$$\mu\left( \{x\in X:~ \left| f(x) - \E[f|Y](\Phi(x))  \right| > \eps/2 \right) <  c_2e^{-\l_2 \eps^2/4}.$$ 
Combine this with (\ref{base})  to obtain
\begin{eqnarray*}
\mu\left( \left|f  - \int f~d\mu\right| > \eps\right) &\le& \mu\left( \left|f  - \E[f|Y](\Phi(x)) \right| > \eps/2 \right) +  \mu\left( \left|\E[f|Y](\Phi(x))  - \int f~d\mu\right| > \eps/2\right) \\
 &\le& c_2e^{-\l_2 \eps^2/4}+  c_1e^{-\l_1 \eps^2/4}
\end{eqnarray*}
which implies the lemma.

\end{proof}

\subsection{Specific considerations}

Given an equitable coloring $\chi:[n] \to \{0,1\}$, let $H_\chi$ be the stabilizer of $\chi$:
$$H_\chi = \{g \in \Sym(n):~ \chi(gv)=\chi(v)~\forall v \in [n]\}.$$

\begin{lem}\label{lem:sym2}
The group $H_\chi$ is $(4, n/16)$-concentrated (when equipped with the uniform probability measure and the restriction of the normalized Hamming metric $d_{\Sym(n)}$).
\end{lem}

\begin{proof}
The group $H_\chi$ is isomorphic to the direct product $\Sym(\chi^{-1}(0)) \times \Sym(\chi^{-1}(1)) $ which is isomorphic to $\Sym(n/2)^{2}$. By \cite[Corollary 4.3]{MR1849347}, $\Sym(n/2)$ is $(2,n/16)$-concentrated.  So the result follows from Lemmas \ref{lem:product-concentration} and \ref{lem:lipschitz}. This uses that the inclusion map from $\Sym(n/2)^2$ to itself is $(1/2)$-Lipschitz when the source is equipped with the sum of the $d_{\Sym(n/2)}$-metrics and the target equipped with the $d_{\Sym(n)}$ metric. 
\end{proof}

We need to show that certain subsets of the group $\Sym(n)$ are concentrated. To define these subsets, we need the following terminology.

Recall that a {\bf $k$-partition} of $[n]$ is an unordered partition $\pi =\{P_1,\ldots,P_{n/k}\}$ of $[n]$ such that each $P_i$ has cardinality $k$. Let $\Part(n,k)$ be the set of all $k$-partitions of $[n]$. The group $\Sym(n)$ acts on $\Part(n,k)$ by $g \pi = \{gP_1,\ldots, gP_{n/k} \}$. 

Let $\s \in \Sym(n)$. The {\bf orbit-partition} of $\s$ is the partition $\Orb(\s)$ of $[n]$ into orbits of $\s$. For example, for any $v \in [n]$ the element of $\Orb(\s)$ containing $v$ is $\{\s^i v:~ i \in \Z\} \subset [n]$. Let $\Sym(n,k) \subset \Sym(n)$ be the set of all permutations $\s \in \Sym(n)$ such that the orbit-partition of $\s$ is a $k$-partition. 


Recall from \S \ref{sec:almost} that a $k$-partition $\pi$ {\bf has type $\vt =(t_j)_{j=0}^k \in [0,1]^{k+1}$ with respect to a coloring $\chi$} if the number of partition elements $P$ of $\pi$ with $|P \cap \chi^{-1}(1)| = j$ is $t_jn$. We will also say that a permutation $\s \in \Sym(n,k)$ {\bf has type $\vt =(t_j)_{j=0}^k \in [0,1]^{k+1}$ with respect to a coloring $\chi$} if its orbit-partition $\Orb(\s)$ has type $\vt$ with respect to $\chi$. 

Let $\Sym(n,k;\chi,\vt)$ be the set of all permutations $\s \in \Sym(n,k)$ such that $\s$ has type $\vt$ with respect to $\chi$. 

\begin{lem}\label{lem:orbit-concentration}
The subset $\Sym(n,k;\chi,\vt)$ is either empty or $(6,\l n)$-concentrated (when equipped with the normalized Hamming metric $d_{\Sym(n)}$ and the uniform probability measure) where $\l>0$ is a constant depending only on $k$. 
\end{lem}

\begin{proof}
Let $\Part(n,k;\chi,\vt)$ be the set of all (unordered) $k$-partitions of $[n]$ with type $\vt$ (with respect to $\chi$). We will consider this set as a metric space in which the distance between partitions $\pi,\pi' \in \Part(n,k;\chi,\vt)$ is $d(\pi,\pi')=\frac{k|\pi \vartriangle \pi'|}{2n}$ where $\vartriangle$ denotes symmetric difference. 

Let $\Orb:\Sym(n,k;\chi,\vt) \to \Part(n,k;\chi,\vt)$ be the map which sends a permutation to its orbit-partition. We will verify the conditions of Lemma \ref{lem:base} with $X=\Sym(n,k;\chi,\vt)$, $Y=\Part(n,k;\chi,\vt)$ and $\Phi=\Orb$. Condition (1) is immediate.

 Observe that $\Orb$ is surjective and constant-to-1. In fact for any partition $\pi \in \Part(n,k;\chi,\vt)$, $|\Orb^{-1}(\pi)| = (k-1)!^{n/k}$ since an element $\s\in \Orb^{-1}(\pi)$ is obtained by choosing a $k$-cycle for every part of $\pi$. To be precise, if $\pi=\{P_1,\ldots, P_{n/k}\}$ then $\Orb^{-1}(\pi)$ is the set of all permutations $\s$ of the form $\s = \prod_{i=1}^{n/k} \s_i$ where $\s_i$ is a $k$-cycle with support in $P_i$.  This verifies condition (2) of Lemma \ref{lem:base}.

Observe that $H_\chi$ acts transitively on $\Part(n,k;\chi,\vt)$. Fix $\pi \in \Part(n,k;\chi,\vt)$ and define a map $\phi:H_\chi \to \Part(n,k;\chi,\vt)$ by $\phi(h)=h\pi$. We claim that $\phi$ is $k^2/2$-Lipschitz. Indeed, if $h_1,h_2 \in H_\chi$ then
\begin{eqnarray*}
d(h_1\pi, h_2\pi) &=& \frac{k|h_1\pi \vartriangle h_2\pi|}{2n} \\
&\le& \frac{k^2 \#\{p \in [n]:~ h_1(p) \ne h_2(p) \} }{2n} \\
&=& \frac{k^2}{2}d_{\Sym(n)}(h_1,h_2). 
\end{eqnarray*}
Because $H_\chi$ is $(4, n/16)$-concentrated by Lemma \ref{lem:sym2}, Lemma \ref{lem:lipschitz} implies 
$\Part(n,k;\chi,\vt)$ is $(4, n/4k^4)$-concentrated. This verifies condition (3) of Lemma \ref{lem:base}.

We claim that $\Orb^{-1}(\pi)$ is $(2, n/2k)$-concentrated. To see this, let $\pi=\{P_1,\ldots, P_{n/k}\}$ and let $\Sym_k(P_i) \subset \Sym(n)$ be the set of all $k$-cycles with support in $P_i$. Then $\Orb^{-1}(\pi)$ is isometric to $\Sym_k(P_1) \times \cdots \times \Sym_k(P_{n/k})$. The diameter of $\Sym_k(P_i)$, viewed as a subset of $\Sym(n)$ with the normalized Hamming metric on $\Sym(n)$, is $k/n$.  So the claim follows from  \cite[Corollary 1.17]{MR1849347}. This verifies condition (4) of Lemma \ref{lem:base}.

For $\pi_1,\pi_2 \in \Part(n,k;\chi,\vt)$, let $X_{\pi_1,\pi_2}$ be the set of all pairs $(\s_1,\s_2) \in \Orb^{-1}(\pi_1)\times \Orb^{-1}(\pi_2)$ such that if $P \in \pi_1\cap \pi_2$ then the restriction of $\s_1$ to $P$ equals the restriction of $\s_2$ to $P$. Observe that $X_{\pi_1,\pi_2}$ is non-empty and the projection maps $X_{\pi_1,\pi_2} \to\Orb^{-1}(\pi_i)$ ($i=1,2$) are constant-to-1. In fact, for any $\s_1 \in \Orb^{-1}(\pi_1)$, the set of $\s_2$ with $(\s_1,\s_2) \in X_{\pi_1,\pi_2}$ is bijective with the set of assignments of $k$-cycles to parts in $\pi_2 \setminus \pi_1$. 

Let $\mu_{\pi_1,\pi_2}$ be the uniform probability measure on $X_{\pi_1,\pi_2}$. Since the projection maps are constant-to-1, the marginals of $\mu_{\pi_1,\pi_2}$ are uniform.  Moreover, if $(\s_1,\s_2) \in X_{\pi_1,\pi_2}$ then 
$$\{i \in [n]:~ \s_1(i) \ne \s_2(i)\} \subset \cup_{P \in \pi_1 \setminus \pi_2} P.$$
Thus
$$d_{\Sym(n)}(\s_1,\s_2) \le n^{-1} |\cup_{P \in \pi_1 \setminus \pi_2} P|  = n^{-1} k |\pi_1 \vartriangle \pi_2|/2 = d(\pi_1,\pi_2).$$
This verifies condition (5) of Lemma \ref{lem:base}.

We have now verified all of the conditions of Lemma \ref{lem:base}. The lemma follows.

\end{proof}

Let $\Sym(n,k;\chi)$ be the set of all $\s \in \Sym(n,k)$ such that if $\vt = (t_j)_{j=0}^k$ is the type of $\s$ with respect to $\chi$ then $t_0=t_k=0$. In other words, $\s \in \Sym(n,k;\chi)$ if and only if the orbit-partition $\pi$ of $\s$ is proper with respect to $\chi$ (where we think of $\pi$ as a collection of hyper-edges). 

Let $\vs=(s_j)$ with $s_0 = s_k = 0$ and $s_j = \frac{1}{k(2^k-2)} {k \choose j}$ for $0 < j < k$. For $\delta>0$ let $\Sym_\delta(n,k;\chi)$ be the set of all $\s \in \Sym(n,k;\chi)$ such that if $\vt=(t_i)_{i=0}^k$ is the type of $\s$ (with respect to $\chi$) then 
$$\sum_{i=0}^k |s_i - t_i|^2 < \delta^2.$$

\begin{lem}\label{lem:average}
With notation as above, for sufficiently large $n$
$$ \frac{|\Sym_\delta(n,k;\chi)|}{|\Sym(n,k;\chi)|} \ge 1 - e^{-\l_1 \delta^2 n}$$
where $\l_1>0$ is a constant depending only on $k$.
\end{lem}

\begin{proof}
Let 
\begin{eqnarray*}
\Part(n, k; \chi,\vt) &=& \{\pi \in \Part(n,k):~\pi \textrm{ has type $\vt$ with respect to $\chi$}\},\\
\Part(n, k; \chi) &=& \{ \pi \in \Part(n,k):~ \chi \textrm{ is proper with respect to $\pi$} \}, \\
\Part_\delta(n, k; \chi) &=& \left\{ \pi \in \Part(n,k;\chi):~ \textrm{ if $\vt$ is the type of $\pi$ with respect to $\chi$ then } \sum_{i=0}^k |t_i - s_i|^2 < \delta^2\right\}.
\end{eqnarray*}
The orbit-partition map from $\Sym(n,k) \to \Part(n,k)$ is constant-to-1 and maps $\Sym(n,k;\chi)$ onto $\Part(n,k;\chi)$ and $\Sym_\delta(n,k;\chi)$ onto $\Part_\delta(n,k;\chi)$. Therefore, it suffices to prove
$$ \frac{|\Part_\delta(n,k;\chi)|}{|\Part(n,k;\chi)|} \ge 1 - e^{-\l \delta^2 n}$$
where $\l>0$ is a constant depending only on $k$.

Let $\widetilde{\sM}$ be the set of all vectors $\vt = (t_i)_{i=0}^{k} \in [0,1]^{k+1}$ such that $t_0=t_k=0$, $\sum_i t_i = 1/k$ and $\sum_j jt_j = 1/2$. 

 Recall from Lemma \ref{lem:numberofpartitions} that if $\vt \in \widetilde{\sM}$  and $n\vt$ is $\Z$-valued then
$$(1/n)\log |\Part(n,k;\chi,\vt) | = (1-1/k)(\log(n)-1) -\log(2)+J(\vt) + O(n^{-1}\log(n))$$
where $J(\vt) = H(\vec{t})-\sum_{j=0}^k t_j \log(j!(k-j)!)$. By the proof of Theorem  \ref{thm:1moment} (specifically equation (\ref{max-type})), $J$ is uniquely maximized in $\widetilde{\sM}$ by the vector $\vs$. 

In order to get a lower bound on $|\Part(n,k;\chi)|$, observe that there exists $\vtr \in  \widetilde{\sM}$ such that $n\vtr$ is $\Z$-valued and $|s_i-r_i| \le k/n$ for all $i$. Thus $J(\vtr)-J(\vs)=O(1/n)$. It follows that
\begin{eqnarray*}
\frac{1}{n} \log |\Part(n,k;\chi)| &\ge& \frac{1}{n} \log |\Part(n,k;\chi, \vtr)| = (1-1/k)(\log(n)-1) -\log(2)+J(\vs) + O(n^{-1}\log(n)).
\end{eqnarray*}

We claim that the Hessian of $J$ is negative definite. To see this, one can consider $J$ to be a function of  $[0,1]^{k+1}$. The linear terms in $J$ do not contribute to its Hessian. Since the second derivative of $x\mapsto -x\log x$ is $-1/x$, 
\begin{displaymath}
\frac{\partial^2 J}{\partial t_{i} \partial t_{j}} = \left\{
 \begin{array}{cc} 
0 & i \ne j \\
-1/t_{i} & i = j 
\end{array}\right.
\end{displaymath}
Thus the Hessian is diagonal and every eigenvalue is negative; so it is negative definite. 

Thus if $\vt \in \widetilde{\sM}$ is such that $\sum_i |t_i-s_i|^2 \ge \delta^2$ then
$$(1/n)\log |\Part(n,k;\chi,\vt) | \le  (1-1/k)(\log(n)-1) -\log(2)+J(\vs) - \delta^2 \l_1' + O(n^{-1}\log(n))$$
where $\l_1' = \frac{1}{2} \min_{\vt \in \widetilde{\sM}} \min_{1\le i \le k-1} 1/t_i$ is half the smallest absolute value of an eigenvalue of the Hessian of $J$ on $\widetilde{\sM}$.

If $\vt$ is the type of a $k$-partition $\pi$ of $n$ then $t_i \in \{0,1/n,2/n,\ldots, 1\}$. Thus the number of different types of $k$-partitions of $[n]$ is bounded by a polynomial in $n$ (namely $(n+1)^{k+1}$). Thus 
\begin{eqnarray*}
 \frac{|\Part_\delta(n,k;\chi)|}{|\Part(n,k;\chi)|} &\ge& 1 - (n+1)^{k+1} \frac{\exp(n  [(1-1/k)(\log(n)-1) -\log(2)+J(\vs) - \delta^2 \l_1' + O(n^{-1}\log(n))])}{\exp(n  [(1-1/k)(\log(n)-1) -\log(2)+J(\vs) + O(n^{-1}\log(n))])} \\
 &=& 1 - n^c \exp(-\delta^2 \l_1' n)
\end{eqnarray*}
where $c = O_k(1)$. This implies the lemma.
 
\end{proof}

Recall that a {\bf $k$-cycle} is a permutation $\pi \in \Sym(n)$ of the form $\pi=(v_1,\ldots, v_k)$ for some $v_1,\ldots, v_k \in [n]$. In other words, $\pi$ has $n-k$ fixed points and one orbit of size $k$. The {\bf support} of $\pi \in \Sym(n)$ is the complement of the set of $\pi$-fixed points. It is denoted by $\supp(\pi)$. Two permutations are {\bf disjoint} if their supports are disjoint. A permutation $\pi \in \Sym(n)$ is a {\bf disjoint product of $k$-cycles} if there exist pairwise disjoint $k$-cycles $\pi_1,\ldots, \pi_m$ such that $\pi = \pi_1\cdots \pi_m$. In this case we say that each $\pi_i$ is {\bf contained in $\pi$}.

\begin{lem}\label{lem:coupling}
Let $\vt, \vu \in [0,1]^{k+1}$. Suppose
$$\sum_{i=0}^k |t_i-u_i| < \delta.$$
Suppose $ \Sym(n,k;\chi,\vt)$ and $ \Sym(n,k;\chi,\vu)$ are non-empty (for some integer $n$ and equitable coloring $\chi$). 

For $\s, \s' \in \Sym(n,k)$, let $|\s \vartriangle \s'|$ be the number of $k$-cycles $\tau$ that are either in $\s$ or in $\s'$ but not in both. Let
$$Z = \{(\s,\s') \in \Sym(n,k;\chi, \vt) \times \Sym(n,k;\chi, \vu):~| \s \vartriangle \s'| \le  \delta n \}.$$
Then $Z$ is non-empty and there exists a probability measure $\mu$ on $Z$ with marginals equal to the uniform probability measures on $ \Sym(n,k;\chi,\vt)$ and $ \Sym(n,k;\chi,\vu)$ respectively. 
\end{lem}

\begin{proof}
Let $\rho \in \Sym(n)$ be a disjoint product of $k$-cycles. The {\bf type of $\rho$  with respect to $\chi$} is the vector $\vtr=(r_i)_{i=0}^k$ defined by: $r_i$ is $1/n$ times the number of $k$-cycles $\rho'$ contained in $\rho$ such that $|\supp(\rho') \cap \chi^{-1}(1)| = i$.

Let $\s \in \Sym(n,k;\chi,\vt)$. Then there exist disjoint $k$-cycles $\s'_1,\ldots, \s'_m$ in $\s$ such that if $\rho=\s'_1\cdots \s'_m$ and $\vtr=(r_i)_{i=0}^k$ is the type of $\rho$ then $r_i = \min(t_i,u_i)$. Note $m \ge n(1/k-\delta/2)$ by assumption on $\vt$ and $\vu$. Moreover, there exist $k$-cycles $\s'_{m+1},\cdots, \s'_{n/k}$ such that the collection $\s'_1,\ldots, \s'_{n/k}$ is pairwise disjoint and the type of $\s' = \s'_1\cdots \s'_{n/k}$ is $\vu$. Then $|\s \vartriangle \s'| = 2(n/k-m) \le \delta n$. So $(\s,\s') \in Z$ which proves $Z$ is non-empty. 

We claim that there is a constant $C_1>0$ such that for every $\s \in \Sym(n,k;\chi,\vt)$ the number of $\s' \in \Sym(n,k;\chi,\vu)$ with $(\s,\s') \in Z$ is $C_1$. Indeed the following algorithm constructs all such $\s'$ with no duplications:

{\bf Step 1}. Let $\s = \s_1\cdots \s_{n/k}$ be a representation of $\s$ as a disjoint product of $k$-cycles. Choose a vector $\vtr=(r_i)_{i=0}^k$ such that
\begin{enumerate}
\item there exists a subset $ S \subset [n/k]$ with cardinality $|S| \ge n(1/k- \delta/2)$ such that if  $\rho = \prod_{i \in S} \s_i$ then $\vtr$ is the type of $\rho$;
\item $r_i \le u_i$ for all $i$. 
\end{enumerate}

{\bf Step 2}. Choose a subset $S \subset [n/k]$ satisfying the condition in Step 1. 

{\bf Step 3}. Choose pairwise disjoint $k$-cycles $\s'_1,\ldots, \s'_{n/k -|S|}$ such that 
\begin{enumerate}
\item $\supp(\s_i) \cap \supp(\s'_j) = \emptyset$ ($\forall i\in S$) ($\forall j$);
\item $\s'_j$ is not contained in $\s$ ($\forall j$);
\item if $\s' = \prod_{i \in S} \s_i \prod_j \s'_j$ then $\s'$ has type $\vu$.
\end{enumerate}

The range of possible vectors $\vtr$ in Step 1 depends only on $k,n,\vt,\vu$. The number of choices in Steps 2 and 3 depends only on the choice of $\vtr$ in Step 1 and on $k,n,\vt,\vu$. This proves the claim.

Similarly,  there is a constant $C_2>0$ such that for every $\s' \in \Sym(n,k;\chi,\vu)$ the number of $\s \in \Sym(n,k;\chi,\vt)$ with $(\s,\s') \in Z$ is $C_2$. It follows that the uniform probability measure on $Z$ has marginals equal to the uniform probability measures on $ \Sym(n,k;\chi,\vt)$ and $ \Sym(n,k;\chi,\vu)$ respectively.

\end{proof}

\begin{cor}\label{cor:1-permutation}
Let $U_{\Sym(n,k;\chi)}$ denote the uniform probability measure on $\Sym(n,k;\chi)$ and let $\E_{\Sym(n,k;\chi) }$ be the associated expectation operator. There are constants $c,\l>0$ (depending only on $k$) such that for every $\d>0$, there exists $N_\d$ such that for all $n>N_\d$, for every  1-Lipschitz $f:\Sym(n,k;\chi) \to \R$,
\begin{eqnarray*}
U_{\Sym(n,k;\chi)}\left( \left| f - \E_{\Sym(n,k;\chi) } [f] \right| > \delta \right) \le c\exp(-\l \d^2n).
\end{eqnarray*}

Moreover $\d \mapsto N_\d$ is monotone decreasing. 
\end{cor}

\begin{proof}
The set $\Sym(n,k;\chi)$ is the disjoint union of $ \Sym(n,k;\chi,\vt)$ over $\vt \in [0,1]^{k+1}$. Let $\d>0$.  Lemmas \ref{lem:orbit-concentration}, \ref{lem:average} and \ref{lem:coupling} imply that for all sufficiently large $n$, this decomposition of $\Sym(n,k;\chi)$ satisfies the criterion in  Lemma \ref{lem:concentration-criterion} where we set $c=3$, $\eta = \exp(-\l_1 \d^2n)$ and $\l = \l_0 n$ where $\l_0,\l_1>0$ depend only on $k$. So for every $1$-Lipschitz function $f:\Sym(n,k;\chi) \to \R$, every $\eps>\d+2\eta$ and all sufficiently large $n$,
\begin{eqnarray*}
U_{\Sym(n,k;\chi)}\left( \left| f - \E_{\Sym(n,k;\chi) } [f] \right| > \eps\right) \le \exp(-\l_1 \d^2n) + c\exp\left( - \l_0n\left( \eps -\delta - 2\eta \right)^2\right). 
\end{eqnarray*}
In particular, there exist $N_\d$ such that if $n>N_\d$ the inequality above holds and $2\eta < \delta$. By choosing $N_\d$ larger if necessary, we require that  $\d \mapsto N_\d$ is monotone decreasing.

Set $\eps=3\delta$. Because $\eps-\d-2\eta \ge \d$
\begin{eqnarray*}
U_{\Sym(n,k;\chi)}\left( \left| f - \E_{\Sym(n,k;\chi) } [f]\right| > 3\delta \right) \le \exp(-\l_1 \d^2n) + c\exp\left( - \l_0n\d^2\right) \le  (1+c) \exp(-\l \d^2n) 
\end{eqnarray*}
where $\l=\min(\l_0,\l_1)$. The corollary is now finished by changing variables.

\end{proof}

\begin{proof}[Proof of Theorem \ref{thm:planted-concentration}]
The space of homomorphisms $\Hom_\chi(\G,\Sym(n))$ is the $d$-fold direct power of the spaces $\Sym(n,k;\chi)$. So the Theorem follows from Corollary \ref{cor:1-permutation} and the proof of Lemma \ref{lem:product-concentration}.

\end{proof}


\bibliography{biblio}
\bibliographystyle{unsrt}

\end{document}